\documentclass[11pt,a4paper]{article}
\usepackage[utf8]{inputenc}

\usepackage[pages=all, color=black, position={current page.south}, placement=bottom, scale=1, opacity=1, vshift=5mm]{background}
\SetBgContents{
}      %

\usepackage[margin=2.5cm]{geometry}
\usepackage{amsmath}
\usepackage{diagbox}
\usepackage{times}
\usepackage{bm}
\usepackage{algorithm}
\usepackage{algpseudocode}
\usepackage{tabularx}
\usepackage{array}
\usepackage{multirow}
\usepackage{booktabs}
\usepackage{upgreek}
\usepackage{comment}
\usepackage{float}
\usepackage{stmaryrd}
\usepackage{authblk}
\usepackage{amsthm}
\usepackage{pgf-umlcd}
\usepackage{arydshln}
\usepackage{amssymb}
\usepackage{todonotes}
\usepackage{subcaption}
\usepackage{pgfplots}
\usepackage{listings}
\usepackage{textcomp}
\usepackage{pgfmath}
\usepackage{pgfplotstable}
\usepackage{adjustbox} 
\usepackage{pgfplotstable}
\usepackage{filecontents}
\usepackage{mathtools}
\usepackage[section]{placeins}
\usepackage{forest}
\usepackage{graphics}

\usepackage{lineno}

\usepackage[normalem]{ulem}

\usepackage{xcolor}

\definecolor{lightgray}{rgb}{0.95,0.95,0.95} %
\definecolor{paramcolor}{rgb}{0.2,0.4,0.8}    %
\definecolor{keywordcolor}{rgb}{0.8,0.2,0.2}  %

\definecolor{lbcolor}{rgb}{0.93,0.93,0.93}
\newtheorem{theorem}{Theorem}
\newtheorem{remark}{Remark}
\newtheorem{proposition}{Proposition}
\newtheorem{lemma}{Lemma}

\usepackage{hyperref}
\usepackage[sort&compress,numbers,square]{natbib}
\pgfplotsset{compat=1.17}

\usepackage{graphicx}
\graphicspath{{fig/}}

\begin{document}

\title{Scalable augmented Lagrangian preconditioners for fictitious domain problems}
\date{}

\author[1]{Michele Benzi}
\affil[1]{Scuola Normale Superiore, Piazza dei Cavalieri, 7, 56126 Pisa, Italy.

        {\small{michele.benzi@sns.it, federica.mugnaioni@sns.it}}}

\author[2]{Marco~Feder}
\affil[2]{Department of Mathematics, University of Pisa, Largo B. Pontecorvo, 5, Pisa, 56127, Italy.
        {\small{marco.feder@dm.unipi.it, luca.heltai@unipi.it}}}

\author[2]{Luca~Heltai}
\author[1]{Federica Mugnaioni$^{\star,}$}

\maketitle

\begin{center}
    \itshape \small Dedicated to the memory of Howard Elman
\end{center}

\bigskip

\begin{abstract}
    We present preconditioning techniques to solve linear systems of equations with a block two-by-two
    and three-by-three structure arising  from finite element discretizations of the fictitious domain method with Lagrange multipliers. In
    particular, we propose two augmented Lagrangian-based preconditioners to accelerate the convergence of iterative solvers for such classes of linear systems.
    We consider two relevant examples to illustrate the performance of these preconditioners when used in conjunction with flexible GMRES: the Poisson and the Stokes fictitious domain problems.  A spectral analysis is established for both exact and inexact versions of the preconditioners. We show the effectiveness of the proposed approach and the robustness of our
    preconditioning strategy through extensive numerical tests in both two and three dimensions.
    \newline
    \newline
    \newline
    \textbf{Keywords:} Preconditioning; Iterative solvers; Fictitious domain method; Non-matching meshes; Finite element method.
    \newline
    \newline
\end{abstract}

\section{Introduction}\label{sec:introduction}
\noindent Over the years, several numerical approaches have been developed to address problems involving overlapping or non-matching computational meshes. Among these, we highlight a small selection of well-established methods: the Immersed Boundary Method~\cite{Peskin_2002} (IBM), the Finite Cell Method~\cite{FCM} (FCM), the Cut Finite Element Method~\cite{CutFEM} (CutFEM), and the Fictitious Domain approach~\cite{Glowinski,BoffiGastaldiDLM,Hou_Wang_Layton_2012} (FD).

\noindent These methods differ in how they treat cells and degrees of freedom near the interface. CutFEM and FCM involve adjustments to the degrees of freedom and integral routines, since the degrees of freedom on cut cells are duplicated across the interface, and bulk integrals are recomputed over the polygonal or polytopal shapes resulting from the intersections of the elements with the cutting interface. For moving interface problems, such as in fluid-structure interactions, this may require recomputing the bulk contribution at each time step. In contrast, FD methods treat all cells in the same way without duplicating degrees of freedom, although this comes at the price of lower global order of convergence, which has been shown to be only a local effect near the interface~\cite{Heltai2019}. Furthermore, when using Lagrange multipliers in fictitious domain methods, coupling operators are assembled to encode interactions between basis functions defined on non-matching and overlapping meshes~\cite{boffi2023comparison}. %
 
\noindent Nevertheless, integration issues are limited to the coupling term only, making these methods attractive for moving interface problems by avoiding recomputation of bulk integrals. Recent developments
have shown that it is also possible to compute these integrals without explicitly dealing with mesh intersections, and this does not deteriorate the convergence properties
of the method~\cite{boffi2024quadratureerrorestimatesnonmatching}.  Moreover, it is robust with respect to small cuts~\cite{boffi2025stabilityconditioningfictitiousdomain}, while CutFEM and FCM may require additional stabilization in such cases. Therefore, when the coupling between the bulk and the interface problems can be modeled through a Lagrange multiplier, the
fictitious domain method offers a viable alternative which is generally simpler to implement, and may be computationally more efficient, especially in the presence of moving domains.
Several aspects of these methods have been addressed in the literature. For example, stabilization of the method when used in conjunction with CutFEM has been presented in~\cite{Burman2010FictitiousDF}, the impact of different types of coupling operators on the convergence of the resulting method has been analyzed in~\cite{BOFFI2022115650,boffi2024quadratureerrorestimatesnonmatching}, while possible strategies for the implementation of these operators in parallel computing environments, where the non-overlapping meshes are generally distributed with different partitioning, have been presented in~\cite{KrauseZulian,feder2024matrixfreeimplementationnonnestedmultigrid}.

The major drawback of using the Lagrange multiplier based Fictitious Domain method remains the computational demands of solving the resulting large-scale problems, both in terms of time and memory. In this work, we answer this challenge by proposing efficient preconditioning techniques that significantly reduce computational costs while maintaining robustness and scalability.

We focus on two model problems stemming from the application of the fictitious domain approach to the Poisson and Stokes problem. After discretization, the resulting linear systems exhibit the following two-by-two and three-by-three block structures:
\begin{equation}\label{eqn:intro}
  \begin{bmatrix}
    \mathsf{A} & \mathsf{C^T} \\
    \mathsf{C} & 0
  \end{bmatrix}
  \begin{bmatrix}
    \mathsf{u} \\
    \mathsf{\uplambda}
  \end{bmatrix}=
  \begin{bmatrix}
    \mathsf{f} \\
    \mathsf{g}
  \end{bmatrix} \qquad \text{and} \qquad
  \begin{bmatrix}
    \mathsf{A} & \mathsf{B^T} & \mathsf{C^T} \\
    \mathsf{B} & 0            & 0            \\
    \mathsf{C} & 0            & 0
  \end{bmatrix}
  \begin{bmatrix}
    \mathsf{ u} \\
    \mathsf{ p} \\
    \mathsf{ \lambda}
  \end{bmatrix}=
  \begin{bmatrix}
    \mathsf{ f} \\
    \mathsf{ 0} \\
    \mathsf{ g}
  \end{bmatrix},
\end{equation}
where $\mathsf{A} \in \mathbb{R}^{n \times n}$ is symmetric positive definite (SPD), $\mathsf{B} \in \mathbb{R}^{m \times n}$  and $\mathsf{C} \in \mathbb{R}^{l \times n}$. \\Two-by-two block linear systems of this type are an example of saddle point problems which arise in many areas of computational science and engineering. In~\cite{Benzi2005}, a wide variety of technical and scientific applications leading to these saddle point problems is reviewed, including mixed finite element methods and linear and nonlinear optimization. Similarly, three-by-three block linear systems of this form arise, for instance, in finite element modelling of potential fluid flow problems~\cite{potflow} or elasticity problems~\cite{boffi2013mixed}.

{\color{black}The large-scale saddle point problems in~\eqref{eqn:intro} make the use of direct solvers prohibitive, as they do not scale well with the size of the problem, in terms of both computational cost and memory usage. This is particularly evident when solving problems that result from discretizing partial differential equations in three-dimensional space.} Moreover, as the mesh size approaches zero, the condition number of the system matrices increases, leading to a deterioration in the convergence rate of iterative methods. Designing a suitable preconditioner is crucial to reduce or even eliminate this dependency on the mesh size. Developing robust solvers is also essential for efficiently handling computations with a large number of time steps and fine spatial discretizations within a reasonable timeframe.

Regarding the two-by-two block linear system of saddle point type in~\eqref{eqn:intro}, a large selection of solution methods is reviewed in~\cite{Benzi2005}, which also provides a detailed survey on preconditioners. In computational fluid dynamics, block diagonal and block triangular preconditioners are particularly popular~\cite{ElmanSilvesterWathen}. The effectiveness of such preconditioners depends on the availability of good approximations for both the (1,1)-block $\mathsf{A}$ and the Schur complement $\mathsf{S} = \mathsf{C A^{-1}C^T}$. However, finding a suitable approximation for the dense matrix $\mathsf{S }$ is especially challenging in the context of interface non-fitted meshes; {\color{black}difficulties arise} because the matrix $\mathsf{C}$ in~\eqref{eqn:intro} does not consist of integrals defined on the \emph{same} computational mesh, but rather on two arbitrarily overlapping grids. Preconditioning strategies for unfitted methods are an active area of research, and have been explored in~\cite{DEPRENTER2019604,cutFEM_prec,de2017condition}, to mention a few. However, to the best of our knowledge, for the fictitious domain with Lagrange multipliers method, it remains unclear how to construct a good approximation for the algebraic Schur complement and to which matrix it is spectrally equivalent. For this reason, in this paper, we propose an effective augmented Lagrangian-based preconditioner that eliminates the need for a good approximation to the dense Schur complement matrix. On the other hand, as shown in~\cite{trace1d2d,haznics}, in the framework of fitted (or matching) meshes, the algebraic Schur complement $\mathsf{S}$ is tightly related to
the continuous properties of the trace operator, requiring the computation of the action of fractional operators. A technique to overcome the computational burden
associated with this matter is presented in~\cite{li2024reducedkrylovbasismethods}.

Recently, several studies have focused on iterative methods and robust preconditioners for solving three-by-three block systems~\cite{ItermDoubleSP,DoubleSP,StokesDarcy,bakrani2023preconditioningtechniquesclassdouble}. In~\cite{ItermDoubleSP}, various block diagonal and block triangular preconditioners for Krylov subspace methods are described and analyzed for double saddle point systems with the same structure as the one in~\eqref{eqn:intro}. Several alternatives have been proposed in the literature, particularly for cases where the (3,3)-block is non-zero~\cite{nullspace, Benzi2019, Benzi2018, liqcri2022, alter}. In addition, preconditioners for Fictitious Domain with Distributed
Lagrange Multipliers (FD-DLM) formulations have been recently explored in~\cite{boffi2023parallel, BOFFI2024406,alshehri2025multigridpreconditioningfddlmmethod}, while Uzawa iterative methods were presented in~\cite{WangDLMFD_prec,Berrone2017AnOA}.
Most of the available preconditioners, such as those proposed in~\cite{ItermDoubleSP}, require finding an approximation for both Schur complements $\mathsf{S_B=BA^{-1}B^T}$ and $\mathsf{S_C=CA^{-1}C^T}$. We note that in our formulation the matrices $\mathsf{A}$ and $\mathsf{B}$ in~\eqref{eqn:intro} are the same as those in the problem without an immersed boundary. Therefore, existing literature on how to approximate $\mathsf{S_B}$ can be directly used (for instance, in the Stokes scenario, $\mathsf{A}$ and $\mathsf{B}$ correspond exactly to the discrete Laplacian and the discrete divergence of the classical Stokes problem). The main challenge is therefore related to finding approximations for $\mathsf{S_C}$. To address this issue, we propose an extension of the augmented Lagrangian (AL)-based preconditioner proposed for the two-by-two block linear system, as an effective preconditioner for the three-by-three block case.

In both cases, we follow the main ideas of augmented Lagrangian preconditioners in the context of stable finite element discretizations of
the Oseen and Navier-Stokes equations~\cite{ALprec,modALprec}, and derive an equivalent AL formulation for the problems at hand. The global linear system
of equations is solved using a {\color{black}(flexible)} Krylov subspace method. In this work, we use Flexible GMRES (FGMRES)~\cite{FGMRES} since the preconditioner is a variable one due to the way the (1,1)-block is {\color{black}approximately inverted}.
As previously indicated, to illustrate the performance of the proposed preconditioners in the context of fictitious domain methodologies, we consider two relevant test problems: the Poisson fictitious domain problem, which results in the two-by-two block linear system in~\eqref{eqn:intro}, and the Stokes fictitious domain problem, where the imposition of a constraint on the velocity field on the immersed boundary yields a \emph{double} saddle
point problem as the one in~\eqref{eqn:intro}. For the first model problem, the two-by-two block linear system is non-singular, whereas the system matrix in the second case may be singular (due to the rank deficiency of $\mathsf{B}$) if Dirichlet boundary conditions are applied on the whole boundary, leaving the scalar pressure field defined up to a constant.

In contrast to classical preconditioners for Stokes problems, where only the mass matrix on the pressure space is exploited while building the augmented block~\cite{ALprec}, also the mass matrix on the space of the Lagrange multiplier is needed when considering a fictitious domain formulation.
We provide a spectral analysis of the proposed AL preconditioner for the Stokes problem, assuming both exact and inexact solves (using similar arguments as in~\cite{bakrani2023preconditioningtechniquesclassdouble}). Analogous considerations hold for the Poisson case. Through extensive numerical experiments, uniform convergence with respect to the discretization parameters is observed. Moreover, having in mind three-dimensional scenarios where the computational complexity of sparse direct solvers for inverting individual blocks of the preconditioner soon becomes prohibitive, we have investigated the performance of our solvers when only inner iterative solvers are employed, confirming the good convergence properties. We provide a distributed memory
implementation of the proposed methodology using the \textsc{deal.II} finite element library~\cite{dealIIdesign,dealII95}.

The work is structured as follows. In Section~\ref{sec:poisson_problem}, we introduce the relevant notation and setup through a
Poisson problem with an internal boundary, and derive an augmented Lagrangian-based
preconditioner for the model problem. In Section~\ref{sec:stokes_problem}, we extend the same technique
to the Stokes system and perform a spectral analysis of the proposed preconditioner in Section~\ref{sec:spectral_analysis}. In Section~\ref{sec:inexact}, we discuss in detail the eigenvalue distribution of the preconditioned matrix when an inexact version of the proposed AL-based preconditioner is used, while
Section~\ref{sec:numerical_experiments} provides several numerical experiments to validate the robustness of our methodology across different scenarios.
Finally, Section~\ref{sec:conclusions} summarizes our conclusions and points to further research directions.

\section{Poisson fictitious domain problem and notation}\label{sec:poisson_problem}

Let $\omega$ be a closed and bounded domain of $\mathbb{R}^d$, $d=2,3$, with Lipschitz continuous boundary $\Gamma \coloneqq \partial \omega$, and $\Omega \subset \mathbb{R}^d$ a Lipschitz domain such that $\omega \Subset \Omega$. We consider the Poisson model problem
\begin{equation}
  \begin{cases}\label{eqn:model_problem}
    -\Delta u  = f \>\> \quad \text{in} \>\> \Omega \setminus \Gamma, \\
    \quad  \> \> \>u = g  \>\> \quad \text{on} \>\> \Gamma,           \\
    \quad  \> \> \>u = 0  \>\> \quad \text{on} \>\> \partial \Omega,
  \end{cases}
\end{equation}for given data $f \in L^2(\Omega)$ and $g \in H^{\frac{1}{2}}(\Gamma)$. Throughout this work we refer to $\Omega$ as the \textit{background} domain, while we refer to $\omega$ as the \textit{immersed} domain, and $\Gamma$ as the \emph{immersed boundary}. {\color{black}Hence, $\Gamma$ is a subset of $\Omega$ of codimension one (a surface in dimension three or a curve in dimension two)}. The rationale behind this setting is that it allows solving problems in a complex and possibly time-dependent domain $\omega$, by embedding it in a simpler background domain $\Omega$ -- typically a box -- and imposing some constraints on the immersed boundary $\Gamma$. For the sake of simplicity, we consider the case in which the immersed domain is \textit{entirely} contained in the background domain, but more general configurations may be considered. A prototypical configuration is shown in Figure~\ref{fig:domains}.

Given a domain $D\subset\mathbb{R}^d$ and a real number $s\ge 0$, we denote by $\|\cdot\|_{s,D}$ the standard Sobolev norm of $H^s(D)$.  In particular, $\|\cdot\|_{0,D}$ stands for the $L^2$-norm stemming from the standard $L^2$-inner product  $(\cdot, \cdot)_{D}$ on $D$. Finally, with $\langle \cdot, \cdot \rangle_{\Gamma}$ we denote the standard duality pairing between $H^{\frac12}(\Gamma)$ and its dual. At the discrete level, it can be evaluated using the scalar product in $L^2(\Gamma)$, provided that $\Lambda_h \subset L^2(\Gamma)$. As ambient spaces, we consider $$V(\Omega)\coloneqq H_0^1(\Omega) = \{v \in H^1(\Omega)\colon v_{|\partial \Omega} =0\},$$and $$\Lambda (\Gamma)\coloneqq H^{-\frac{1}{2}}(\Gamma),$$ which guarantee the well-posedness of the continuous problem thanks to the fulfillment of the \emph{ellipticity on the kernel} and the \textit{inf-sup} condition~\cite{Brezzi}. \\
Problem~\eqref{eqn:model_problem} can be written as a constrained minimization problem by introducing the Lagrangian $\mathcal{L}: V(\Omega) \times \Lambda (\Gamma)  \rightarrow \mathbb{R}$ defined as \begin{equation}
  \mathcal{L}(v,\mu)\coloneqq\frac{1}{2}(\nabla v,\nabla v)_{\Omega} - (f,v )_{\Omega} + \langle \mu, v-g \rangle_{\Gamma}.
\end{equation}

\noindent
Looking for stationary points of $\mathcal{L}$ gives the following saddle point problem of finding a pair $(u,\lambda) \in V(\Omega) \times \Lambda (\Gamma)$ such that
\begin{eqnarray}
  \label{eqn:LM1}
  (\nabla u, \nabla v)_{\Omega} + \langle\lambda, v\rangle_{\Gamma} &=& (f,v)_{\Omega} \qquad \forall v \in V(\Omega), \\
  \label{eqn:LM2}
  \langle  \mu, u\rangle_{\Gamma} &=& \langle \mu,g\rangle_{\Gamma} \qquad \forall \mu \in \Lambda (\Gamma),
\end{eqnarray}
which in operator form reads as
$$\begin{bmatrix}\label{eqn:operator_form}
    A & C^T \\
    C & 0
  \end{bmatrix}
  \begin{bmatrix}
    u \\
    \lambda
  \end{bmatrix}=
  \begin{bmatrix}
    F \\
    G
  \end{bmatrix},$$
where $$A \colon V(\Omega) \rightarrow V(\Omega)'  \quad \langle Au,v\rangle_{V(\Omega)',V(\Omega)} =  (\nabla u, \nabla v)_{\Omega},$$ $$C \colon V(\Omega) \rightarrow \Lambda(\Gamma)' \quad \langle C  v,\lambda \rangle_{\Lambda(\Gamma)',V(\Omega)} = \langle \lambda ,v \rangle_{\Gamma},$$ $$F \colon V(\Omega) \rightarrow \mathbb{R} \quad F(v)=(f,v)_{\Omega},$$ $$G \colon \Lambda(\Gamma) \rightarrow \mathbb{R} \quad G(\mu)=\langle \mu,g \rangle_{\Gamma}.$$ As is clear from the previous paragraph,
we use normal font to denote linear operators, e.g. $A$, while matrices and vectors are denoted by the sans serif font, e.g., $\mathsf{A}$ and $\mathsf{x}$, respectively. Moreover, we will use a calligraphic font to denote block matrices associated with the saddle point system, e.g., $\mathcal{A}$. We will
use the letter $\lambda$ to denote both the Lagrange multiplier and a generic eigenvalue of a matrix. The different meaning will be clear from the context and no confusion should arise.

\begin{figure}[ht]
  \centering
  \includegraphics{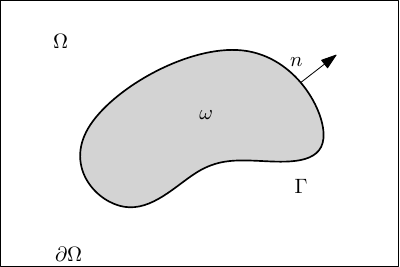}
  \caption{Model problem setting, with immersed domain $\omega$, immersed boundary $\Gamma$, and background domain $\Omega$.}
  \label{fig:domains}
\end{figure}

\subsection{Finite element discretization and mesh assumptions}
We discretize the computational meshes associated with the background domain $\Omega$ and with the immersed boundary $\Gamma$ in an \textit{unfitted} or \emph{non-matching} way, in the sense that the two computational grids are constructed independently, and no alignment conditions are asked. The computational meshes $\Omega_h$ and $\Gamma_h$ consist {\color{black}of a finite number $N_{\Omega}$ and $N_{\Gamma}$ of disjoint elements $T_i$ and $K_i$ such that $\{T_i\}_{1 \leq i \leq N_{\Omega}}$ and $\{K_i\}_{1 \leq i \leq N_{\Gamma}}$ form a partition of $\Omega$ and $\Gamma$, i.e. $\overline{\Omega} = \bigcup_{i=1}^{N_{\Omega}} T_i$, and $\overline{\Gamma} = \bigcup_{i=1}^{N_{\Gamma}} K_i$.} When $d=2$, $\Omega_h$ will be a quadrilateral mesh and $\Gamma_h$ a mesh composed by straight line segments. For $d=3$, $\Omega_h$ will be a hexahedral mesh and $\Gamma_h$ is a surface mesh whose elements are quadrilaterals embedded in the three-dimensional space. We denote by $h_\Omega$ and $h_\Gamma$ the mesh sizes of $\Omega_h$ and $\Gamma_h$, respectively. For simplicity, we
assume that the mesh sizes $h_\Omega$ and $h_\Gamma$ are small enough so that the geometrical error is negligible with respect to the discretization error. Furthermore, we assume $\Gamma_h$ to be a quasi-uniform discretization of $\Gamma$. {\color{black}In practice, we will consider meshes $\Gamma_h$ with a uniform mesh-size,  using background meshes $\Omega_h$ that may be locally refined in the vicinity of $\Gamma_h$ to capture the geometry of the immersed boundary.}

We consider the finite element discretization of $V(\Omega)$ based on linear standard Lagrange finite elements, namely
\begin{equation}\label{eqn:V_h}
  V_h \coloneqq \{v_h \in H_0^1(\Omega)\colon v_h{|_T} \in \mathcal{Q}_1(T), \forall T \in \Omega_h\},
\end{equation}where $\mathcal{Q}_1(T)$ denotes the space of polynomials of degree $1$ in each variable on the element $T$. For the Lagrange multiplier space $\Lambda (\Gamma)$ we set 
\begin{equation}
  \label{eqn:Lambda_h}\Lambda_h \coloneqq \{\mu_h \in L^2(\Gamma)\colon \mu_h{|_K} \in \mathcal{Q}_1(K), \forall K \in \Gamma_h\}.
\end{equation}
Given basis functions $\{ \varphi_i\}_{i=1}^n$ and $\{ \psi_\alpha\}_{\alpha=1}^l$ such that $V_h \coloneqq \text{span} \{ \varphi_i\}_{i=1}^n$ and $\Lambda_h \coloneqq \text{span} \{ \psi_\alpha\}_{\alpha=1}^l$, we have that the discrete version of~\eqref{eqn:LM1}, \eqref{eqn:LM2} can be written as the following algebraic problem with unknowns $\mathsf{u}$ and $\mathsf{\uplambda}$:

\begin{equation}
  \begin{bmatrix}\label{eqn:algebraic_form}
    \mathsf{A} & \mathsf{C^T} \\
    \mathsf{C} & 0
  \end{bmatrix}
  \begin{bmatrix}
    \mathsf{u} \\
    \mathsf{\uplambda}
  \end{bmatrix}=
  \begin{bmatrix}
    \mathsf{f} \\
    \mathsf{g}
  \end{bmatrix} \quad \text{or} \quad \mathcal{A} \mathsf{x} = \mathsf{{b}},
\end{equation}
where
\begin{eqnarray*}
  \mathsf{A}_{ij} &=& (\nabla \varphi_j, \nabla \varphi_i)_{\Omega}   \qquad i,j=1,\dots,n \\
  \mathsf{C}_{\alpha i} &=& \langle \varphi_i,\psi_\alpha \rangle_{\Gamma}  \qquad \quad \>\>\> i=1,\dots,n \quad \alpha = 1,\dots, l  \\
  \mathsf{f}_{i} &=& (f, \varphi_i)_{\Omega}   \qquad \qquad  \> i=1,\dots,n \\
  \mathsf{g}_{\alpha} &=& \langle  \psi_\alpha, g \rangle_{\Gamma} \qquad \quad \>\>\>\>\> \alpha = 1,\dots, l.
\end{eqnarray*}

Note that $\mathsf{A} \in \mathbb{R}^{n \times n}$ is symmetric positive definite and
$\mathsf{C} \in \mathbb{R}^{l \times n}$. Next, we recall the following theorem
that states the existence, uniqueness, and stability of the discrete solution together with
optimal error estimates. The interested reader is referred to~\cite[Sect. 5]{BoffiGastaldiDLM}, for
    the details. Notice that $\mathsf{C}$ must have full row rank in order to ensure the solvability of the algebraic problem~\eqref{eqn:algebraic_form}, which is guaranteed by the inf-sup condition~\eqref{inf-sup} below.

\begin{proposition}\label{prop:infsup}
  Let $V_h$ and $\Lambda_h$ be defined as above. If $h_{\Omega}/h_{\Gamma}$ is sufficiently
  small and the mesh $\Gamma_h$ is quasi-uniform, then there exists $\beta_2$ independent of the mesh sizes $h_\Omega$ and $h_\Gamma$ such
  that the following discrete inf-sup condition holds:
  \begin{equation}\label{inf-sup}
    \inf_{\mu_h \in \Lambda_h} \sup_{v_h \in V_h} \frac{\langle \mu_h, v_h \rangle_{\Gamma}}{\|\mu_h\|_{-\frac{1}{2},\Gamma} \> \|v_h\|_{1,\Omega}}\ge \beta_2.
  \end{equation}
\end{proposition}

\noindent Standard theory of saddle point problems~\cite{boffi2013mixed} yields the following a priori estimate.
\begin{theorem}\label{thm:err_estimate}
  Let $V_h$ and $\Lambda_h$ be defined as above. If $h_{\Omega}/h_{\Gamma}$ is sufficiently small and the mesh $\Gamma_h$ is
  quasi-uniform, then the unique solution $(u_h,\lambda_h)\in V_h\times \Lambda_h$ of the discretization of~\eqref{eqn:LM1},~\eqref{eqn:LM2},
  satisfies

  \begin{equation}
    \|u-u_h\|_{1,\Omega}+\|\lambda-\lambda_h\|_{-1/2,\Gamma}\le
    {\rm C} \inf_{\substack{v_h\in V_h\\ \mu_h\in \Lambda_h}}
    \left(\|u-v_h\|_{1,\Omega}+\|\lambda-\mu_h\|_{-1/2,\Gamma}\right),
    \label{eq:err_estimate}
  \end{equation}
  with ${\rm C}>0$ a constant independent of the mesh sizes $h_\Omega$ and $h_\Gamma$.
\end{theorem}

\begin{remark}  
  These results show that the accuracy of the FD method depends on the approximation properties of both the solution and the Lagrange multiplier spaces, provided the continuous and discrete inf-sup conditions hold. In particular, Theorem~\ref{thm:err_estimate} suggests that discretization spaces should be chosen so that their approximation properties are balanced both in terms of polynomial degrees, and in terms of the mesh ratio; otherwise, the global error is dominated by the least accurate variable, and the inf-sup condition may deteriorate. Indeed, the constant ${\rm C}$ in~\eqref{eq:err_estimate} is independent of the mesh sizes, but not of their ratio, as the discrete inf-sup condition of the coupling operator reflects this dependence. Hence, no advantage would be gained by choosing a very small or a very large ratio. The proposed preconditioner is designed to be effective under these conditions, ensuring robust convergence and a consistently low iteration count.
\end{remark}

\subsection{Augmented Lagrangian preconditioner}\label{sec:AL_preconditioner}

We present our augmented Lagrangian-based preconditioner for the model problem~\eqref{eqn:model_problem} discretized by
stable finite element pairs. Our derivation closely follows
early works on augmented Lagrangian preconditioning for the Oseen problem~\cite{ALprec,modALprec}.

\vspace{0.2cm}
The idea behind the classical AL approach is to replace the original system~\eqref{eqn:algebraic_form} with the equivalent formulation
\begin{equation}
  \begin{bmatrix}\label{eqn:matrix_aug_p}
    \mathsf{A+\gamma C^T W^{-1}C} & \mathsf{C^T} \\
    \mathsf{C}                    & 0
  \end{bmatrix}
  \begin{bmatrix}
    \mathsf{u} \\
    \mathsf{\uplambda}
  \end{bmatrix}=
  \begin{bmatrix}
    \mathsf{f +\gamma C^T W^{-1}g } \\
    \mathsf{g}
  \end{bmatrix} \quad \text{or} \quad \mathcal{A_\gamma} \mathsf{x} = \mathsf{\hat{b}},
\end{equation}
where $\mathsf{W}$ is a properly chosen SPD matrix and $\gamma$ is a positive real number. Having defined the augmented term as $\mathsf{{A}_{\gamma} \coloneqq A + \gamma C^T W^{-1}C}$, an ideal preconditioner for problem~\eqref{eqn:matrix_aug_p} is given by the block triangular matrix
\begin{equation}\label{eqn:preconditioner}
  \mathcal{P_{\gamma}} \coloneqq
  \begin{bmatrix}
    \mathsf{{{A}_{\gamma}}} & \mathsf{C^T}                \\
    0                       & -\frac{1}{\gamma}\mathsf{W}
  \end{bmatrix}.
\end{equation}
Since $\mathsf{A}$ is SPD and $\mathsf{C}$ has full row rank, the conditions of Lemma 4.1 in \cite{ALprec} are satisfied, thus $\gamma \mathsf{W^{-1}}$ provides a good approximation for the inverse of the Schur complement $\mathsf{S}_\gamma = \mathsf{C A_\gamma ^{-1}C^T }$ when $\gamma$ is large. However, as $\gamma$ increases, $\mathsf{A_\gamma}$ becomes increasingly ill conditioned, making it preferable to keep $\gamma$ at a moderate value. The choice of the user-defined parameter $\gamma$ is crucial for the performance of the preconditioner. Indeed, it should be selected based on a trade-off between the number of inner iterations required to approximate the action of $\mathsf{A_\gamma^{-1}}$ and the number of outer iterations needed to solve the entire system in~\eqref{eqn:matrix_aug_p}. Note that the potential benefit of $\mathcal{A}_\gamma$ being symmetric is lost when a nonsymmetric preconditioner, such as the one in~\eqref{eqn:preconditioner}, is used. However, if good approximations to $\mathsf{A_\gamma}$ and $\mathsf{S}_\gamma$ are available, using a method such as (F)GMRES with block triangular preconditioning will lead to rapid convergence,  making the overhead incurred from the use of a non-symmetric solver negligible \cite{Benzi2005}. In practice, the action of $\mathcal{P}_{\gamma}^{-1}$ is given by
$$\mathcal{P}_{\gamma}^{-1}
  =
  \begin{bmatrix}
    \mathsf{{{A}_{\gamma}}^{-1}} & 0            \\
    0                            & \mathsf{I}_l
  \end{bmatrix}
  \begin{bmatrix}
    \mathsf{I}_n & \mathsf{C^T}  \\
    0            & \mathsf{-I}_l
  \end{bmatrix}
  \begin{bmatrix}
    \mathsf{I}_n & 0                      \\
    0            & \gamma \mathsf{W^{-1}}
  \end{bmatrix},
$$
where $\mathsf{I}_n$ and $\mathsf{I}_l$ are identity matrices of size $n$ and $l$, respectively. The last identity implies that the application of the preconditioner to a vector requires one solve with $\mathsf{W}$, and one solve with the augmented term $\mathsf{{A_\gamma}}$. On top of that, we observe that this reformulation avoids the need to find a suitable approximation for the dense matrix $\mathsf{S = CA^{-1}C^T}$, a task that is particularly challenging in the context of interface non-fitted meshes. As already mentioned, the difficulty stems from the fact that the off-diagonal blocks in~\eqref{eqn:algebraic_form} involve integrals of basis functions that are not defined on the \emph{same} computational mesh, but rather on two grids that may overlap arbitrarily. %

The choice of the matrix $\mathsf{W}$ is crucial for the practical performance and applicability of the preconditioner. We propose to choose \begin{equation}\label{eqn:choice}\mathsf{W} \coloneqq \mathsf{M_{\lambda}^2},\end{equation} where
$\mathsf{M_{\lambda}}$ is the mass matrix on the immersed space $\Lambda_h \subset \Lambda(\Gamma)$, i.e.
\begin{equation}\label{eqn:mass_matrix_immersed}
  \bigl(\mathsf{M}_{\lambda}\bigr)_{\alpha,\beta} \coloneqq \int_{\Gamma} \psi_\alpha \psi_\beta, \qquad \alpha,\beta=1,\ldots,l.
\end{equation}
The reason for this choice will become clear in Section~\ref{sec:spectral_analysis}, where we discuss the Stokes case. The main idea is that the solutions of the generalized eigenvalue problem
\begin{equation}
  \begin{bmatrix}\label{eqn:matrix_aug}
    \mathsf{A_\gamma} & \mathsf{C^T} \\
    \mathsf{C}        & 0
  \end{bmatrix}
  \begin{bmatrix}
    \mathsf{x} \\
    \mathsf{y}
  \end{bmatrix}= \lambda
  \begin{bmatrix}
    \mathsf{{{A}_{\gamma}}} & \mathsf{C^T}                \\
    0                       & -\frac{1}{\gamma}\mathsf{W}
  \end{bmatrix} \begin{bmatrix}
    \mathsf{x} \\
    \mathsf{y}
  \end{bmatrix},
\end{equation}
are given by $\lambda=1$ and the associated  eigenvector $(\mathsf x, -\gamma \mathsf{W}^{-1} \mathsf{Cx})$, or $\lambda = \frac{\gamma\> \mathsf{x^T}   \mathsf{C^T} \mathsf{W}^{-1}\mathsf{C} \mathsf x }{\mathsf{x^T} (\mathsf{A} + \gamma \mathsf{C^T} \mathsf{W^{-1}} \mathsf{C} )\mathsf x}, $ for $\mathsf{x} \notin  \ker (\mathsf{C}) $. In Section~\ref{sec:spectral_analysis}, we will prove that choosing $\mathsf{W} \coloneqq \mathsf{M_\lambda^2}$ ensures that $\lambda$ remains bounded away from zero uniformly in $h_\Omega$ and $h_\Gamma$ (see Equation~\eqref{eq:secondcase}). This result holds for both the Poisson and the Stokes cases.\\
Due to the expensive solve associated especially with ${\mathsf{A}}_\gamma$, the preconditioner $\mathcal{P}_{\gamma}$ in~\eqref{eqn:preconditioner} is only \emph{ideal} and not of practical use. For this reason, it is necessary to replace the exact solves with approximate ones, resulting in a preconditioner which we denote as follows

\begin{equation}\label{eqn:preconditioner_approx}
  \mathcal{\widehat{P}}_{\gamma} \coloneqq
  \begin{bmatrix}
    \mathsf{{\widehat{A}_{\gamma}}} & \mathsf{C^T}                           \\
    0                               & -\frac{1}{\gamma} \mathsf{\widehat{W}}
  \end{bmatrix},
\end{equation}where $\mathsf{\widehat{A}_{\gamma}}$ and $\widehat{\mathsf{W}}$ are approximations of $\mathsf{A}_{\gamma}$ and $\mathsf{W}$, respectively, defined through the action of their inverses on vectors. The action of $\mathsf{A}_{\gamma}^{-1}$  is approximated through the conjugate-gradient (CG) solver preconditioned by a V-cycle of Algebraic MultiGrid (AMG). It is well known that, provided that a flexible GMRES solver is used to solve~\eqref{eqn:matrix_aug_p}, the inversion of the (1,1)-block {of the preconditioner} does not need to be done with high accuracy, so the tolerance for this block will be set to a rather loose value in the numerical experiments. Concerning the (2,2)-block, we observe that the number of degrees of freedom (DoF) associated with the immersed space $\Lambda_h$ is generally much lower than the number of DoF in the background space $V_h$. Indeed, while on a quasi-uniform mesh the number of background cells in $\Omega_h$ scales with $\mathcal{O}(h^{-d})$, the number of facets in $\Gamma_h$ scales with $\mathcal{O}(h^{-(d-1)})$. This would suggest that sparse direct solvers (such as UMFPACK~\cite{UMFPACK}) could still be employed for the inversion of $\mathsf{W} = \mathsf{M_\lambda^2}$, therefore it is not necessary to find an approximation $\widehat{\mathsf{ W}}$ also when the number of DoF in the background is large. When it becomes impractical, we choose to approximate $\mathsf{W}$ with a diagonal matrix whose entries are the squares of the main diagonal entries of $\mathsf{M_\lambda}$, i.e.
\begin{equation}\label{eqn:W_approx}
  \mathsf{\widehat{W}} \coloneqq \text{diag}(\mathsf{M_\lambda})^2.
\end{equation} Notably, if $\mathsf{\widehat{W}}$ is diagonal, the sparsity pattern of the augmented matrix $\mathsf{A_\gamma}$ remains close to the sparsity pattern of the original matrix $\mathsf{A}$. In particular, the product of terms $\mathsf{C^T}$ and $\mathsf{C}$ adds non-zero entries corresponding to DoF living on the
interface only. {\color{black}Note that} using a \emph{discontinuous} space $\Lambda_h$ for the Lagrange multiplier $\lambda$ automatically gives a diagonal $\mathsf{W}$.

\subsection{Spectrum of the preconditioned matrix}
We conduct numerical tests to evaluate the impact of the proposed preconditioner on the spectrum of $\mathcal{A_\gamma}$ ~\eqref{eqn:matrix_aug_p}.
We consider the Poisson problem~\eqref{eqn:model_problem} with the following configuration:
\begin{itemize}
  \item $\Omega = [0,1]^2$,
  \item $\omega = \mathcal{B}_{r}(\boldsymbol{c})$, where $\boldsymbol{c} = (0.5, 0.5)$ and $r = 0.21$.
\end{itemize}
We employ $\mathcal{Q}_1$ elements both for the background space $V_h$ and the immersed space $\Lambda_h$. We consider { four global refinements of the unit square, leading to $\mathsf{A} \in \mathbb{R}^{289\times289}$. The immersed mesh
    discretizing the interface consists of a uniform grid with $16$ facets, so that $\mathsf{C} \in \mathbb{R}^{17 \times 289}$, giving a global block matrix $\mathcal{A_\gamma} \in \mathbb{R}^{306\times306}$.} Above, $\mathcal{B}_{r}(\boldsymbol{c})$ denotes the ball of radius $r$ centered at $\boldsymbol{c}$.
We report in Figure~\ref{fig:laplace_spectrum} the spectrum of both $\mathcal{A_\gamma}$ and $\mathcal{P}_{\gamma}^{-1}\mathcal{A_\gamma}$ matrices for different values of the parameter $\gamma$ when using the \emph{ideal} variant with $\mathsf{W} =\mathsf{M_{\lambda}^2}$. {We point out that some of the eigenvalues of the original system are negative and close to zero, accordingly with the indefiniteness of the saddle point system.} The strong clustering of the whole spectrum near one when the parameter $\gamma$ increases is evident. This is consistent with the fact that $\gamma \mathsf{W^{-1}}$ increasingly provides a better approximation for the inverse of the Schur complement of $\mathcal{A_\gamma}$ as $\gamma$ becomes larger. Notably, all the eigenvalues { of the preconditioned system} are positive and real.

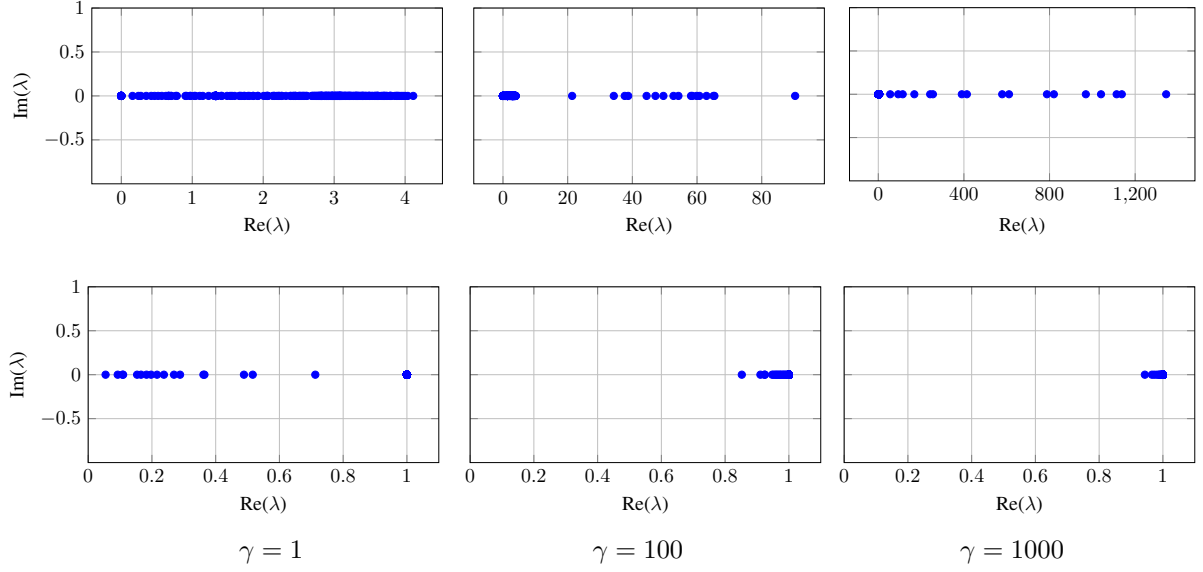
\begin{figure}[h]
  \centering
  \begin{subfigure}{0.3\textwidth} %
    \centering
    \begin{tikzpicture}[scale=0.68] %
      \begin{axis}[
          xlabel={Re($\lambda$)}, ylabel={Im($\lambda$)},
          width = 8.4cm,
          height = 5cm,
          grid=both,
        ]
        \addplot[only marks, mark=*, mark size=2pt, color=blue] table {data/laplace/gamma1.dat};
      \end{axis}
    \end{tikzpicture}
  \end{subfigure}\hspace{1.1cm} %
  \begin{subfigure}{0.3\textwidth}
    \centering
    \begin{tikzpicture}[scale=0.68]
      \begin{axis}[
          xlabel={Re($\lambda$)}, yticklabels={,,},
          width = 8.4cm,
          height = 5cm,
          grid=both,
        ]
        \addplot[only marks, mark=*, mark size=2pt, color=blue] table {data/laplace/gamma100.dat};
      \end{axis}
    \end{tikzpicture}
  \end{subfigure}\hspace{0.05cm} %
  \begin{subfigure}{0.3\textwidth}
    \centering
    \begin{tikzpicture}[scale=0.67]
      \begin{axis}[
          xlabel={Re($\lambda$)}, yticklabels={,,},
          width = 8.4cm,
          height = 5cm,
          grid=both,
          xtick distance=400,  %
        ]
        \addplot[only marks, mark=*, mark size=2pt, color=blue] table {data/laplace/gamma1000.dat};
      \end{axis}
    \end{tikzpicture}
  \end{subfigure}

  \vspace{0.4cm} %

  \begin{subfigure}{0.3\textwidth} %
    \centering
    \begin{tikzpicture}[scale=0.68] %
      \begin{axis}[
          xlabel={Re($\lambda$)}, ylabel={Im($\lambda$)},
          xmin=0, xmax=1.1,  %
          width = 8.4cm,
          height = 5cm,
          grid=both,
        ]
        \addplot[only marks, mark=*, mark size=2pt, color=blue] table {data/laplace/gamma1_prec.dat};
      \end{axis}
    \end{tikzpicture}
    \caption*{\hspace{2.2cm} $\gamma=1$}
  \end{subfigure}\hspace{1.1cm} %
  \begin{subfigure}{0.3\textwidth}
    \centering
    \begin{tikzpicture}[scale=0.68]
      \begin{axis}[
          xlabel={Re($\lambda$)}, yticklabels={,,},
          xmin=0, xmax=1.1,  %
          width = 8.4cm,
          height = 5cm,
          grid=both,
        ]
        \addplot[only marks, mark=*, mark size=2pt, color=blue] table {data/laplace/gamma100_prec.dat};
      \end{axis}
    \end{tikzpicture}
    \caption*{$\gamma=100$}
  \end{subfigure}\hspace{0.05cm} %
  \begin{subfigure}{0.3\textwidth}
    \centering
    \begin{tikzpicture}[scale=0.68]
      \begin{axis}[
          xlabel={Re($\lambda$)}, yticklabels={,,},
          xmin=0, xmax=1.1,  %
          width = 8.4cm,
          height = 5cm,
          grid=both,
        ]
        \addplot[only marks, mark=*, mark size=2pt, color=blue] table {data/laplace/gamma1000_prec.dat};
      \end{axis}
    \end{tikzpicture}
    \caption*{$\gamma=1000$}
  \end{subfigure}
  \vspace{0.1cm}
  \caption{Spectrum of the \emph{original} matrix $\mathcal{A_\gamma}$ (top row) and $\mathcal{P}_{\gamma}^{-1} \mathcal{A_\gamma}$ (bottom row) for increasing values of $\gamma$.}
  \label{fig:laplace_spectrum}
\end{figure}

\section{Stokes fictitious domain problem}\label{sec:stokes_problem}

In this Section, we extend the AL-based preconditioner developed in Section~\ref{sec:AL_preconditioner} to the Stokes problem. We use bold letters
to denote vector-valued functions, and we consider the following strong formulation:
\begin{equation}
  \begin{cases}\label{eqn:stokes_problem}
    \begin{aligned}
      -\Delta \mathbf{u} + \nabla p & = \mathbf{f}     &  & \text{in } \Omega\setminus \Gamma , \\
      \nabla \cdot \mathbf{u}       & = 0              &  & \text{in } \Omega\setminus \Gamma , \\
      \mathbf{u}                    & = \mathbf{g}     &  & \text{on } \Gamma,                  \\
      \mathbf{u}                    & = \boldsymbol{0} &  & \text{on } \partial \Omega,
    \end{aligned}
  \end{cases}
\end{equation}where $\mathbf{f} \in [L^2(\Omega)]^d$ and $\mathbf{g} \in [H^{\frac{1}{2}}(\Gamma)]^d$ are given data. The first two equations represent the classical Stokes problem, where $\mathbf{u}$ is a vector-valued function representing the velocity of the fluid, $p$ a scalar function representing its pressure, and $\mathbf{f}$ are external body forces. The third equation
imposes a constraint for the velocity field $\mathbf{u}$ on the immersed boundary $\Gamma$. \\ Concerning the incompressibility constraint, we observe
that the boundary datum $\mathbf{g}$ must satisfy the following compatibility condition
\begin{equation}\label{eqn:constraint_incompressibility}
  \int_{\Gamma} \mathbf{g}\cdot \boldsymbol{n}=0,
\end{equation}
where $\boldsymbol{n}$ is the outward unit normal to $\Gamma$.

We now give the functional spaces for the continuous problem~\eqref{eqn:stokes_problem}. For the velocity $\mathbf{u}$
we consider the usual choice $$V(\Omega)\coloneqq [H_0^1(\Omega)]^d =\{\mathbf{v} \in [H^1(\Omega)]^d\colon \mathbf{v}_{|\partial \Omega} = \mathbf{0}\}.$$ The pressure $p$ naturally belongs to $L^2(\Omega)$ and is determined up to a constant that we fix such that $p$ belongs to $$Q(\Omega)\coloneqq L_0^2(\Omega)= \Bigl\{ q \in L^2(\Omega)\colon \int_{\Omega} q =0 \Bigr\}.$$The functional space for the multiplier is chosen as $$\Lambda(\Gamma)\coloneqq [H^{-\frac{1}{2}}(\Gamma)]^d.$$

\noindent Problem~\eqref{eqn:stokes_problem} can be written as a constrained minimization problem by introducing the Lagrangian $\mathcal{L}: V(\Omega) \times Q(\Omega) \times \Lambda(\Gamma) \rightarrow \mathbb{R}$ defined as \begin{equation}
  \mathcal{L}(\mathbf{v},q,\boldsymbol{\mu})\coloneqq \frac{1}{2}\big(\nabla \mathbf{v}, \nabla \mathbf{v}\big)_{\Omega} -(\nabla \cdot \mathbf{v},q)_\Omega - (\mathbf{f},\mathbf{v} )_{\Omega} + \langle \boldsymbol{\mu}, \mathbf{v}-\mathbf{g} \rangle_{\Gamma}.
\end{equation}

\noindent Looking for stationary points of $\mathcal{L}$ gives the following \emph{double} saddle point problem of finding $(\mathbf{u}, p, \boldsymbol{\lambda}) \in V(\Omega) \times Q(\Omega) \times \Lambda(\Gamma)$ such that

\begin{equation}\label{eqn:LM_stokes_2}
  \begin{aligned}
    (\nabla \mathbf{v},  \nabla\mathbf{u})_{\Omega} - (\nabla \cdot \mathbf{v},p)_{\Omega} + \langle \boldsymbol{\lambda}, \mathbf{v}\rangle_{\Gamma} & = & (\mathbf{f},\mathbf{v})_{\Omega} \qquad \> \forall \mathbf{v} \in V(\Omega),                                \\
    (\nabla \cdot \mathbf{u},q)_{\Omega}                                                                                                              & = & 0 \qquad \qquad \>\> \> \forall q \in Q(\Omega),                                                            \\
    \langle  \boldsymbol{\mu}, \mathbf{u}\rangle_{\Gamma}                                                                                             & = & \langle \boldsymbol{\mu},\mathbf{g}\rangle_{\Gamma} \qquad \> \forall \boldsymbol{\mu} \in \Lambda(\Gamma).
  \end{aligned}
\end{equation}

{In order to verify the well-posedness of the double saddle point problem above, we need to fulfill two inf-sup conditions (see~\cite{boffi2013mixed}). The first inf-sup condition is the classical one for the Stokes problem.
\begin{proposition}\label{prop:inf-sup-stokes}
  There exists a constant $\beta_p > 0$ such that for all $q \in Q(\Omega)$
  \begin{equation}
    \label{eqn:inf_sup_div}
    \sup_{\mathbf{v} \in V(\Omega)} \frac{(\nabla \cdot \mathbf{v},q)_{\Omega}}{\|\mathbf{v}\|_{1,\Omega}} \geq \beta_p \|q\|_{0,\Omega}.
  \end{equation}
\end{proposition}
The second one is less standard and it is the inf-sup condition for the bilinear form $\langle  \boldsymbol{\mu}, \mathbf{v} \rangle_{\Gamma}$ associated with the Lagrange multiplier $\boldsymbol{\lambda}$. Its precise statement follows as a particular case of the more general setting presented in~\cite{BoffiGastaldiDLM} for the full fluid-structure interaction problem. Let us first set $$K_0 \coloneqq \{\mathbf{v} \in V(\Omega):  (\nabla \cdot \mathbf{v},q)_\Omega=0 \quad \forall q \in Q(\Omega) \} = \{\mathbf{v} \in V(\Omega):  \nabla \cdot \mathbf{v}=0 \text{ in } \Omega \},$$where
the second equality follows upon testing against $q= \nabla \cdot \mathbf{v} \in Q(\Omega)$, since $\mathbf{v} \in \bigl[H_0^1(\Omega)\bigr]^d$.

\begin{proposition}[\cite{BoffiGastaldiDLM}, Proposition 13]
  \label{prop:inf_sup_lambda}
  There exists a constant $\beta_\mu > 0$ such that for all $\boldsymbol{\mu} \in \Lambda(\Gamma)$ it holds
  \begin{equation}
    \sup_{\mathbf{v} \in K_0} \frac{\langle  \boldsymbol{\mu}, \mathbf{v} \rangle_{\Gamma}}{\|\mathbf{v}\|_{1,\Omega}} \geq \beta_\mu \|\boldsymbol{\mu}\|_{-\frac{1}{2},\Gamma}.
  \end{equation}
\end{proposition}
\begin{proof}Since $\Lambda(\Gamma) = [H^{-\frac{1}{2}}(\Gamma)]^d$, by definition of dual norm we have

  $$\|\boldsymbol{\mu}\|_{-\frac{1}{2},\Gamma} = \sup_{\mathbf{z} \in H^{1/2}(\Gamma)} \frac{\langle\boldsymbol{\mu}, \mathbf{z} \rangle_{\Gamma}}{\|\mathbf{z}\|_{\frac{1}{2},\Gamma}}.$$Let us now consider a maximizing sequence $\{\mathbf{z}_n\}_{n \in \mathbb{N}}$ such that

  \[\lim_{n \rightarrow +\infty} \frac{\langle\boldsymbol{\mu}, \mathbf{z}_n \rangle_{\Gamma}}{\|\mathbf{z}_n\|_{\frac{1}{2},\Gamma}} = \|\boldsymbol{\mu}\|_{-\frac{1}{2},\Gamma}.\]
  Thanks to the surjectivity of the trace operator from $[H_0^1(\Omega)]^d$ to $[H^{1/2}(\Gamma)]^d$, it is possible to show (cfr.~\cite[Lemma 12]{BoffiGastaldiDLM})
  the existence of $\mathbf{u}_n \in K_0$ such that $\mathbf{u}_{n|\Gamma}=\mathbf{z}_n$, with $\|\mathbf{u}_n\|_{1,\Omega} \leq c \|\mathbf{z}_n\|_{\frac{1}{2},\Gamma}$. The desired inequality is then obtained as follows:

  \[
    \sup_{\mathbf{v} \in  K_0} \frac{\langle \boldsymbol{\mu},\mathbf{v} \rangle_{\Gamma}}{\|\mathbf{v}\|_{1,\Omega}} \geq \frac{\langle \boldsymbol{\mu},\mathbf{u}_n \rangle_{\Gamma}}{\|\mathbf{u}_n\|_{1,\Omega}}
    \geq \frac{1}{c} \frac{\langle \boldsymbol{\mu},\mathbf{z}_n \rangle_{\Gamma}}{\|\mathbf{z}_n\|_{\frac{1}{2},\Gamma}} \geq \frac{1}{2c} \|\boldsymbol{\mu}\|_{-\frac{1}{2},\Gamma},
  \]where the last inequality follows from the fact that $\{\mathbf{z}_n\}_{n \in \mathbb{N}}$ is a maximizing sequence.

\end{proof}Having these two inf-sup conditions, we can state a \emph{combined inf-sup} condition involving $V(\Omega)$ and the graph space $M \coloneqq Q(\Omega)\times \Lambda(\Gamma)$ equipped with the natural norm $\|(q,\boldsymbol{\mu})\|^2_{M} \coloneqq \|q\|_{0,\Omega}^2 + \|\boldsymbol{\mu}\|_{-\frac{1}{2},\Gamma}^2$, which will be needed in Section~\ref{sec:spectral_analysis} for the spectral analysis of our preconditioner. This result is based on the abstract
framework developed in~\cite{CombinedInfSups} in the context of twofold saddle point problems, which we quote here in a simplified form for
ease of exposition.

\begin{theorem}[\cite{CombinedInfSups}, Theorem 3.1]
  \label{thm:combined_inf_sup_general}
  Let $U,P_1$ and $P_2$ be reflexive Banach spaces, and let $b_1 \colon U\times P_1 \rightarrow \mathbb{R}$, and $b_2\colon  U \times P_2 \rightarrow \mathbb{R}$ be bilinear and continuous forms. Let
  $$Z_{b_1} \coloneqq \{v \in U \colon b_1(v, q_1) = 0 \quad \forall q_1 \in P_1\} \subset U,$$
  then the following conditions are equivalent:
  \begin{enumerate}
    \item There exists $c > 0$ such that
          $$\sup_{v \in U} \frac{b_1(v,p_1) + b_2(v,p_2)}{\|v\|_U} \geq c (\|p_1\|_{P_1} + \|p_2\|_{P_2}) \qquad (p_1, p_2) \in P_1 \times P_2$$
    \item There exists $c > 0$ such that
          $$\sup_{v \in U} \frac{b_1(v,p_1)}{\|v\|_U} \geq c \|p_1\|_{P_1}, \quad p_1 \in P_1$$ and $$\sup_{v \in Z_{b_1}} \frac{b_2(v, p_2)}{\|v\|_U} \geq c \|p_2\|_{P_2}, \quad p_2 \in P_2.$$
  \end{enumerate}
\end{theorem}
We can therefore state the following combined inf-sup condition, which is a direct application of Proposition~\ref{prop:inf-sup-stokes} and Proposition~\ref{prop:inf_sup_lambda} above.

\begin{proposition}
  \label{thm:combined_inf_sup_stokes}
  There exists a constant $C > 0$ such that
  \begin{equation}\label{eqn:combined_inf_sup_stokes}
    \sup_{\mathbf{v} \in V(\Omega)} \frac{(\nabla \cdot \mathbf{v},q)_{\Omega} + \langle \boldsymbol{\mu},\mathbf{v} \rangle_{\Gamma}}{\|\mathbf{v}\|_{1,\Omega}} \geq C \|(q,\boldsymbol{\mu})\|_{M} \qquad (q,\boldsymbol{\mu}) \in M.
  \end{equation}
\end{proposition}
\begin{proof}
  We set $U=V(\Omega)$, $P_1=Q(\Omega)$, and $P_2=\Lambda(\Gamma)$ in the statement of Theorem~\ref{thm:combined_inf_sup_general}. As bilinear forms, it is sufficient to define $b_1(\mathbf{v}, q) \coloneqq (\nabla \cdot \mathbf{v},q)_{\Omega}$, and $b_2(\boldsymbol{\mu},\mathbf{v}) \coloneqq \langle \boldsymbol{\mu},\mathbf{v} \rangle_{\Gamma}$. Finally, we define $Z_{b_1}=K_0$.
  Then, Theorem~\ref{thm:combined_inf_sup_general} implies that there exists a constant $C > 0$ such that
  \[\sup_{\mathbf{v} \in V(\Omega)} \frac{(\nabla \cdot \mathbf{v},q)_{\Omega} + \langle \boldsymbol{\mu},\mathbf{v} \rangle_{\Gamma}}{\|\mathbf{v}\|_{1,\Omega}} \geq C \Bigl(\|q\|_{0,\Omega} + \|\boldsymbol{\mu}\|_{-\frac{1}{2},\Gamma} \Bigr) \qquad (q,\boldsymbol{\mu}) \in M.\]The thesis follows using the fact that $\|q\|_{0,\Omega} + \|\boldsymbol{\mu}\|_{-\frac{1}{2},\Gamma} \geq \sqrt{\|q\|_{0,\Omega}^2 + \|\boldsymbol{\mu}\|_{-\frac{1}{2},\Gamma}^2}  = \|(q,\boldsymbol{\mu})\|_{M}$.

\end{proof}

}

\subsection{Finite element discretization}\label{sec:discr_stokes_problem}We consider a finite element discretization based on mixed finite elements. In particular, we use the \emph{stable} Taylor-Hood $\mathcal{Q}_2$-$\mathcal{Q}_1$ pair, i.e. continuous piecewise quadratic velocities and linear pressures~\cite{boffi2013mixed}. We denote the associated finite-dimensional subspaces
with $V_h \subset V(\Omega)$ and $Q_h \not\subset Q(\Omega)$ (we do not ask functions of $Q_h$ to have zero average, therefore $Q_h \not\subset Q(\Omega)$). {Analogously to the Poisson problem,
    we employ as a finite-dimensional subspace for the multiplier
    \begin{equation}\label{eqn:Lambda_h_S}
      \Lambda_h \coloneqq \{\boldsymbol{\mu}_h \in [L^2(\Gamma)]^d\colon \boldsymbol{\mu}_h{|_K} \in [\mathcal{Q}_1(K)]^d, \forall K \in \Gamma_h\}.
    \end{equation}
  }

With these choices of finite element spaces, and denoting with $\{\boldsymbol{\varphi}_i\}_{i=1}^n$, $\{\phi_k\}_{k=1}^m$, and $\{\boldsymbol{\psi}_\alpha\}_{\alpha=1}^l$ the
selected basis functions of $V_h, Q_h,$ and $\Lambda_h$,  we end up with the following linear system of equations for $(\mathsf{u},\mathsf{p}, \mathsf{\uplambda})$

\begin{equation}
  \begin{bmatrix}\label{eqn:matrix}
    \mathsf{A} & \mathsf{B^T} & \mathsf{C^T} \\
    \mathsf{B} & 0            & 0            \\
    \mathsf{C} & 0            & 0
  \end{bmatrix}
  \begin{bmatrix}
    \mathsf{ u} \\
    \mathsf{ p} \\
    \mathsf{\uplambda}
  \end{bmatrix}=
  \begin{bmatrix}
    \mathsf{ f} \\
    \mathsf{ 0} \\
    \mathsf{ g}
  \end{bmatrix},
\end{equation}
where
\begin{eqnarray*}
  \mathsf{A}_{ij} &=& ( \nabla \boldsymbol{\varphi}_i, \nabla \boldsymbol{\varphi}_j)_{\Omega}   \qquad \>\> i,j=1,\dots,n \\
  \mathsf{B}_{ki} &=& -( \nabla \cdot \boldsymbol{\varphi}_i, \phi_k)_{\Omega}   \quad \quad i=1,\dots,n, \>\> k=1,\ldots,m \\
  \mathsf{C}_{\alpha i} &=& \langle  \boldsymbol{\varphi}_i, \boldsymbol{\psi}_\alpha \rangle_{\Gamma} \quad \qquad \quad i=1,\dots,n, \>\> \alpha = 1,\dots,l. \\
\end{eqnarray*}
As is customary in Immersed Boundary or Fictitious Domain methods, the velocity is fixed to $0$ on the background boundary, i.e. $\mathbf{u}_{|{\partial \Omega}}=\mathbf{0}$. Due to this particular choice, the pressure field $p$ is determined up to a constant. Choosing $Q$ to be in {$L^2_0(\Omega)$} fixes the uniqueness of the solution for the pressure in the continuous case, but not enforcing the zero average condition at the discrete level generates a matrix $\mathsf{B}$ which is not full row-rank and $0$ is an eigenvalue of the system. To summarize, $\mathsf{A} \in \mathbb{R}^{n \times n}$ is symmetric positive definite, $\mathsf{B} \in \mathbb{R}^{m \times n}$ is rank deficient by 1 and $\mathsf{C} \in \mathbb{R}^{l \times n}$. {As for the Poisson
    problem in Section~\ref{sec:poisson_problem}, the fact that $\mathsf{C}$ is full row-rank is a consequence of the fulfillment of the inf-sup condition for the velocity-multiplier pair.

    Sufficient conditions for the existence and uniqueness of the discrete problem are the discrete versions of the two inf-sup
    conditions for the continuous problem. Since the pair $V_h \times Q_h$ is stable, the discrete inf-sup condition for the Stokes problem is the
    standard discrete inf-sup condition for the divergence operator.
    \begin{proposition}\label{prop:discr_inf-sup-stokes}
      Let $V_h$ and $Q_h$ be defined as above, then there exists a positive constant $\beta_1$, independent of the mesh parameters, such that for all $q_h \in Q_h$
      it holds

      \begin{equation}\label{eqn:inf_sup_div_discrete}
        \sup_{\mathbf{v}_h \in V_h} \frac{(\nabla \cdot \mathbf{v}_h,q_h)_{\Omega}}{\|\mathbf{v}_h\|_{1,\Omega}} \geq \beta_1 \|q_h\|_{0,\Omega}.
      \end{equation}
    \end{proposition}
    The discrete analogue of Proposition~\eqref{prop:inf_sup_lambda}, showing the inf-sup condition involving the discrete Lagrange multiplier, is
    a particular case of ~\cite[Proposition 16]{BoffiGastaldiDLM}. We start by introducing the discrete version of the kernel $K_0$

    $$K_{0,h} \coloneqq \{\mathbf{v}_h \in V_h:  (\nabla \cdot \mathbf{v}_h,q_h)_{\Omega}=0  \quad \forall q_h \in Q_h \} \subset V_h.$$

    \begin{proposition}[\cite{BoffiGastaldiDLM}, Proposition 16]
      \label{prop:inf_sup_lambda_h}
      Let $V_h,Q_h$ and $\Lambda_h$ be defined as above. If $h_{\Omega}/h_{\Gamma}$ is sufficiently small and the mesh $\Gamma_h$ is
      quasi-uniform, then  there exists a constant $\beta_2 > 0$, independent of $h_{\Omega}$ and $h_{\Gamma}$, such that for all $\boldsymbol{\mu}_h \in \Lambda_h$ it holds

      \begin{equation}\label{eqn:inf_sup_lambda_h}
        \sup_{\mathbf{v}_h \in K_{0,h}} \frac{\langle \boldsymbol{\mu}_h, \mathbf{v}_h \rangle_{\Gamma}}{\|\mathbf{v}_h\|_{1,\Omega}} \geq \beta_2 \|\boldsymbol{\mu}_h\|_{-\frac{1}{2},\Gamma}.
      \end{equation}
    \end{proposition}

    \begin{remark}
      By definition, $K_{0,h} \subset V_h$. This allows us to rewrite the inf-sup condition in~\eqref{eqn:inf_sup_lambda_h} on the whole $V_h$. More precisely, for all $\boldsymbol{\mu}_h \in \Lambda_h$ it holds
      \begin{equation}\label{eqn:inf_sup_lambda_h_V_h}
        \beta_2 \|\boldsymbol{\mu}_h\|_{-\frac{1}{2},\Gamma} \leq  \sup_{\mathbf{v}_h \in K_{0,h}} \frac{\langle \boldsymbol{\mu}_h, \mathbf{v}_h \rangle_{\Gamma}}{\|\mathbf{v}_h\|_{1,\Omega}} \leq \sup_{\mathbf{v}_h \in V_{h}} \frac{\langle \boldsymbol{\mu}_h, \mathbf{v}_h \rangle_{\Gamma}}{\|\mathbf{v}_h\|_{1,\Omega}}.
      \end{equation}
    \end{remark}

    Proceeding exactly as in the continuous case, we can exploit the two individual discrete inf-sup conditions~\eqref{eqn:inf_sup_div_discrete} and~\eqref{eqn:inf_sup_lambda_h}, use
    Theorem~\ref{thm:combined_inf_sup_general}, and obtain a combined inf-sup condition for the discrete case. We denote
    with $M_h \coloneqq Q_h \times \Lambda_h$ the discrete version of the graph space $M$.
    \begin{proposition}
      \label{prop:combined_inf_sup_stokes_discrete}
      Let $V_h,Q_h$ and $\Lambda_h$ be defined as above. If $h_{\Omega}/h_{\Gamma}$ is sufficiently small and the mesh $\Gamma_h$ is
      quasi-uniform, then there exists a constant $\beta_3 > 0$, independent of $h_{\Omega}$ and $h_{\Gamma}$, such that
      for all $(q_h,\boldsymbol{\mu}_h) \in M_h$ it holds

      \begin{equation}\label{eqn:combined_inf_sup_stokes_discrete}
        \sup_{\mathbf{v}_h \in V_h} \frac{(\nabla \cdot \mathbf{v}_h,q_h)_{\Omega} + \langle \boldsymbol{\mu}_h,\mathbf{v}_h \rangle_{\Gamma}}{\|\mathbf{v}_h\|_{1,\Omega}} \geq \beta_3 \|(q_h,\boldsymbol{\mu}_h)\|_{M}.
      \end{equation}

    \end{proposition}

  }

\begin{remark}[Finite element spaces]
  Concerning the particular choice of the finite-dimensional subspaces, we notice that other choices are possible, such as $\mathcal{Q}_2$-$\mathcal{P}_{1}$, which use discontinuous piecewise linear pressures~\cite{BoffiGastaldiQ2P1}. However, we stress that the forthcoming spectral analysis { is generic and} relies
  only on the choice of inf-sup stable elements.
\end{remark}

\subsection{Augmented Lagrangian preconditioner}\label{sec:algebraic_form_stokes}
Consider the $(n+m+l)\times (n+m+l)$ linear system of equations~\eqref{eqn:matrix}. To apply the same reasoning used for the Poisson case, we propose augmenting the (1,1)-block twice and reformulating the Stokes fictitious domain problem as the following equivalent system
\begin{equation}
  \begin{bmatrix}\label{eqn:algebraic_form_stokes}
    \mathsf{A}+\gamma \mathsf{B^T} \mathsf{Q^{-1}} \mathsf{B} +\delta \mathsf{C^T} \mathsf{W^{-1}} \mathsf{C} & \mathsf{B^T} & \mathsf{C^T} \\
    \mathsf{B}                                                                                                & 0            & 0            \\
    \mathsf{C}                                                                                                & 0            & 0
  \end{bmatrix}
  \begin{bmatrix}
    \mathsf{ u} \\
    \mathsf{ p} \\
    \mathsf{ \lambda}
  \end{bmatrix}=
  \begin{bmatrix}
    \mathsf{ f}+\delta \mathsf{C^T} \mathsf{W^{-1}} \mathsf{ g} \\
    \mathsf{ 0}                                                 \\
    \mathsf{ g}
  \end{bmatrix} \quad \text{or} \quad \mathcal{A_{\gamma\delta}}\> \mathsf{x}=\mathsf{\hat{b}}.
\end{equation} Both $\mathsf{Q}$ and $\mathsf{W}$ are arbitrary SPD matrices and $\gamma >0$, $\delta >0$ are two real parameters to be selected. Then, we propose the following preconditioner:
\begin{equation}\label{eqn:prec_AL_stokes}
  \mathcal{P_{ \gamma \delta}} =
  \begin{bmatrix}
    \mathsf{A_{\gamma\delta}} & \mathsf{B^T}                & \mathsf{C^T}                \\
    0                         & -\frac{1}{\gamma}\mathsf{Q} & 0                           \\
    0                         & 0                           & -\frac{1}{\delta}\mathsf{W}
  \end{bmatrix},
\end{equation}
which has inverse
\begin{equation}\label{eqn:factor}
  \mathcal{P}_{\gamma\delta}^{-1}
  =
  \begin{bmatrix}
    \mathsf{{A}_{\gamma\delta}^{-1}} & 0            & 0            \\
    0                                & \mathsf{I}_m & 0            \\
    0                                & 0            & \mathsf{I}_l
  \end{bmatrix}
  \begin{bmatrix}
    \mathsf{I}_n & \mathsf{B^T}  & \mathsf{C^T}  \\
    0            & -\mathsf{I}_m & 0             \\
    0            & 0             & -\mathsf{I}_l
  \end{bmatrix}
  \begin{bmatrix}
    \mathsf{I}_n & 0                      & 0                      \\
    0            & \gamma \mathsf{Q^{-1}} & 0                      \\
    0            & 0                      & \delta \mathsf{W^{-1}}
  \end{bmatrix},
\end{equation}with $\mathsf{A_{\gamma\delta}} \coloneqq \mathsf{A}+\gamma \mathsf{B^T} \mathsf{Q^{-1}} \mathsf{B} +\delta \mathsf{C^T} \mathsf{W^{-1}} \mathsf{C}$. It is well known~\cite{ElmanSilvesterWathen} that for LBB-stable discretizations of the Stokes problem, if the Schur complement {$\mathsf{S_B}=\mathsf{B A^{-1}B^T}$} is nonsingular, it is spectrally equivalent to the pressure mass matrix $\mathsf{M_p}$. Otherwise, spectral equivalence holds on {$\operatorname{Ran}(\mathsf{B})$}{, the range of $\mathsf{B}$.} We set $\mathsf{Q\coloneqq M_p}$. Motivated by its effectiveness in the scalar case investigated in Section~\eqref{sec:poisson_problem} for the Poisson problem, we set again $\mathsf{W \coloneqq M_{\lambda}^2}$, defined analogously to the scalar case in~\eqref{eqn:mass_matrix_immersed} as
\begin{equation}\label{eqn:mass_matrix_immersed_stokes}
  \bigl(\mathsf{M}_{\lambda}\bigr)_{\alpha,\beta} \coloneqq \int_{\Gamma} \boldsymbol{\psi}_\alpha \cdot \boldsymbol{\psi}_\beta \qquad \alpha,\beta=1,\ldots,l.
\end{equation}All in all, {\color{black} from Equation~\eqref{eqn:factor} it is evident that} the application of the proposed preconditioner $\mathcal{P_{ \gamma \delta}}$ to a vector requires one solve with $\mathsf{A_{\gamma\delta}}$, one with $\mathsf{Q}$, and one with $\mathsf{W}$, plus two sparse matrix-vector products. The considerations made for the scalar case in Section~\ref{sec:poisson_problem} also apply identically in this case. In particular, the overall complexity is shifted to the solution of the augmented velocity in the (1,1)-block.

{\color{black} Of course, one could think of avoiding
assembling $\mathsf{A_{\gamma\delta}}$ and instead use it as an operator, but this would prevent the usage of preconditioners such as AMG or ILU decompositions. \\ We first focus on the practical construction of $\mathsf{B^T} \mathsf{Q^{-1}} \mathsf{B}$, which appears in the augmentation of the (1,1)-block. Note that if $\mathsf{Q^{-1}}$ is a dense matrix, then the product $\mathsf{B^T} \mathsf{Q^{-1}} \mathsf{B}$ will also be dense. Furthermore, even when a diagonal approximation of $\mathsf{Q}$ is used, the product $\mathsf{B^TB}$ may contain significantly more non-zero entries than $\mathsf{A}$ itself.

A possible remedy to this issue is to use a \emph{Grad-Div} stabilization. As discussed
in~\cite{graddiv}, the \emph{Grad-Div} stabilization can be interpreted as the sum of { an algebraic term, given by $(\mathsf{B^T} \mathsf{Q^{-1}} \mathsf{B})_{ij}$ $\forall i,j \in \{1, ..., n\}$}, and a projection-type stabilization term; see also~\cite{Reusken}. This additional projection term vanishes asymptotically as the mesh is refined, making it reasonable to use the same AL approach to build the preconditioner, but without explicitly adding the augmentation to the (1,1)-block. Specifically, the contribution $\gamma \bigl(\nabla \cdot \boldsymbol{\varphi}_i,\nabla \cdot \boldsymbol{\varphi}_j \bigr)_{i,j}$ is directly added to the global matrix $\mathsf{A}$ during the assembly phase of the system, thus avoiding the issues related to the explicit computation of the product $\mathsf{B^T} \mathsf{Q^{-1}} \mathsf{B}$. Replacing this augmentation by  \emph{Grad-Div} stabilization implies that new entries are added only where the components of the velocity couple.\\ Figure~\ref{fig:sparsity_comparison} illustrates the sparsity pattern of matrix $\mathsf{A}$ when \emph{Grad-Div} stabilization is included (see Figure~\ref{fig:sparsity1}) and compares it to the sparsity pattern of the augmented matrix $\mathsf
  {A}+\mathsf{B^T} \mathsf{Q^{-1}} \mathsf{B}$ when $\mathsf{A}$ does not include the stabilization term, { and $\mathsf{Q}$ is replaced with its main diagonal} (see Figure~\ref{fig:sparsity2}). %

On the other hand, we explicitly construct the term $\mathsf{C^T W^{-1} C}$. When a diagonal approximation of $\mathsf{W}$ is used, as will be the case in Section~\ref{subsec:stokes}, the influence of $\mathsf{C^T C}$ on the sparsity pattern of the whole augmented block is minimal, especially compared to the influence of $\mathsf{B^T B}$. This is clearly shown in Figure~\ref{fig:sparsity3}: the increase in the total number of non-zero entries in the final augmented block, where $\mathsf{A}$ includes the \emph{Grad-Div} stabilization term, is much more favorable. In practice, we will experience the fill-in shown in Figure~\ref{fig:sparsity3}.}

Once the whole $\mathsf{A_{\gamma\delta}}$ term is assembled, it is solved with CG preconditioned by a single V-cycle of AMG, using a very
loose tolerance (such as $10^{-2}$).

\begin{figure}[h]
  \centering
  \begin{subfigure}{0.32\textwidth}
    \centering
    \includegraphics[width=\linewidth]{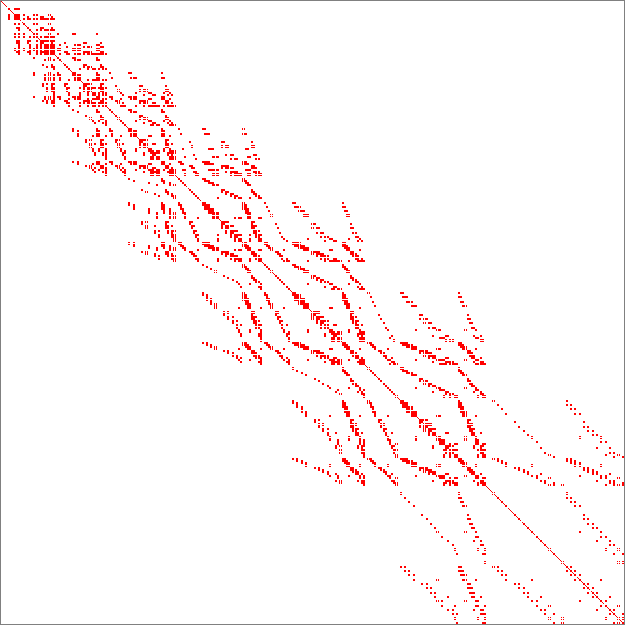}
    \caption{$\mathsf{A_{\text{GD}}}$ \\ \small NNZ: 12,228}
    \label{fig:sparsity1}
  \end{subfigure}
  \begin{subfigure}{0.32\textwidth}
    \centering
    \includegraphics[width=\linewidth]{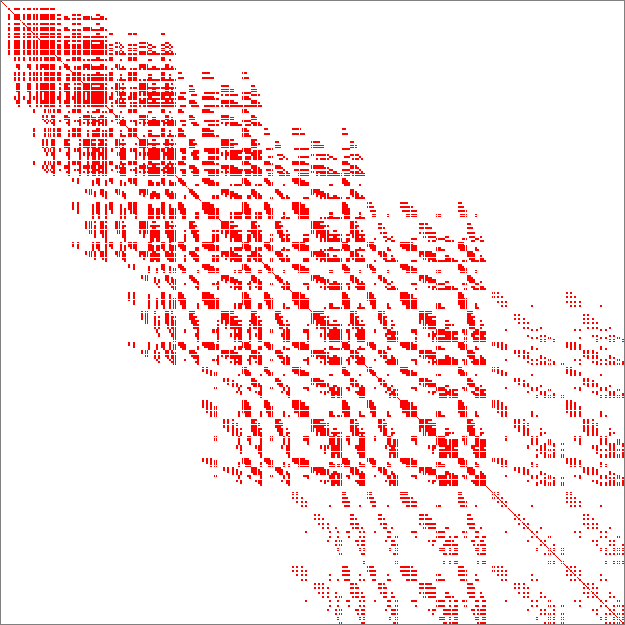}
    \caption{$\mathsf{A + B^T B}$ \\ \small NNZ: 40,168}
    \label{fig:sparsity2}
  \end{subfigure}
  \begin{subfigure}{0.32\textwidth}
    \centering
    \includegraphics[width=\linewidth]{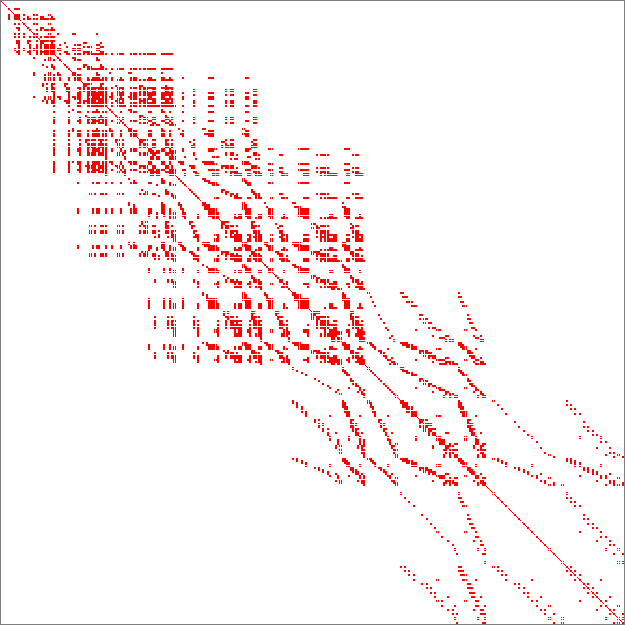}
    \caption{$\mathsf{A_{\text{GD}}+ C^T C}$ \\ \small NNZ: 16,612}
    \label{fig:sparsity3}
  \end{subfigure}

  \caption{{\color{black} Comparison of resulting sparsity patterns after different augmentations. Here we considered $\mathsf{A} \in \mathbb{R}^{578 \times 578}$, $\mathsf{B} \in \mathbb{R}^{81 \times 578}$, and $\mathsf{C} \in \mathbb{R}^{34\times 578}$. The resulting matrices stem from the discretization of the Stokes fictitious domain problem, using the configuration in Section~\ref{subsec:spectrum}. The matrix $\mathsf{A_{\text{GD}}}$ denotes the matrix $\mathsf{A}$ augmented with the Grad-Div stabilization.}}
  \label{fig:sparsity_comparison}
\end{figure}

\begin{remark}[Symmetric preconditioning]
  \label{rmk:SPD_variant}
  The preconditioned system is non-symmetric since we are dealing with a block triangular preconditioner. Therefore, eigenvalue information alone is not sufficient to rigorously justify the mesh independence of the convergence of non-symmetric matrix iterations
  like GMRES~\cite{GMRES_eigs}. For this reason, we have tested a symmetric and positive definite {\color{black}(block diagonal)} variant of our preconditioner, {namely
      \begin{equation}\label{eqn:prec_diagona}
        \mathcal{P_{\text{sym}, \gamma \delta}} =
        \begin{bmatrix}
          \mathsf{A_{\gamma\delta}} & 0                          & 0                          \\
          0                         & \frac{1}{\gamma}\mathsf{Q} & 0                          \\
          0                         & 0                          & \frac{1}{\delta}\mathsf{W}
        \end{bmatrix}.
      \end{equation}}
  In fact, in this case, the preconditioned matrix is similar to a symmetric matrix, and the eigenvalues provide meaningful
  estimates of the convergence rate of a method like preconditioned MINRES. We note that while mesh independence is observed, the total number of outer iterations for the global system solved with MINRES preconditioned with the symmetric positive definite variant of $\mathcal{P_{\gamma \delta}}$
  is higher than that obtained with FGMRES preconditioned by $\mathcal{P_{\gamma \delta}}$. {In addition, note that the MINRES solver requires the preconditioner to be applied exactly (or iteratively but in an accurate way), meaning that the augmented block $\mathsf{A_{\gamma \delta}}$ needs to be solved with a stricter tolerance. This results in a high number of inner iterations. Hence,} the block triangular preconditioner~\eqref{eqn:prec_AL_stokes} outperforms the block diagonal one in terms of computational cost. Even if we lost the possible advantage of $\mathcal{A}_{\gamma\delta}$ being symmetric, using a method like FGMRES with block triangular preconditioning leads --  in practice -- to more rapid convergence.
\end{remark}

\section{Spectral analysis}\label{sec:spectral_analysis}

In this Section, we derive lower and upper bounds for the eigenvalues of the preconditioned matrix $\mathcal{P_{ \gamma \delta}}^{-1} \mathcal{A}_{\gamma \delta}$ for the Stokes problem.  Although
eigenvalue information alone is generally insufficient to predict the convergence behaviour of nonsymmetric Krylov subspace methods, practical experience suggests that convergence is often fast when the spectrum is real, positive, and confined within a moderately sized interval {\color{black}bounded away from 0}. In the following, we
denote with $\operatorname{Spec}(\mathsf{M})$ the spectrum of a generic square matrix $\mathsf{M}$.

The next result, concerning the spectrum of the preconditioned matrix, is purely algebraic {and does not depend on any physical interpretation of the matrices involved. In Section~\ref{sec:boundedness},
        we will refine this analysis for the particular case in which such matrices correspond to those arising from the Stokes
        fictitious domain formulation. In that case, it will be possible to prove that the eigenvalues of the preconditioned system
        are bounded away from zero uniformly in the discretization parameters.}

\begin{theorem}\label{thm:spectral_AL}
    Suppose that $\mathcal{A}_{\gamma \delta}$ and $\mathcal{P}_{\gamma \delta}$ are defined by the matrices in~\eqref{eqn:algebraic_form_stokes} and~\eqref{eqn:prec_AL_stokes}, respectively. The non-zero eigenvalues of $\mathcal{P}^{-1}_{ \gamma \delta} \mathcal{A}_{\gamma \delta}$ are all real and positive. More precisely, {let $(\mathsf{x; y; z})$ be an eigenvector of $\mathcal{P}_{\gamma \delta}^{-1} \mathcal{A}_{\gamma \delta}$, it holds}
    $$\operatorname{Spec}(\mathcal{P}_{\gamma \delta}^{-1} \mathcal{A}_{\gamma \delta}) \subseteq \left [ \min (\eta, \epsilon, \theta), 1 \right ], $$
    where
    $$\eta = \min \left \{\frac{\gamma \mathsf{x^T}  \mathsf{B^T} \mathsf{Q}^{-1}\mathsf{B} \mathsf x }{\mathsf{x^T} (\mathsf{A}  + \gamma \mathsf{B^T} \mathsf{Q^{-1}} \mathsf{B} )\mathsf x } \;\biggm|\; \mathsf x \in  \ker(\mathsf{C}) \setminus  \ker(\mathsf{B})\right \},$$
    \vspace{0.2cm}
    $$\epsilon = \min \left \{\frac{\delta \mathsf{x^T}  \mathsf{C^T} \mathsf{W}^{-1}\mathsf{C} \mathsf x }{\mathsf{x^T} (\mathsf{A}  + \delta \mathsf{C^T} \mathsf{W^{-1}} \mathsf{C} )\mathsf x } \;\biggm|\; \mathsf x \in  \ker(\mathsf{B}) \setminus  \ker(\mathsf{C})\right \},$$
    \vspace{0.2cm}
    $$\theta = \min \left \{\frac{\mathsf{x^T} (\gamma  \mathsf{B^T} \mathsf{Q}^{-1}\mathsf{B}  + \delta  \mathsf{C^T} \mathsf{W}^{-1}\mathsf{C}) \mathsf x}{\mathsf{x^T}(\mathsf{A}+\gamma \mathsf{B^T} \mathsf{Q^{-1}} \mathsf{B} +\delta \mathsf{C^T} \mathsf{W^{-1}} \mathsf{C})\mathsf x} \;\biggm|\; \mathsf x \not\in  \ker (\mathsf{B}) \cup  \ker (\mathsf{C}) \right \},$$
    \vspace{0.2cm}
    \noindent
    and with $\lambda = 1$ being an eigenvalue of algebraic multiplicity at least $n$.
\end{theorem}

\begin{proof}
    Let $\lambda$ be an arbitrary eigenvalue of $\mathcal{P}^{-1}_{ \gamma \delta} \mathcal{A}_{\gamma \delta}$ with a corresponding eigenvector $(\mathsf x; \mathsf y; \mathsf z)$.
    The generalized eigenvalue problem can be written as
    \begin{equation}
        \begin{bmatrix}\label{eqn:eigenproblem}
            \mathsf{A_{\gamma\delta}} & \mathsf{B^T} & \mathsf{C^T} \\
            \mathsf{B}                & 0            & 0            \\
            \mathsf{C}                & 0            & 0
        \end{bmatrix}
        \begin{bmatrix}
            {\mathsf x} \\
            {\mathsf y} \\
            {\mathsf z}
        \end{bmatrix}= \lambda
        \begin{bmatrix}
            \mathsf{A_{\gamma\delta}} & \mathsf{B^T}                 & \mathsf{C^T}                 \\
            0                         & -\frac{1}{\gamma} \mathsf{Q} & 0                            \\
            0                         & 0                            & -\frac{1}{\delta} \mathsf{W}
        \end{bmatrix}
        \begin{bmatrix}
            {\mathsf x} \\
            {\mathsf y} \\
            {\mathsf z}
        \end{bmatrix},
    \end{equation}
    which can be written explicitly as
    \begin{align}
        \mathsf{A_{\gamma\delta}} \mathsf x + \mathsf{B^T} \mathsf y + \mathsf{C^T} \mathsf z & = \lambda (\mathsf{A_{\gamma\delta}} \mathsf x + \mathsf{B^T} \mathsf y + \mathsf{C^T} \mathsf z), \label{eqn:barAx} \\
        \mathsf{B} \mathsf x                                                                  & = -\frac{\lambda}{\gamma} \> \mathsf{Q} \mathsf y,\label{eqn:Bx}                                                     \\
        \mathsf{C} \mathsf x                                                                  & = -\frac{\lambda}{\delta } \> \mathsf{W} \mathsf z.  \label{eqn:Cx}
    \end{align}

    Note that the zero eigenvalue of the matrix in~\eqref{eqn:algebraic_form_stokes} related to the rank deficiency of $\mathsf{B}$ does not affect the convergence of preconditioned GMRES and can be excluded from the discussion~\cite{modALprec,ElmanSilvesterWathen}.\\
    Notice that $\mathsf x \ne \mathsf 0$; otherwise, in view of the fact that $\mathsf{Q}$ and $\mathsf{W}$ are SPD matrices, $\mathsf x=\mathsf 0$ implies $(\mathsf x; \mathsf y; \mathsf z) = (\mathsf 0; \mathsf 0; \mathsf 0)$, in contradiction with the fact that $(\mathsf x; \mathsf y; \mathsf z)$ is an eigenvector. {\color{black}Hence, from now on we assume that $\mathsf{x} \ne 0$.}\\
    It is clear from equations~\eqref{eqn:barAx}, \eqref{eqn:Bx}, and~\eqref{eqn:Cx} that $\lambda = 1$ is an eigenvalue of $\mathcal{P}^{-1}_{ \gamma \delta} \mathcal{A}_{\gamma \delta}$ with the corresponding eigenvector $(\mathsf x; -\gamma \mathsf{Q}^{-1} \mathsf{B} \mathsf x; -\delta \mathsf{W}^{-1} \mathsf{C}\mathsf x)$ when $\mathsf x \not\in  \ker(\mathsf{B}) \cup  \ker(\mathsf{C})$. Obviously, $\lambda = 1$ is an eigenvalue associated also with the eigenvector $(\mathsf x; 0; -\delta \mathsf{W}^{-1} \mathsf{C}\mathsf x)$ when $\mathsf x \in  \ker(\mathsf{B}) \setminus  \ker(\mathsf{C})$, with the eigenvector $(\mathsf x; -\gamma \mathsf{Q}^{-1} \mathsf{B} \mathsf x; \mathsf 0)$ when $\mathsf x \in  \ker(\mathsf{C}) \setminus  \ker(\mathsf{B})$ and with the eigenvector $(\mathsf x; \mathsf 0; \mathsf 0)$ when $\mathsf x \in  \ker(\mathsf{B}) \cap  \ker(\mathsf{C})$.
    The algebraic multiplicity of this eigenvalue is therefore at least $n$.\\ From now on, we assume that $\lambda \ne 1$ (and $\lambda \ne 0$). From~\eqref{eqn:barAx} we derive
    \begin{equation} \label{eqn:only_x}
        \mathsf{A_{\gamma\delta}} \mathsf x + \mathsf{B^T} \mathsf y + \mathsf{C^T} \mathsf z = 0,
    \end{equation}
    whereas from~\eqref{eqn:Bx} and~\eqref{eqn:Cx}, we respectively obtain
    $$\mathsf y = -\frac{\gamma}{\lambda} \mathsf{Q}^{-1} \mathsf{B} \mathsf x \quad \text{and} \quad \mathsf z = -\frac{\delta}{\lambda} \mathsf{W}^{-1} \mathsf{C} \mathsf x.$$
    By substituting the preceding two relations in~\eqref{eqn:only_x}, it follows that
    \begin{equation}\label{eqn:11}
        \mathsf{A_{\gamma\delta}}\> \mathsf x - \frac{\gamma}{\lambda} \mathsf{B^T} \mathsf{Q}^{-1}\mathsf{B}\> \mathsf x - \frac{\delta}{\lambda} \mathsf{C^T} \mathsf{W}^{-1}\mathsf{C}\> \mathsf x = \mathsf 0.
    \end{equation}
    Multiplying both sides of~\eqref{eqn:11} by $\lambda \mathsf x^\ast$, we get
    $$\lambda \mathsf x^\ast \mathsf{A_{\gamma\delta}} \mathsf x - \mathsf x^\ast (\gamma  \mathsf{B^T} \mathsf{Q}^{-1}\mathsf{B}  + \delta  \mathsf{C^T} \mathsf{W}^{-1}\mathsf{C}) \mathsf x = 0,  $$
    which is equivalent to
    \begin{equation} \label{eqn:lambda}
        \lambda =  \frac{\mathsf x^\ast (\gamma  \mathsf{B^T} \mathsf{Q}^{-1}\mathsf{B}  + \delta  \mathsf{C^T} \mathsf{W}^{-1}\mathsf{C}) \mathsf x}{\mathsf x^\ast(\mathsf{A}+\gamma \mathsf{B^T} \mathsf{Q^{-1}} \mathsf{B} +\delta \mathsf{C^T} \mathsf{W^{-1}} \mathsf{C})\mathsf x}.
    \end{equation}

    \vspace{0.2cm}
    \noindent
    Taking into account that $\gamma  \mathsf{B^T} \mathsf{Q}^{-1}\mathsf{B}  + \delta  \mathsf{C^T} \mathsf{W}^{-1}\mathsf{C}$ is the sum of two symmetric positive semidefinite matrices and $\mathsf{A}$ is SPD, we conclude that all the eigenvalues of $\mathcal{P}^{-1}_{ \gamma \delta} \mathcal{A}_{\gamma \delta}$ are real\footnote{Note that since $\lambda$ is real, the corresponding eigenvector can also be chosen to be real and therefore $\mathsf x^\ast$ can be replaced by $\mathsf{x^T}$.}. \\
    From~\eqref{eqn:lambda} we also deduce that $\lambda < 1$, and that all eigenvalues satisfying~\eqref{eqn:lambda} tend to $1$ for $\gamma$ and/or $\delta \rightarrow \infty$. Summarizing, we have the following cases:

    \begin{itemize}
        \item $\mathsf x \in  \ker(\mathsf{C}) \setminus  \ker(\mathsf{B})$,
              $$
                  0 <\eta \le \lambda = \frac{\gamma\mathsf x^\ast   \mathsf{B^T} \mathsf{Q}^{-1}\mathsf{B} \mathsf x }{\mathsf x^\ast (\mathsf{A}  + \gamma \mathsf{B^T} \mathsf{Q^{-1}} \mathsf{B} )\mathsf x} < 1,
              $$
              where
              $$\eta = \min \left \{\frac{\gamma \mathsf x^\ast  \mathsf{B^T} \mathsf{Q}^{-1}\mathsf{B} \mathsf x }{\mathsf x^\ast (\mathsf{A}  + \gamma \mathsf{B^T} \mathsf{Q^{-1}} \mathsf{B} )\mathsf x } \;\biggm|\; \mathsf x \in  \ker(\mathsf{C}) \setminus  \ker(\mathsf{B})\right \}.$$

              \vspace{0.3cm}
        \item $\mathsf x \in  \ker(\mathsf{B}) \setminus  \ker(\mathsf{C})$,
              $$
                  0<\epsilon \le \lambda = \frac{\delta\mathsf x^\ast   \mathsf{C^T} \mathsf{W}^{-1}\mathsf{C} \mathsf x }{\mathsf x^\ast (\mathsf{A}+\delta \mathsf{C^T} \mathsf{Q^{-1}} \mathsf{C} )\mathsf x} < 1,
              $$
              where
              $$\epsilon = \min \left \{\frac{\delta \mathsf x^\ast  \mathsf{C^T} \mathsf{W}^{-1}\mathsf{C} \mathsf x }{\mathsf x^\ast (\mathsf{A}  + \delta \mathsf{C^T} \mathsf{W^{-1}} \mathsf{C} )\mathsf x } \;\biggm|\; \mathsf x \in  \ker(\mathsf{B}) \setminus  \ker(\mathsf{C})\right \}.$$

              \vspace{0.3cm}
        \item $\mathsf x \not\in  \ker (\mathsf{B}) \cup  \ker (\mathsf{C})$,
              $$
                  0<\theta \le \lambda =  \frac{\mathsf x^\ast (\gamma  \mathsf{B^T} \mathsf{Q}^{-1}\mathsf{B}  + \delta  \mathsf{C^T} \mathsf{W}^{-1}\mathsf{C}) \mathsf x}{\mathsf x^\ast(\mathsf{A}+\gamma \mathsf{B^T} \mathsf{Q^{-1}} \mathsf{B} +\delta \mathsf{C^T} \mathsf{W^{-1}} \mathsf{C})\mathsf x}<1,
              $$
              where
              $$\theta = \min \left \{\frac{\mathsf x^\ast (\gamma  \mathsf{B^T} \mathsf{Q}^{-1}\mathsf{B}  + \delta  \mathsf{C^T} \mathsf{W}^{-1}\mathsf{C}) \mathsf x}{\mathsf x^\ast(\mathsf{A}+\gamma \mathsf{B^T} \mathsf{Q^{-1}} \mathsf{B} +\delta \mathsf{C^T} \mathsf{W^{-1}} \mathsf{C})\mathsf x} \;\biggm|\; \mathsf x \not\in  \ker (\mathsf{B}) \cup  \ker (\mathsf{C}) \right \}.$$
    \end{itemize}

    It is easy to check that if $\lambda \ne 1$, then $\mathsf x \in  \ker(\mathsf{B}) \cap  \ker(\mathsf{C})$ implies $\mathsf{A_{\gamma\delta}} \mathsf x = 0$. Since $\mathsf{A_{\gamma\delta}}$ is SPD, this means
    that $\mathsf x$ should be the zero vector, which is impossible.
\end{proof}

\subsection{Spectrum of preconditioned matrix}\label{subsec:spectrum}
We perform numerical tests to evaluate the theoretical findings in Theorem~\ref{thm:spectral_AL} and analyze the impact of the proposed preconditioner on
the spectrum of the system in~\eqref{eqn:algebraic_form_stokes}. We start with the following configuration for the Stokes problem:
\begin{itemize}
    \item $\Omega = [0,1]^2$,
    \item $\omega = \mathcal{B}_{r}(\boldsymbol{c})$, where $\boldsymbol{c} = (0.45,0.45)$ and $r = 0.21$.
\end{itemize}We consider a discretization {with $\mathcal{Q}_2$-$\mathcal{Q}_1$ elements for the velocity and pressure unknowns, while $\mathcal{Q}_1$ elements are used for the Lagrange multiplier. The background mesh consists of three uniform refinements of the
        unit square, while the immersed mesh consists of a uniform grid with $33$ facets.} For this discretization, we have $\mathsf{A} \in \mathbb{R}^{578 \times 578}$, the (negative) divergence matrix is $\mathsf{B} \in \mathbb{R}^{81 \times 578}$, and the coupling matrix is $\mathsf{C} \in \mathbb{R}^{34\times 578}${, resulting in a global matrix $\mathcal{A}_{\gamma\delta} \in \mathbb{R}^{677\times 677}$}. The choices for $\mathsf{Q}$ and $\mathsf{W}$ are
\begin{itemize}
    \item $\mathsf{Q} \coloneqq \mathsf{M_p}$ (pressure mass matrix),
    \item $\mathsf{W} \coloneqq \mathsf{M_{\lambda}^2}$ (immersed mass matrix {squared}).
\end{itemize}
We report in Figure~\ref{fig:stokes_spectrum} the spectrum of the unpreconditioned (top row) and preconditioned (bottom row) system matrix
for increasing (but identical) values of the AL parameters $\delta$ and $\gamma$ for the \emph{ideal} preconditioner. We observe that, except for
the zero eigenvalue, the rest of the eigenvalues are real and lie in the interval $(0, 1]$. Moreover, the {\color{black}non-zero eigenvalues} of the preconditioned system are (incrementally) shifted
towards $1$ as the parameters $\delta$ and $\gamma$ increase.

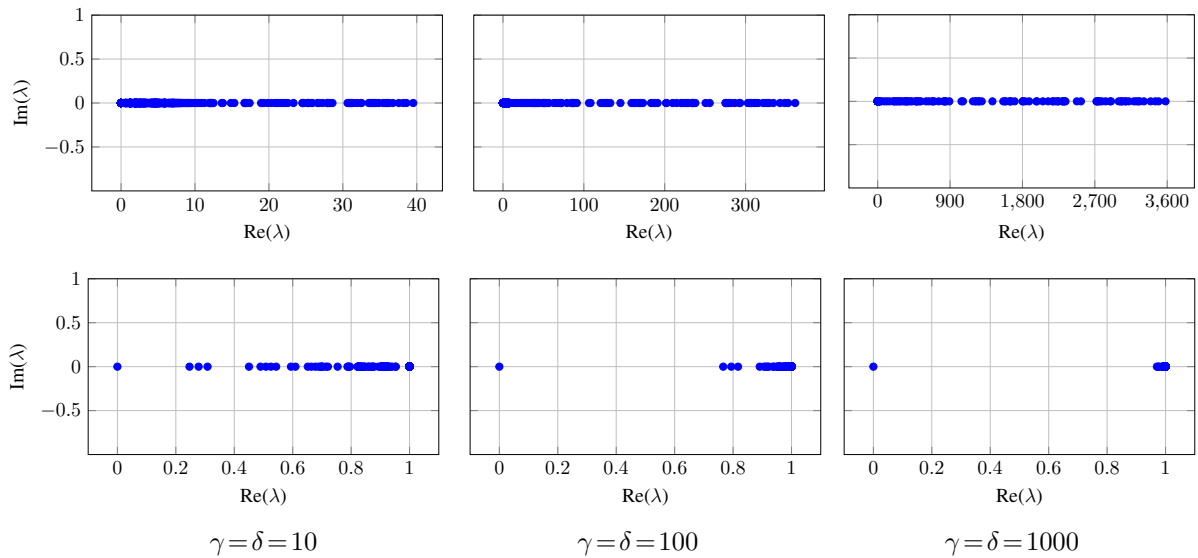
\begin{figure}[h]
    \centering
    \begin{subfigure}{0.3\textwidth} %
        \centering
        \begin{tikzpicture}[scale=0.68] %
            \begin{axis}[
                    xlabel={Re($\lambda$)}, ylabel={Im($\lambda$)},
                    width = 8.4cm,
                    height = 5cm,
                    grid=both,
                ]
                \addplot[only marks, mark=*, mark size=2pt, color=blue] table {data/stokes/gammadelta10.dat};
            \end{axis}
        \end{tikzpicture}
    \end{subfigure}\hspace{1.1cm} %
    \begin{subfigure}{0.3\textwidth}
        \centering
        \begin{tikzpicture}[scale=0.68]
            \begin{axis}[
                    xlabel={Re($\lambda$)}, yticklabels={,,},
                    grid=both,
                    width = 8.4cm,
                    height = 5cm,
                ]
                \addplot[only marks, mark=*, mark size=2pt, color=blue] table {data/stokes/gammadelta100.dat};
            \end{axis}
        \end{tikzpicture}
    \end{subfigure}\hspace{0.05cm} %
    \begin{subfigure}{0.3\textwidth}
        \centering
        \begin{tikzpicture}[scale=0.67]
            \begin{axis}[
                    xlabel={Re($\lambda$)}, yticklabels={,,},
                    width = 8.4cm,
                    height = 5cm,
                    grid=both,
                    xtick distance=900,  %
                ]
                \addplot[only marks, mark=*, mark size=2pt, color=blue] table {data/stokes/gammadelta1000.dat};
            \end{axis}
        \end{tikzpicture}
    \end{subfigure}

    \vspace{0.2cm} %

    \begin{subfigure}{0.3\textwidth} %
        \centering
        \begin{tikzpicture}[scale=0.68] %
            \begin{axis}[
                    xlabel={Re($\lambda$)}, ylabel={Im($\lambda$)},
                    xmin=-0.1, xmax=1.1,  %
                    width = 8.4cm,
                    height = 5cm,
                    grid=both,
                ]
                \addplot[only marks, mark=*, mark size=2pt, color=blue] table {data/stokes/gammadelta10_prec.dat};
            \end{axis}
        \end{tikzpicture}
        \caption*{\hspace{2.0cm} $\gamma\! =\! \delta \! =\! 10$}
    \end{subfigure}\hspace{1.1cm} %
    \begin{subfigure}{0.3\textwidth}
        \centering
        \begin{tikzpicture}[scale=0.68]
            \begin{axis}[
                    xlabel={Re($\lambda$)}, yticklabels={,,},
                    xmin=-0.1, xmax=1.1,  %
                    width = 8.4cm,
                    height = 5cm,
                    grid=both,
                ]
                \addplot[only marks, mark=*, mark size=2pt, color=blue] table {data/stokes/gammadelta100_prec.dat};
            \end{axis}
        \end{tikzpicture}
        \caption*{$\gamma\! =\! \delta \! =\! 100$}
    \end{subfigure}\hspace{0.05cm} %
    \begin{subfigure}{0.3\textwidth}
        \centering
        \begin{tikzpicture}[scale=0.68]
            \begin{axis}[
                    xlabel={Re($\lambda$)}, yticklabels={,,},
                    xmin=-0.1, xmax=1.1,  %
                    width = 8.4cm,
                    height = 5cm,
                    grid=both,
                ]
                \addplot[only marks, mark=*, mark size=2pt, color=blue] table {data/stokes/gammadelta1000_prec.dat};
            \end{axis}
        \end{tikzpicture}
        \caption*{$\gamma\! =\! \delta \! =\! 1000$}
    \end{subfigure}
    \vspace{0.1cm}
    \caption{Spectrum of the \emph{original} system matrix $\mathcal{A}_{\gamma\delta}$ (top row) and $\mathcal{P}^{-1}_{ \gamma \delta} \mathcal{A_{\gamma\delta}}$ (bottom row) for increasing values of $\gamma$ and $\delta$ applied to the Stokes test case.}
    \label{fig:stokes_spectrum}
\end{figure}

\subsection{Mesh-independence for lower bound}\label{sec:boundedness}
Concerning the lower bound {in Theorem~\ref{thm:spectral_AL}}, for suitable choices of the matrices $\mathsf{Q}$ and $\mathsf{W}$ it is possible to
prove that $\lambda$ is bounded away from zero uniformly in $h_\Omega$ and $h_\Gamma$, for all
fixed $\gamma$ and $\delta$. {We will first prove the result for $d=2$, i.e. when $\Gamma$ is a curve. The extension to the case $d=3$
        follows similarly and is given as a Remark at the end of this Section. Before delving into the proof, we collect some useful results we will need in the sequel. We start
        with the following inverse estimate, whose proof can be found in Appendix~\ref{sec:app:inverse_estimate}.
        \begin{lemma}{\label{thm:inverse_estimate}}
            Let $\Lambda_h$ be the space for the discrete Lagrange multiplier, defined in~\eqref{eqn:Lambda_h_S}. Then, there exists a constant $C>0$ independent of the mesh size $h_\Gamma$ such that the following estimate holds:
            \begin{equation}\label{eqn:inverse_estimate}
                ||\mu_h||_{-\frac{1}{2},\Gamma} \geq C h_{\Gamma}^{\frac{1}{2}} ||\mu_h||_{0,\Gamma} \qquad \forall \mu_h \in \Lambda_h.
            \end{equation}
        \end{lemma}Using Lemma~\eqref{thm:inverse_estimate} above, we get the following alternative characterization of the discrete inf-sup condition~\eqref{eqn:inf_sup_lambda_h_V_h}.

        \begin{proposition}\label{prop:infsup2}
            Let $V_h$ and $\Lambda_h$ be defined as in Section~\ref{sec:discr_stokes_problem}. If $h_{\Omega}/h_{\Gamma}$ is sufficiently small and the mesh $\Gamma_h$ is
            quasi-uniform, then there exists a positive constant $\tilde\beta_2$ independent of the mesh sizes $h_\Omega$ and $h_{\Gamma}$ such that the following rescaled inf-sup condition holds
            \begin{equation}\label{inf-sup2}
                \inf_{\mu_h \in \Lambda_h} \sup_{v_h \in V_h} \frac{\langle\mu_h, v_h \rangle_{\Gamma}}{||\mu_h||_{0,\Gamma} \> ||v_h||_{1,\Omega}}\ge h_{\Gamma}^{-\frac{1}{2}} \tilde\beta_2.
            \end{equation}
        \end{proposition}Alternative equivalence results
        between scaled norm and $H^{\frac{1}{2}}$ norms have been shown through localization techniques in~\cite{Faermann1} and recently extended to the $H^{-\frac{1}{2}}$ case by Bertoluzza in~\cite{bertoluzza2023localization}. \\
        In addition, we report the following well-known result concerning the general theory of the generalized Rayleigh
        quotient, which will be useful in what follows.

        \begin{remark}\label{thm:gen_theory}
            Let $\mathsf{M}$ and $\mathsf{N}$ be symmetric and symmetric positive definite matrices, respectively, with generalized eigenvalues $\lambda_1 \le \cdots \le \lambda_n$ and eigenvectors $\mathsf{v}_1, \ldots, \mathsf{v}_n \in \mathbb{R}^n$, such that
            $\mathsf{M} \mathsf{v}_i = \lambda_i \mathsf{N} \mathsf{v}_i.$
            Let $\mathsf{x} \in \mathbb{R}^n$, denote an arbitrary vector. Then:
            \begin{itemize}
                \item The smallest eigenvalue $\lambda_1$ can be characterized as
                      \[
                          \lambda_1 = \min_{\mathsf{x} \neq 0} \frac{\mathsf{x^T} \mathsf{M} \mathsf{x}}{\mathsf{x^T} \mathsf{N} \mathsf{x}},
                          \quad \text{achieved when } \mathsf{x} = \pm \mathsf{v}_1.
                      \]
                \item The second smallest eigenvalue $\lambda_{2}$ satisfies
                      \[
                          \lambda_{2} = \min_{\substack{\mathsf{x} \neq 0 \\ \mathsf{x^T} \mathsf{N} \mathsf{v}_1 = 0}}
                          \frac{\mathsf{x^T} \mathsf{M} \mathsf{x}}{\mathsf{x^T} \mathsf{N} \mathsf{x}},
                          \quad \text{achieved when } \mathsf{x} = \pm \mathsf{v}_{2},
                      \]
            \end{itemize}
            and so on.
        \end{remark}

        \noindent Finally, norms of finite element functions can be computed using the following definitions:
        \begin{equation}\label{eq:norm1}
            |v_h|_{1,\Omega} \coloneqq  (\mathsf{v^T} \mathsf{A}  \mathsf{v})  ^{1/2},
        \end{equation}
        \begin{equation}\label{eq:norm2}
            ||\mu_h||_{0,\Gamma} \coloneqq ( \mathsf{\upmu^T} \mathsf{M_\lambda}  \mathsf{\upmu} ) ^{1/2},
        \end{equation}
        \begin{equation}\label{eq:norm3}
            ||q_h||_{0,\Omega} \coloneqq ( \mathsf{q^T} \mathsf{M_p}  \mathsf{q}  ) ^{1/2},
        \end{equation}
        where $\mathsf v$, $\mathsf{q}$ and $\mathsf{\upmu}$ are vectors of the coefficients associated with
        the velocity, pressure and Lagrange multiplier basis functions, and $\mathsf{A}, \mathsf{M_p}, \mathsf{M_{\lambda}}$ are
        the usual stiffness and mass matrices.
        To prove mesh-independence for the lower bound of the preconditioned system, we will link the discrete inf-sup conditions to suitable generalized eigenvalue problems. The following three lemmas provide an algebraic
        interpretation of the discrete inf-sup stability conditions presented in Section~\ref{sec:stokes_problem}.

        \begin{lemma}\label{lemma:algebraic_infsup_B}
            Let $\mathsf{A}, \mathsf{B}$, and $\mathsf{M_p}$ be the matrices defined in Section~\ref{sec:stokes_problem}, and assume that the discrete inf-sup condition in Proposition~\ref{prop:discr_inf-sup-stokes} is satisfied. Then, the inf-sup constant $\beta_1 > 0$ is such that
            \begin{equation}\label{eq:13}
                \beta_1^2 = \min_{\{\mathsf{v} \in \mathbb{R}^n | \mathsf{u^T}\mathsf{Av} =0 \>\> \text{for}\> \mathsf{u} \in \ker(\mathsf{B})\}}  \frac{\mathsf{v^T} \mathsf{B^T} \mathsf{M_p^{-1}} \mathsf{B} \mathsf v}{\mathsf{v^T} \mathsf{A} \mathsf v}.
            \end{equation}
        \end{lemma}
        \begin{proof}
            A proof of this result can be found in~\cite{ElmanSilvesterWathen} (p. 193), or in~\cite{malkus}, where the classical Stokes problem is analyzed in detail. Notice that the matrices $\mathsf{A}$, $\mathsf{B}$ and $\mathsf{M_p}$ are exactly those of the classical Stokes problem defined in $\Omega$, without an immersed boundary.
        \end{proof}

        \begin{lemma}\label{lemma:algebraic_infsup_C}
            Let $\mathsf{A}, \mathsf{C}$, and $\mathsf{M_\lambda}$ be the matrices defined in Section~\ref{sec:stokes_problem}, and assume that the discrete inf-sup condition~\eqref{eqn:inf_sup_lambda_h_V_h} is satisfied. Then, there exists a positive constant $\bar \beta_2 $ independent of $h_\Gamma$, such that
            \begin{equation}
                \bar\beta_2^2 \le \min_{\{\mathsf{v} \in \mathbb{R}^n | \mathsf{u^T}\mathsf{Av} =0 \>\> \text{for}\> \mathsf{u} \in \ker(\mathsf{C})\}} \frac{\mathsf{v^T} \mathsf{C^T} \mathsf{M_\lambda^{-2}}\mathsf{C}  \mathsf v}{\mathsf{v^T}\mathsf{A} \mathsf v}.
            \end{equation}
        \end{lemma}
        \begin{proof}
            Using the matrix norms~\eqref{eq:norm1},~\eqref{eq:norm2} and given that $\langle \mu_h, v_h \rangle_{\Gamma} =  \mathsf{v^T}\mathsf{C^T}  \mathsf{\upmu} $, it is possible to give an algebraic interpretation of the discrete inf-sup stability condition in Proposition~\ref{prop:infsup2}. In particular we have:
            \begin{equation}\label{eq:inf-sup3}
                \inf_{\mathsf{\upmu} } \sup_{\mathsf{v}} \frac{\mathsf{v^T C^T} \mathsf{\upmu}}{(\mathsf{v^T A v})^{\frac{1}{2}}(\mathsf{\upmu^T} h_\Gamma \mathsf{M_\lambda} \mathsf{\upmu})^{\frac{1}{2}}}\ge  \tilde\beta_2.
            \end{equation}
            Arguing as in~\cite{malkus}, let $0<\sigma_1 \leq \sigma_2 \leq \ldots \leq \sigma_l$ be the $l$ largest generalized eigenvalues of
            \begin{equation}\label{eq:eigenLBB}
                \mathsf{C^T} h_\Gamma^{-1} \mathsf{M_\lambda^{-1}}\mathsf{C}\mathsf v = \sigma \mathsf{A} \mathsf v.
            \end{equation}
            Therefore, the constant appearing in the inf-sup condition~\eqref{inf-sup2} is given by the square root of $\sigma_1$, which is hence independent of the mesh size (this result can be proved following the same approach used in~\cite{malkus}, where the classical Stokes problem is considered). \\
            In other words, considering Remark~\ref{thm:gen_theory} and given the fact that $\mathsf{C^T} h_\Gamma^{-1} \mathsf{M_\lambda^{-1}}\mathsf{C}$ is a symmetric and positive semidefinite matrix, the following characterization holds:
            \begin{equation} \label{eq:min}
                \tilde\beta_2^2 = \sigma_1 = \min_{\{\mathsf{v} \in \mathbb{R}^n | \mathsf{u^T}\mathsf{Av}  =0 \>\> \text{for}\> \mathsf{u} \in \ker(\mathsf{C})\}} \frac{\mathsf{v^T} \mathsf{C^T} h_\Gamma^{-1} \mathsf{M_\lambda^{-1}}\mathsf{C}  \mathsf v}{\mathsf{v^T}\mathsf{A} \mathsf v}.
            \end{equation}
            Furthermore, for a one-dimensional immersed domain embedded in a two-dimensional background domain, it can be shown that $(h_\Gamma \mathsf{M_\lambda})^{-1}$ and $\mathsf{M_\lambda^{-2}}$ are spectrally equivalent, i.e.
            \begin{equation}\label{eq:bounds}
                0 < \frac{c}{C^2} \le \frac{\mathsf{\upmu^T} \mathsf{M_\lambda^{-2}} \mathsf{\upmu}}{\mathsf{\upmu^T} (h_\Gamma \mathsf{M_\lambda})^{-1} \mathsf{\upmu}} \le \frac{C}{c^2},
            \end{equation}for some positive constants $c, C$ independent of the discretization parameters (the details can be found in Appendix~\ref{sec:app:spectral_h_mass}). Setting $\mathsf{\upmu} = \mathsf{C} \mathsf v$, and multiplying each term of (\ref{eq:bounds}) by $\frac{\mathsf{\upmu^T} (h_\Gamma \mathsf{M_\lambda})^{-1} \mathsf{\upmu}}{\mathsf{v^T} \mathsf{A} \mathsf v}$, we get
            $$0 < \frac{c}{C^2} \> \frac{\mathsf{v^T} \mathsf{C^T}(h_\Gamma \mathsf{M_\lambda})^{-1} \mathsf{C}\mathsf v}{\mathsf{v^T} \mathsf{A} \mathsf v}\le \frac{\mathsf{v^T} \mathsf{C^T} \mathsf{M_\lambda^{-2}} \mathsf{C}\mathsf v,}{\mathsf{v^T} \mathsf{A} \mathsf v} \le \frac{C}{c^2}\> \frac{\mathsf{v^T}\mathsf{C^T}(h_\Gamma \mathsf{M_\lambda})^{-1} \mathsf{C} \mathsf v}{\mathsf{v^T} \mathsf{A} \mathsf v}.$$Combining
            this result with ~\eqref{eq:min}, it follows that
            \begin{equation}\label{eq:eiginfsup1}
                \bar\beta_2^2 \le \min_{\{\mathsf{v} \in \mathbb{R}^n | \mathsf{u^T}\mathsf{Av}  =0 \text{ for }  \mathsf{u} \in \ker(\mathsf{C})\}}\frac{\mathsf{v^T}  \mathsf{C^T} \mathsf{M_\lambda^{-2}}\mathsf{C}\mathsf v}{\mathsf{v^T}\mathsf{A} \mathsf v}
            \end{equation}
            (where $\bar\beta_2^2 \coloneqq \frac{{c}}{C^2}\sigma_1$), uniformly in $h_\Gamma$.
        \end{proof}
        \vspace{0.3cm}
        \noindent

        \vspace{0.3cm}
        \noindent
        \begin{lemma}\label{lemma:algebraic_infsup_BC}
            Let $\mathsf{A}, \mathsf{B}, \mathsf{C}, \mathsf{M_p}$ and $\mathsf{M_\lambda}$ be the matrices defined in Section~\ref{sec:stokes_problem}, and assume that the two inf-sup conditions~\eqref{eqn:inf_sup_div_discrete} and~\eqref{eqn:inf_sup_lambda_h_V_h}
            are satisfied. Then, there exists a positive constant $\bar\beta_3>0$ independent of both $h_\Omega$ and $h_\Gamma$, such that
            \begin{equation}\label{eq:13}
                \bar\beta_3^2 \le \min_{\{\mathsf{v} \in \mathbb{R}^n | \mathsf{u^T}\mathsf{Av}  =0 \text{ for }  \mathsf{u} \in \ker(\mathsf{C}) \cap \ker(\mathsf{B})\}} \frac{\mathsf{v^T} (\mathsf{B^T} \mathsf{M_p^{-1}} \mathsf{B} + \mathsf{C^T} \mathsf{M_\lambda^{-2}} \mathsf{C}) \mathsf v}{\mathsf{v^T} \mathsf{A} \mathsf v}.
            \end{equation}

        \end{lemma}
        \begin{proof}
            The discrete counterpart of the original double saddle point problem~\eqref{eqn:LM_stokes_2} can be equivalently rewritten as
            \begin{equation}\label{eqn:LM_stokes_3}
                \begin{aligned}
                    (\nabla \mathbf{v}_h,  \nabla\mathbf{u}_h)_{\Omega} - (\nabla \cdot \mathbf{v}_h,p_h)_{\Omega} + \langle \boldsymbol{\lambda}_h, \mathbf{v}_h\rangle_{\Gamma}
                     & = (\mathbf{f},\mathbf{v}_h)_{\Omega}
                     &                                                         & \quad \forall \mathbf{v}_h \in V_h,              \\[6pt]
                    -(\nabla \cdot \mathbf{u}_h,q_h)_{\Omega}+\langle  \boldsymbol{\mu}_h, \mathbf{u}_h\rangle_{\Gamma}
                     & = \langle \boldsymbol{\mu}_h,\mathbf{g}\rangle_{\Gamma}
                     &                                                         & \quad \forall (q_h, \boldsymbol{\mu}_h) \in M_h.
                \end{aligned}
            \end{equation}
            Problem~\eqref{eqn:LM_stokes_3} is associated with the following partition of the original saddle point matrix

            \begin{equation}
                \left[
                    \begin{array}{c:cc}
                        \mathsf{A} & \mathsf{B^T} & \mathsf{C^T} \\ \hdashline
                        \mathsf{B} & 0            & 0            \\
                        \mathsf{C} & 0            & 0
                    \end{array}
                    \right] =
                \begin{bmatrix}
                    \mathsf{A} & \mathsf{D^T} \\
                    \mathsf{D} & 0            \\
                \end{bmatrix},
            \end{equation}
            where $\mathsf{D} \coloneqq \begin{bmatrix}
                    \mathsf{B} \\
                    \mathsf{C}
                \end{bmatrix}$. Defining the bilinear form $d: V \times M \rightarrow \mathbb{R}$ as $d(\mathbf{v},(q,\boldsymbol{\mu})) = (\nabla \cdot \mathbf{v},q)_{\Omega} + \langle\boldsymbol{\mu},\mathbf{v}\rangle_{\Gamma}$,
            we have that at the discrete level $d(\cdot,\cdot)$ can be evaluated as
            $$d(\mathbf{v}_h,(q_h,\boldsymbol{\mu}_h)) = \mathsf{v^T} \mathsf{B^T} \mathsf{q}  + \mathsf{v^T} \mathsf{C^T}\mathsf{\upmu} .$$
            Using now the matrix norms~\eqref{eq:norm1},~\eqref{eq:norm2},~\eqref{eq:norm3}, and the definition of the norm $$||(\boldsymbol{\mu}_h, q_h)||^2_{M}= ||q_h||_Q^2 + ||\boldsymbol{\mu}_h||_\Lambda^2,$$ it is possible to give an algebraic interpretation of the discrete inf-sup stability condition in Proposition~\ref{prop:combined_inf_sup_stokes_discrete}, using the rescaled variant provided by Lemma~\eqref{thm:inverse_estimate}. In particular, we have:
            \begin{equation}\label{eq:inf-sup3}
                \inf_{(\mathsf{\upmu}, \mathsf{q}) } \sup_{\mathsf{v}} \frac{\mathsf{v^T B^T} \mathsf{q} + \mathsf{v^T C^T} \mathsf{\upmu}}{(\mathsf{v^T A v})^{\frac{1}{2}}(\mathsf{q^T M_p q} + \mathsf{\upmu^T} h_\Gamma \mathsf{M_\lambda} \mathsf{\upmu})^{\frac{1}{2}}}\ge  \tilde\beta_3.
            \end{equation}
            Arguing again as in~\cite{malkus}, let $0<\sigma_1 \leq \sigma_2 \leq \ldots$ be the nonzero generalized eigenvalues of
            \begin{equation}\label{eq:eigenLBB}
                (\mathsf{B^T M_p^{-1}B} + \mathsf{C^T} h_\Gamma^{-1} \mathsf{M_\lambda^{-1}}\mathsf{C})\mathsf v = \sigma \mathsf{A} \mathsf v.
            \end{equation}
            Then, the constant appearing in the inf-sup condition~\eqref{eq:inf-sup3} is given by the square root of $\sigma_1$, which is therefore independent of the mesh size. Then, considering the properties of the generalized Rayleigh quotient, the following characterization holds:
            \begin{equation}\label{eqn:h_combined_algebraic}
                \tilde\beta_3^2 = \sigma_1 = \min_{\{\mathsf{v} \in \mathbb{R}^n | \mathsf{u^T}\mathsf{Av}  =0 \>\> \text{for}\>  \mathsf{u} \in \ker(\mathsf{B}) \cap \ker(\mathsf{C})\}} \frac{\mathsf{v^T B^T M_p^{-1}B v} + \mathsf{v^T} \mathsf{C^T} h_\Gamma^{-1} \mathsf{M_\lambda^{-1}}\mathsf{C}  \mathsf v}{\mathsf{v^T}\mathsf{A} \mathsf v}.
            \end{equation}
            Furthermore, using the spectral equivalence between $(h_\Gamma \mathsf{M_\lambda})^{-1}$ and $\mathsf{M_\lambda^{-2}}$ as in the previous case, we can conclude that
            \begin{equation}\label{eqn:combined_algebraic}
                \bar\beta_3^2 \le \min_{\{\mathsf{v} \in \mathbb{R}^n | \mathsf{u^T}\mathsf{Av}  =0 \>\> \text{for}\>  \mathsf{u} \in \ker(\mathsf{B}) \cap \ker(\mathsf{C})\}} \frac{\mathsf{v^T B^T M_p^{-1}B v} + \mathsf{v^T} \mathsf{C^T} \mathsf{M_\lambda^{-2}}\mathsf{C}  \mathsf v}{\mathsf{v^T}\mathsf{A} \mathsf v}.
            \end{equation}
        \end{proof}
    }

    {
        \begin{theorem}\label{thm:spectral_AL_independence_of_h}
            Let $V_h$, $Q_h$ and $\Lambda_h$ be defined as in Section~\ref{sec:discr_stokes_problem}. If $\mathsf{Q} \coloneqq \mathsf{M_p}$, and $\mathsf{W} \coloneqq \mathsf{M_\lambda^2}$, then the lower bound
            in $$\operatorname{Spec}(\mathcal{P}_{\gamma \delta}^{-1} \mathcal{A}_{\gamma \delta}) \subseteq \left [ \min (\eta, \epsilon, \theta), 1 \right ]$$
            from Theorem~\ref{thm:spectral_AL} is bounded away from zero by a positive constant independent of the discretization parameters $h_\Omega$ and $h_\Gamma$.
        \end{theorem}}
\begin{proof}
    {
        From the proof of Theorem~\ref{thm:spectral_AL} (more precisely from Equation~\eqref{eqn:11}), we know that all eigenvalues of $\mathcal{P}^{-1}_{ \gamma \delta} \mathcal{A}_{\gamma \delta}$, except for those that are equal to $0$ and $1$, coincide with the eigenvalues of the following generalized eigenvalue problem:
        \begin{equation}\label{eq:gen_EP}
            (\gamma  \mathsf{B^T} \mathsf{Q}^{-1}\mathsf{B}  + \delta  \mathsf{C^T} \mathsf{W}^{-1}\mathsf{C}) \mathsf x = \lambda (\mathsf{A}+\gamma \mathsf{B^T} \mathsf{Q^{-1}} \mathsf{B} +\delta \mathsf{C^T} \mathsf{W^{-1}} \mathsf{C})\mathsf x,
        \end{equation}
        and they are all real. Therefore, the corresponding eigenvector can also be chosen to be real and, from now on, $\mathsf x^\ast$ is replaced by $\mathsf{x^T}$.
        Our goal is to show that $\lambda_{\min}^+$, the smallest positive eigenvalue of~\eqref{eq:gen_EP}, is bounded away from zero. To this end, we analyze separately the three cases considered in Theorem~\ref{thm:spectral_AL}. }
    \begin{itemize}
        {
        \item $\mathsf x                 \in \ker (\mathsf{C}) \setminus   \ker (\mathsf{B})$. We want to estimate $\eta$.
              From Theorem~\ref{thm:spectral_AL}, we know that
              \begin{equation}\label{eq:12}
                  \eta \le \lambda = \frac{\gamma \mathsf{x^T}  \mathsf{B^T} \mathsf{M_p^{-1}}\mathsf{B} \mathsf x }{\mathsf{x^T}(\mathsf{A} +\gamma \mathsf{B^T} \mathsf{M_p^{-1}} \mathsf{B} )\mathsf x}.
              \end{equation}
              Therefore, $\eta$ coincides with the smallest positive eigenvalue of the generalized eigenvalue problem
              $$\gamma \mathsf{B^T} \mathsf{M_p^{-1}}\mathsf{B} \mathsf{x} = \lambda (\mathsf{A} +\gamma \mathsf{B^T} \mathsf{M_p^{-1}} \mathsf{B} ) \mathsf{x}.$$
              In other words,
              $$\eta = \min_{\mathsf{v} \in \mathcal{R}}\frac{\gamma \mathsf{v^T} ( \mathsf{B^T} \mathsf{M_p^{-1}}\mathsf{B}) \mathsf v }{\mathsf{v^T}(\mathsf{A} +\gamma \mathsf{B^T} \mathsf{M_p^{-1}} \mathsf{B} )\mathsf v},$$
              where $\mathcal{R}$ is defined as
              $$\mathcal{R}\coloneqq \{\mathsf{v} \in \mathbb{R}^n | \mathsf{u^T}(\mathsf{A} +\gamma \mathsf{B^T} \mathsf{M_p^{-1}} \mathsf{B} )\mathsf v=0 \quad \forall \mathsf{u} \in \ker(\mathsf{B})\}.$$
              Since $\mathsf{u} \in \ker(\mathsf{B})$, it follows that the condition in $\mathcal{R}$ reduces to
              $$\mathsf{u^T}\mathsf{A} \mathsf{v} = 0 \quad \forall \mathsf{u} \in \ker(\mathsf{B}),$$
              and thus we can equivalently define
              $$\mathcal{R}\coloneqq \{\mathsf{v} \in \mathbb{R}^n | \mathsf{u^T}\mathsf{A} \mathsf v=0 \quad \forall \mathsf{u} \in \ker(\mathsf{B})\}.$$
              Next, we multiply and then divide by $\mathsf{v^T} \mathsf{A} \mathsf v$ to obtain
              $$\eta =\min_{\mathsf{v}\in \mathcal{R}} \frac{\gamma \mathsf{v^T}   \mathsf{B^T} \mathsf{M_p^{-1}}\mathsf{B} \mathsf v }{\mathsf{v^T} \mathsf{A} \mathsf v} \>\frac{\mathsf{v^T} \mathsf{A} \mathsf v}{ \mathsf{v^T} \mathsf{A} \mathsf v +\gamma \mathsf{v^T} \mathsf{B^T} \mathsf{M_p^{-1}} \mathsf{B} \mathsf v}. $$
              Let $r(\mathsf v) \coloneqq \frac{\mathsf{v^T} \mathsf{B^T} \mathsf{M_p^{-1}} \mathsf{B} \mathsf v}{\mathsf{v^T} \mathsf{A} \mathsf v} $ be the generalized Rayleigh quotient associated with $\mathsf{B^T} \mathsf{M_p^{-1}}\mathsf{B}$ and $\mathsf{A}$. Then we observe that $\eta$ can be written as
              $$\eta = \min_{\mathsf{v}\in \mathcal{R}} \frac{\gamma r(\mathsf v)}{1+\gamma r(\mathsf v)} = \min_{\mathsf{v}\in \mathcal{R}} f\left(r(\mathsf v)\right).$$
              Since $f(r(\mathsf{v}))=\frac{\gamma r(\mathsf v)}{1+\gamma r(\mathsf v)}$ is a monotone increasing function for $r(\mathsf{v}) > 0$,
              $$\min_{\mathsf{v}\in \mathcal{R}} f\left(r(\mathsf v)\right) = f(\min_{\mathsf{v}\in \mathcal{R}} r(\mathsf v)).$$

              From Lemma~\eqref{lemma:algebraic_infsup_B}  we have $\min_{\mathsf{v}\in \mathcal{R}} r(\mathsf v) = \beta_1^2$, therefore we can conclude that
              \begin{equation}\label{eq:firstbound}
                  \eta = f(\beta_1^2) =\frac{\gamma \beta_1^2}{1+\gamma \beta_1^2}>0,
              \end{equation}
              uniformly in $h_\Omega$, which completes the estimate for $\eta$.}
              {
        \item $\mathsf x \in  \ker (\mathsf{B})\setminus  \ker (\mathsf{C})$. From Theorem~\ref{thm:spectral_AL}, we know that
              \begin{equation}\label{eq:eigenvaluesC}
                  \epsilon \le \lambda = \frac{\delta \mathsf{x^T}   \mathsf{C^T} \mathsf{M_\lambda^{-2}}\mathsf{C} \mathsf x }{\mathsf{x^T} (\mathsf{A} + \delta \mathsf{C^T} \mathsf{M_\lambda^{-2}} \mathsf{C} )\mathsf x}.
              \end{equation}
              Arguing as in the previous case, $\epsilon$ can be written as
              $$\epsilon = \min_{\mathsf{v} \in \mathcal{R}}\frac{\delta \mathsf{v^T}  \mathsf{C^T} \mathsf{M_\lambda^{-2}}\mathsf{C} \mathsf v}{\mathsf{v^T}(\mathsf{A} +\delta \mathsf{C^T} \mathsf{M_\lambda^{-2}} \mathsf{C} )\mathsf v},$$
              where, to keep the notation simple, we will again denote by  $\mathcal{R}$ the set of vectors $\mathsf{v}$ that are $\mathsf{A}$-orthogonal to $\ker(\mathsf{C})$.\\
              We multiply and then divide by $\mathsf{v^T} \mathsf{A} \mathsf v$ to obtain
              $$\epsilon =\min_{\mathsf{v}\in \mathcal{R}} \frac{\delta \mathsf{v^T}   \mathsf{C^T} \mathsf{M_\lambda^{-2}}\mathsf{C} \mathsf v }{\mathsf{v^T} \mathsf{A} \mathsf v} \>\frac{\mathsf{v^T} \mathsf{A} \mathsf v}{ \mathsf{v^T} \mathsf{A} \mathsf v +\delta \mathsf{v^T} \mathsf{C^T} \mathsf{M_\lambda^{-2}} \mathsf{C} \mathsf v}. $$
              Once again, we define the generalized Rayleigh quotient $r(\mathsf v) \coloneqq \frac{ \mathsf{v^T} \mathsf{C^T} \mathsf{M_\lambda^{-2}} \mathsf{C} \mathsf v}{\mathsf{v^T} \mathsf A \mathsf v}$, so that $\epsilon$ can be written as
              $$\epsilon = \min_{\mathsf{v}\in \mathcal{R}} \frac{\delta r(\mathsf v)}{1+\delta r(\mathsf v)} = \min_{\mathsf{v}\in \mathcal{R}} f\left(r(\mathsf v)\right),$$
              where $f(r(\mathsf{v}))$ is a monotone increasing function for $r(\mathsf{v}) > 0$. Therefore, we need to characterize $f(\min_{\mathsf{v}\in \mathcal{R}} r(\mathsf{v}))$. To this aim, we use Lemma~\ref{lemma:algebraic_infsup_C}, from which we know that $\min_{\mathsf{v} \in \mathcal{R}} r(\mathsf v) \ge \bar\beta_2^2$, therefore it follows that
              \begin{equation}\label{eq:secondcase}
                  \epsilon \ge f(\bar\beta_2^2) = \frac{\delta \bar \beta_2^2}{1+\delta \bar \beta_2^2}>0,
              \end{equation}
              uniformly in $h_\Gamma$.
              \begin{remark}
                  We emphasize that this result, derived in the context of the Stokes fictitious domain problem, also applies to the Poisson fictitious domain problem. In particular, it shows that the eigenvalues of the preconditioned matrix remain uniformly bounded away from zero also for the Poisson problem.
              \end{remark}
              }
              
              \vspace{0.1cm}
              {
        \item \( \mathsf{x} \notin \ker(\mathsf{B}) \cup \ker(\mathsf{C}) \).
              From Theorem~\ref{thm:spectral_AL}, we know that
              $$
                  \theta \le \lambda =
                  \frac{ \mathsf{x^T} \big( \gamma\, \mathsf{B^T}\mathsf{M_p^{-1}} \mathsf{B} + \delta\, \mathsf{C^T}\mathsf{M_\lambda^{-2}} \mathsf{C} \big) \mathsf{x} }
                  { \mathsf{x^T}\big( \mathsf{A} + \gamma\, \mathsf{B^T}\mathsf{M_p^{-1}} \mathsf{B} + \delta\, \mathsf{C^T}\mathsf{M_\lambda^{-2}} \mathsf{C} \big) \mathsf{x} }.
              $$
              As in the previous cases, we have
              $$
                  \theta =
                  \min_{\mathsf{v} \in \mathcal{R}}
                  \frac{ \mathsf{v^T}\big( \gamma\, \mathsf{B^T}\mathsf{M_p^{-1}} \mathsf{B} + \delta\, \mathsf{C^T}\mathsf{M_\lambda^{-2}} \mathsf{C} \big) \mathsf{v} }
                  { \mathsf{v^T}\big( \mathsf{A} + \gamma\, \mathsf{B^T}\mathsf{M_p^{-1}} \mathsf{B} + \delta\, \mathsf{C^T}\mathsf{M_\lambda^{-2}} \mathsf{C} \big) \mathsf{v} },
              $$
              where now
              $$
                  \mathcal{R} \coloneqq \big\{ \mathsf{v} \in \mathbb{R}^n \,\big|\, \mathsf{u^T}\mathsf{A} \mathsf{v} = 0 \quad \forall \mathsf{u} \in \ker(\mathsf{B}) \cap \ker(\mathsf{C}) \big\}.
              $$
              Proceeding as before, we multiply and divide by \( \mathsf{v^T}\mathsf{A} \mathsf{v} \), and define the generalized Rayleigh quotient
              $$
                  r(\mathsf{v}) \coloneqq
                  \frac{ \mathsf{v^T}\big(  \mathsf{B^T}\mathsf{M_p^{-1}} \mathsf{B} + \mathsf{C^T}\mathsf{M_\lambda^{-2}} \mathsf{C} \big) \mathsf{v} }
                  { \mathsf{v^T}\mathsf{A} \mathsf{v} }.
              $$

              It follows that
              $$
                  \theta \geq
                  f\left(\min \{\gamma,\delta\} \, \min_{\mathsf{v} \in \mathcal{R}} r(\mathsf{v}) \right),
              $$
              where $ f(\cdot)$  is defined analogously as before. Owing to Lemma~\eqref{lemma:algebraic_infsup_BC}, we know that
              $
                  \min_{\mathsf{v} \in \mathcal{R}} r(\mathsf{v}) \geq \bar\beta_3^2,
              $ which in turn implies
              \begin{equation}\label{eqn:thirdcase}
                  \theta \geq f\left( \min \{\gamma, \delta\} \, \bar\beta_3^2\right) = \frac{\min \{\gamma, \delta\} \,\bar\beta_3^2}{1+ \min \{\gamma, \delta\} \,\bar\beta_3^2} > 0,
              \end{equation}
              uniformly in \(h_\Omega\) and \(h_\Gamma\).
              }
    \end{itemize}
\end{proof}
\vspace{0.1cm}
{
    \begin{remark}[Three-dimensional case]
        When $\Omega \subset \mathbb{R}^3$, the immersed domain $\Gamma$ is discretized with a surface mesh $\Gamma_h$. Assuming a quasi-uniform discretization for $\Gamma_h$, it holds

        \begin{equation}
            c h_\Gamma^2 \le \frac{ \mathsf{\upmu^T}\mathsf{M_\lambda} \mathsf{\upmu}  }{\mathsf{\upmu^T} \mathsf{\upmu}} \le C h_\Gamma^2 \quad \quad \forall \mathsf{\upmu} \in \mathbb{R}^l,
        \end{equation}
        yielding (see also Appendix ~\ref{sec:app:spectral_h_mass})
        \begin{equation}\label{eq:bounds1}
            0 < \frac{c}{C^2 h_\Gamma} \le \frac{\mathsf{\upmu^T} \mathsf{M_\lambda^{-2}} \mathsf{\upmu}}{\mathsf{\upmu^T} (h_\Gamma \mathsf{M_\lambda})^{-1} \mathsf{\upmu}} \le \frac{C}{c^2 h_\Gamma}.
        \end{equation}
        Setting again $\mathsf{\upmu} = C \mathsf v$ in the proof of Lemma~\ref{lemma:algebraic_infsup_C}, it holds that
        \begin{equation}\label{eq:2d}
            \min_{\{\mathsf{v} \in \mathbb{R}^n | \mathsf{u^T}\mathsf{Av}  =0 \text{ for }  \mathsf{u} \in \ker(\mathsf{C}\}}\frac{\mathsf{v^T} \mathsf{C^T} \mathsf{M_\lambda^{-2}}\mathsf{C}\mathsf{v} }{ \mathsf{v^T}\mathsf{A} \mathsf v } \ge \frac{\bar \beta_2^2}{h_\Gamma}.
        \end{equation}
        Following the same steps used to derive Equation~\eqref{eq:secondcase}, we employ the inequality in~\eqref{eq:2d} above to obtain
        \begin{equation}\label{eqn:bound_h_beta}
            0<\frac{\delta \bar \beta_2^2}{h_\Gamma+\delta \bar \beta_2^2} \le \epsilon \quad \quad \text{for}\>\>  \mathsf x \in  \ker (\mathsf{B})\setminus  \ker (\mathsf{C}).
        \end{equation}
        By repeating the same reasoning but starting from Lemma~\ref{lemma:algebraic_infsup_BC}, and simply adapting Equation~\eqref{eqn:h_combined_algebraic}, we get
        \begin{equation}\label{eqn:bound_3d}
            0 <  \frac{\min \{\gamma, \delta\} \, \tilde{\beta}^2_3}{\max \{1, \frac{C^2 h_\Gamma}{c}\} \,+\min \{\gamma, \delta\} \, \tilde{\beta}^2_3} \leq \theta \quad \quad \text{for} \>\>\mathsf{x} \notin \ker(\mathsf{B}) \cup \ker(\mathsf{C}).
        \end{equation}
        Notice how in these cases the lower bound for the eigenvalues of the preconditioned matrix depends explicitly on $h_\Gamma$. However, it is immediate to observe that as $h_\Gamma \rightarrow 0$, the lower bound in~\eqref{eqn:bound_h_beta} tends to $1$, and also~\eqref{eqn:bound_3d} remains bounded away from zero. Therefore, we can still conclude that also in this scenario the bounds are robust with respect to the discretization parameters.
    \end{remark}
}

\section{Spectral analysis of the inexact variant of $\mathcal{P}_{\gamma \delta}$ }\label{sec:inexact}
In this Section, we discuss in detail the eigenvalue distribution of the preconditioned matrix when an inexact version of the proposed AL-based preconditioner~\eqref{eqn:prec_AL_stokes} is employed. We mainly follow the analysis presented in~\cite{bakrani2023preconditioningtechniquesclassdouble}. For the sake of readability, we now drop $\gamma$, $\delta$ and write $\mathsf{{\bar A}}$ in place of $\mathsf{{A_{\gamma\delta}}}$ {\color{black}and $\mathsf{P}$ in place of $\mathcal{P}_{\gamma \delta}$}, so that our \emph{ideal} preconditioner reads
\begin{equation}\label{eq:prec_id}
    \mathsf{P}=
    \begin{bmatrix}
        \mathsf{{\bar A}} & \mathsf{B^T}                & \mathsf{C^T}                \\
        0                 & -\frac{1}{\gamma}\mathsf{Q} & 0                           \\
        0                 & 0                           & -\frac{1}{\delta}\mathsf{W}
    \end{bmatrix}.
\end{equation}
In order to use the same notation as in~\cite{bakrani2023preconditioningtechniquesclassdouble}, we define
$$\mathsf{S} \coloneqq \frac{1}{\gamma}\mathsf{Q} \qquad \text{and} \qquad \mathsf{X} \coloneqq \frac{1}{\delta}\mathsf{W}.$$
Due to the expensive solve associated in particular with the (1,1)-block, the ideal preconditioner $\mathsf{P}$ is not practical and needs to be replaced by an approximation. In practice, we will employ approximations also for the matrices $\mathsf{Q}$ and $\mathsf{W}$. Hence, we are to analyze the following inexact version:
$$\mathsf{\Bar{P}} \coloneqq \begin{bmatrix}
        \widehat{\mathsf{A}} & \mathsf{B^T}          & \mathsf{C^T}          \\
        0                    & -\widehat{\mathsf{S}} & 0                     \\
        0                    & 0                     & -\widehat{\mathsf{X}}
    \end{bmatrix},$$ where $\widehat{\mathsf{A}}$, $\widehat{\mathsf{S}}$, and $\widehat{\mathsf{X}}$ represent symmetric positive definite approximations of $\mathsf{\bar A}$, $\mathsf{S}$, and $\mathsf{X}$, respectively. The spectral properties of the preconditioned matrix will be given in terms of the eigenvalues of $\widehat{\mathsf{A}}^{-1} \mathsf{\bar A}$, $\widehat{\mathsf{S}}^{-1} \mathsf{\Tilde{S}}$, and $\widehat{\mathsf{X}}^{-1} \mathsf{\Tilde{X}}$, where $\Tilde{\mathsf{S}} = \mathsf{B}\widehat{\mathsf{A}}^{-1} \mathsf{B^T}$ and $\Tilde{\mathsf{X}} = \mathsf{C}\widehat{\mathsf{A}}^{-1} \mathsf{C^T}$. We define

\begin{equation}
    \gamma_{\text{min}}^A=\lambda_{\text{min}}(\widehat{\mathsf{A}}^{-1} \mathsf{\bar A}), \qquad \gamma_{\text{max}}^A=\lambda_{\text{max}}(\widehat{\mathsf{A}}^{-1} \mathsf{\bar A}), \qquad \gamma_A \in [\gamma_{\text{min}}^A,\gamma_{\text{max}}^A],
\end{equation}
\begin{equation}
    \gamma_{\text{min}}^S=\lambda_{\text{min}}^+(\widehat{\mathsf{S}}^{-1} \mathsf{\Tilde{S}}), \qquad \gamma_{\text{max}}^S=\lambda_{\text{max}}(\widehat{\mathsf{S}}^{-1} \mathsf{\Tilde{S}}), \qquad \gamma_S \in [\gamma_{\text{min}}^S,\gamma_{\text{max}}^S],
\end{equation}
\begin{equation}
    \gamma_{\text{min}}^X=\lambda_{\text{min}}(\widehat{\mathsf{X}}^{-1} \mathsf{\Tilde{X}}), \qquad \gamma_{\text{max}}^X=\lambda_{\text{max}}(\widehat{\mathsf{X}}^{-1} \mathsf{\Tilde{X}}), \qquad \gamma_X \in [\gamma_{\text{min}}^X,\gamma_{\text{max}}^X].
\end{equation}

We note that $\mathsf{B}$ is rank deficient by 1 and, consequently, the matrix $\Tilde{\mathsf S}$ is symmetric positive semidefinite. Moreover, since $\widehat{\mathsf{S}}^{-1} \Tilde{\mathsf S}$ is similar to $\widehat{\mathsf{S}}^{-\frac{1}{2}} \Tilde{\mathsf{S}} \widehat{\mathsf{S}}^{-\frac{1}{2}}$, which is a symmetric and positive semidefinite matrix, the eigenvalues of $\widehat{\mathsf{S}}^{-1} \Tilde{\mathsf S}$ are nonnegative {\color{black}and its smallest eigenvalue is equal to $0$}.
However, since the presence of a zero eigenvalue does not affect the convergence of preconditioned GMRES, we focus only on the smallest positive eigenvalue, denoted as $\lambda_{\min} ^+(\widehat{\mathsf{S}}^{-1} \mathsf{\Tilde{S}})$.

We then use the fact that looking for the eigenvalues of $\Bar{\mathsf{P}}^{-1}\mathcal{A}_{\gamma \delta}$ is equivalent to solving

\begin{equation}\label{eqn:equivalent_eigs}
    \mathsf{\Bar{D}}^{-\frac{1}{2}}\mathcal{A_{\gamma \delta}}\mathsf{\Bar{D}}^{-\frac{1}{2}} \mathsf{w} = \lambda \mathsf{\Bar{D}}^{-\frac{1}{2}}\mathsf{\Bar{P}}\mathsf{\Bar{D}}^{-\frac{1}{2}} \mathsf{w},
\end{equation}where

$$\mathsf{\Bar{D}} \coloneqq \begin{bmatrix}
        \widehat{\mathsf{A}} & 0                    & 0                    \\
        0                    & \widehat{\mathsf{S}} & 0                    \\
        0                    & 0                    & \widehat{\mathsf{X}}
    \end{bmatrix}.$$

\noindent Let $\mathsf{\Tilde{A}} \coloneqq \widehat{\mathsf{A}}^{-\frac{1}{2}} {\mathsf{\bar A}} \widehat{\mathsf{A}}^{-\frac{1}{2}}$, $\mathsf{R}\coloneqq  \widehat{\mathsf{S}}^{-\frac{1}{2}} \mathsf{B} \widehat{\mathsf{A}}^{-\frac{1}{2}}$, and $\mathsf{K} \coloneqq \widehat{\mathsf{X}}^{-\frac{1}{2}} \mathsf{C} \widehat{\mathsf{A}}^{-\frac{1}{2}}$. Since $\widehat{\mathsf{A}}$, $\widehat{\mathsf{S}}$, and $\widehat{\mathsf{X}}$ are SPD,
$\mathsf{B}$ is rank deficient by 1, and $\mathsf{C}$ has full row rank, then $\Tilde{\mathsf A}$ is SPD, $\mathsf R$ has the same rank as $\mathsf{B}$ and $\mathsf{K}$ has full row rank.
The explicit computation of~\eqref{eqn:equivalent_eigs}  yields the following generalized eigenvalue problem:

\begin{equation}
    \begin{bmatrix}\label{eqn:inexact_prec_system}
        \mathsf{\Tilde{A}} & \mathsf{R^T} & \mathsf{K^T} \\
        \mathsf{R}         & 0            & 0            \\
        \mathsf{K}         & 0            & 0
    \end{bmatrix}
    \begin{bmatrix}
        \mathsf x \\
        \mathsf y \\
        \mathsf z
    \end{bmatrix}=\lambda
    \begin{bmatrix}
        \mathsf{I} & \mathsf{R^T} & \mathsf{K^T} \\
        0          & -\mathsf{I}  & 0            \\
        0          & 0            & -\mathsf{I}
    \end{bmatrix}
    \begin{bmatrix}
        \mathsf x \\
        \mathsf y \\
        \mathsf z
    \end{bmatrix}.
\end{equation} Let $\lambda$ be an eigenvalue of $\Bar{\mathsf{P}}^{-1}\mathcal{A}_{\gamma \delta}$ and $(\mathsf x;\mathsf y; \mathsf z)$ a corresponding eigenvector such that $||\mathsf x||^2+||\mathsf y||^2+||\mathsf z||^2=1$. It follows from~\eqref{eqn:inexact_prec_system} that

\begin{equation}
    \mathsf{\Tilde{A}} \mathsf x - \lambda \mathsf x = (\lambda-1)\mathsf{R^T} \mathsf y + (\lambda-1)\mathsf{K^T} \mathsf z, \label{eqn:first_inexact}
\end{equation}
\begin{equation}
    \mathsf{R}\mathsf x = -\lambda \mathsf y,\label{eqn:second_inexact}
\end{equation}
\begin{equation}
    \mathsf{K} \mathsf x = -\lambda \mathsf z.\label{eqn:third_inexact}
\end{equation}
{\color{black}Following~\cite{bakrani2023preconditioningtechniquesclassdouble}, we make the assumption that $1 \in [\gamma_{\text{min}}^A,\gamma_{\text{max}}^A] $, which is very commonly satisfied in practice. Since $\widehat{\mathsf{A}}^{-1} \mathsf{\bar A}$ is similar to $\Tilde{\mathsf{A}}$, $\lambda =1$ is also an eigenvalue of $\Bar{\mathsf{P}}^{-1}\mathcal{A}_{\gamma \delta}$ with corresponding eigenvector $(\mathsf x;-\mathsf{R}\mathsf x;\mathsf{-K} \mathsf x)$, provided that $\mathsf x\ne0$. Hence, from now on we assume $\lambda \ne 1$ and $\mathsf x \ne 0$}.
\begin{itemize}
    \item $\mathsf x \in \ker(\mathsf{R}) \cap \ker(\mathsf{K})$.\\
          Equations~\eqref{eqn:second_inexact} and~\eqref{eqn:third_inexact} imply $\mathsf y=\mathsf z=0$. Then, we readily obtain from~\eqref{eqn:first_inexact}
          $$ \mathsf{\Tilde{A}} \mathsf x = \lambda \mathsf x, $$which implies that $\lambda $ is real and
          \begin{equation}
              \lambda \in [\gamma_{\text{min}}^A,\gamma_{\text{max}}^A].
          \end{equation}
          The associated eigenvector is of the form $(\mathsf x;0;0)$.
    \item $\mathsf x  \in \ker(\mathsf{K})\setminus \ker(\mathsf{R})$.\\
          Equation~\eqref{eqn:third_inexact} implies $\mathsf z=0$, which gives

          \begin{equation}\label{eqn:x_inexact_kerK}
              \mathsf{\Tilde{A}} \mathsf x - \lambda \mathsf x = (\lambda-1)\mathsf{R^T} \mathsf y,
          \end{equation}
          \begin{equation}\label{eqn:Rx_inexact_kerK}
              \mathsf{R} \mathsf x = -\lambda \mathsf y.
          \end{equation}
          After multiplying the first equation by $\mathsf x^{*}$, the {\color{black}conjugate transpose} of the second by $\mathsf y$ on the right, and inserting the second equation in the first we obtain
          $$\mathsf x^{*}  \mathsf{\Tilde{A}} \mathsf x - \lambda ||\mathsf x||^2 = -|\lambda|^2 ||\mathsf y||^2 + \Bar{\lambda} ||\mathsf y||^2.$$ Using the fact that $||\mathsf x||^2=1-||\mathsf y||^2$, we get
          \begin{equation}\label{eqn:kerK}
              \mathsf x^{*}  \mathsf{\Tilde{A}} \mathsf x - \lambda +(\lambda-\Bar{\lambda}) ||\mathsf y||^2 + |\lambda|^2 ||\mathsf y||^2=0.
          \end{equation}Writing now $\lambda =a + ib$, we obtain the following system for the real and imaginary parts
          \begin{equation}\label{eqn:be_real}
              \begin{cases}
                  \mathsf x^{*}  \mathsf{\Tilde{A} \mathsf x} - a + (a^2+b^2)||\mathsf y||^2=0, \\
                  b(2||\mathsf y||^2 -1)=0.
              \end{cases}
          \end{equation}
          From the second equation, it follows that $b=0$, or $||\mathsf y||^2=\frac{1}{2}.$ We assume $b\ne0$. Therefore, Equation~\eqref{eqn:kerK} in the system above reads as
          \begin{equation}\label{eqn:kerK_bne0}
              2\,\mathsf x^{*} \mathsf{\Tilde{A}} \mathsf x - \lambda-\Bar{\lambda} + |\lambda|^2=0
          \end{equation}

          Exploiting the identity $|\lambda|^2-\lambda-\Bar{\lambda}= |\lambda-1|^2 -1$, and dividing the last equation by $\mathsf x^{*}\mathsf x= ||\mathsf x||^2=\frac{1}{2},$ we get
          $$2 |\lambda - 1|^2 = 2-2\> \frac{\mathsf x^{*} \mathsf{\Tilde{A}} \mathsf x}{\mathsf x^{*}\mathsf x},$$
          therefore
          $$|\lambda -1|^2 \leq 1 - \gamma_{\text{min}}^A.$$
          Hence, if $1-\gamma_{\text{min}}^A \geq 0$, we have
          \begin{equation}\label{eq:circle}
              |\lambda -1| \leq \sqrt{1 - \gamma_{\text{min}}^A}.
          \end{equation}
          Conversely, if $1 - \gamma_{\text{min}}^A <0$, then there exists no $\lambda$ with nonzero imaginary part satisfying equality~\eqref{eqn:kerK_bne0}. Using $\lambda = a+ib$, Equation~\eqref{eq:circle} can be written as
          $$(a-1)^2 +b^2 \le 1-\gamma_{\text{min}}^A,$$
          which represents a circle centered in $(1,0)$ with radius $\sqrt{1-\gamma_{\text{min}}^A}$, meaning that if $1-\gamma_{\text{min}}^A \ge 0$, then the eigenvalues $\lambda$ lie in this circle.

          We now consider the case $b=0$. In this case, $\lambda$ is real, and the corresponding eigenvector can also be chosen to be real.
          Solving for $\mathsf x$ the Equation~\eqref{eqn:x_inexact_kerK} gives
          $$\mathsf x = \bigl( \mathsf{\Tilde{A} - \lambda I} \bigr)^{-1} (\lambda-1)\mathsf{R^T} \mathsf y,$$ which, plugged into Equation~\eqref{eqn:Rx_inexact_kerK} gives
          \begin{equation}\label{eqn:eqlemma}
              \mathsf R \bigl( \mathsf{\Tilde{A} - \lambda I} \bigr)^{-1} (1-\lambda)\mathsf{R^T} \mathsf y=\lambda \mathsf y
          \end{equation}

          We now quote the following Lemma from~\cite{bakrani2023preconditioningtechniquesclassdouble}, which will be used to characterized the real eigenvalues not lying in $[\gamma_{\min}^A, \gamma_{\max}^A]$.

          \begin{lemma}\label{lemma:lemma_1}

              {\color{black}Suppose that there exists an eigenvalue} $\lambda \not \in [\gamma_{\text{min}}^A,\gamma_{\text{max}}^A]$. Then, for arbitrary $\mathsf z \ne 0$, there exists a vector $\mathsf s \ne 0$ such that

              \begin{equation}
                  \frac{\mathsf {z^T} \bigl(\mathsf{\Tilde{A} - \lambda I} \bigr)^{-1}\mathsf z}{\mathsf {z^T} \mathsf z} = \Biggl( \frac{\mathsf{s^T} \mathsf{\Tilde{A}}\mathsf{s}}{\mathsf{s^T} \mathsf{s}} - \lambda \Biggr)^{-1} = (\gamma_A -\lambda)^{-1}
              \end{equation}where $\gamma_A \coloneqq \frac{\mathsf{s^T} \mathsf{\Tilde{A}}\mathsf{s}}{\mathsf{s^T} \mathsf{s}}$.
          \end{lemma}

          {
          Multiplying Equation~\eqref{eqn:eqlemma} by $\frac{\mathsf{y^T}}{\mathsf{y^T} \mathsf{y}}$, we obtain
          \[
              (1 - \lambda)\, \frac{(\mathsf{R^T} \mathsf{y})^\mathsf{T} (\mathsf{\Tilde{A}} - \lambda \mathsf{I})^{-1} (\mathsf{R^T} \mathsf{y})}{\mathsf{y^T} \mathsf{y}}=\lambda.
          \]
          Using the lemma above, this yields
          \[
              (1 - \lambda)\, (\gamma_A - \lambda)^{-1} \frac{\mathsf{y^T} \mathsf{R} \mathsf{R^T} \mathsf{y}}{\mathsf{y^T} \mathsf{y}} = \lambda.
          \]
          Note that, by the definition of $\mathsf{R}$, we have $\mathsf{R} \mathsf{R^T} = \mathsf{\hat{S}}^{-1/2} \mathsf{\Tilde{S}} \mathsf{\hat{S}}^{-1/2}$, which is similar to $\mathsf{\hat{S}}^{-1} \mathsf{\Tilde{S}}$. Therefore,
          \[
              \frac{\mathsf{y^T} \mathsf{R} \mathsf{R^T} \mathsf{y}}{\mathsf{y^T} \mathsf{y}} \in [\gamma_{\min}^S, \gamma_{\max}^S],
          \]
          and
          $$(1-\lambda) (\gamma_A -\lambda)^{-1} \gamma_S = \lambda.$$
          It follows that $\lambda$ satisfies the quadratic equation
          \[
              \lambda^2 - (\gamma_A + \gamma_S)\lambda + \gamma_S = 0.
          \]}
          \emph{The next steps are identical to the ones in~\cite{bakrani2023preconditioningtechniquesclassdouble}, and are reported here only for completeness.}

          The two solutions of the quadratic equation above are
          $$\lambda_{1,2}=\frac{\gamma_A+\gamma_S \pm \sqrt{(\gamma_A+\gamma_S)^2 - 4 \gamma_S}}{2}.$$
          The first root $\lambda_{1}$ can be bounded by
          $$\lambda_{1}=\frac{\gamma_A+\gamma_S + \sqrt{(\gamma_A+\gamma_S)^2 - 4 \gamma_S}}{2} \leq \gamma_A +\gamma_S\leq \gamma_{\text{max}}^A + \gamma_{\text{max}}^S,$$while
          $$\lambda_{2}=\frac{2 \gamma_S}{\gamma_A+\gamma_S + \sqrt{(\gamma_A+\gamma_S)^2 - 4 \gamma_S}} \geq \frac{\gamma_{\text{min}}^S}{\gamma_{\text{max}}^A+\gamma_{\text{max}}^S}.$$All in all, we have
          \begin{equation}
              \lambda \in \left[\frac{\gamma_{\text{min}}^S}{\gamma_{\text{max}}^A+\gamma_{\text{max}}^S}, \gamma_{\text{max}}^A + \gamma_{\text{max}}^S \right].
          \end{equation}

    \item $\mathsf x \in \ker(\mathsf{R}) \setminus \ker(\mathsf{K})$.

          \emph{Thanks to the particular form of the preconditioned system in~\eqref{eqn:inexact_prec_system}, this case is completely analogous to the previous one. Therefore, we only report the final results.}

          Proceeding in the same fashion as in the previous paragraph, we have two cases to distinguish, depending on the solution $b$ of Equation~\eqref{eqn:be_real}.
          When $b\ne0$:
          if $1-\gamma_{\text{min}}^A \geq 0$, we have
          \begin{equation}
              |\lambda -1| \leq \sqrt{1 - \gamma_{\text{min}}^A}.
          \end{equation}
          If $b=0$:
          \begin{equation}
              \lambda \in \left[\frac{\gamma_{\text{min}}^X}{\gamma_{\text{max}}^A+\gamma_{\text{max}}^X}, \gamma_{\text{max}}^A + \gamma_{\text{max}}^X \right].
          \end{equation}

    \item $\mathsf{x} \not \in \ker(\mathsf R) \cup \ker(\mathsf K)$.

          We start by multiplying Equation~\eqref{eqn:first_inexact} by $\mathsf x^{*}$, the {\color{black}conjugate transpose} of Equation~\eqref{eqn:second_inexact} by $\mathsf y$ on the right, and the {\color{black}conjugate transpose} of Equation~\eqref{eqn:third_inexact} by $\mathsf z$ on the right, obtaining
          \begin{equation}\label{eqn:x_not_in_both_kers}
              \mathsf x^{*}\mathsf{\Tilde{A}} \mathsf x - \lambda ||\mathsf x||^2 = (\lambda-1)\mathsf x^{*}\mathsf{R^T} \mathsf y + (\lambda-1)\mathsf x^{*}\mathsf{K^T} \mathsf z,
          \end{equation}
          \begin{equation}
              \mathsf x^{*}\mathsf{R^T}\mathsf y = -\Bar{\lambda} ||\mathsf y||^2,
          \end{equation}
          \begin{equation}
              \mathsf x^{*}\mathsf{K^T}\mathsf z = -\Bar{\lambda} ||\mathsf z||^2.
          \end{equation}

          By insertion of the last two equations in~\eqref{eqn:x_not_in_both_kers} we get
          $$\mathsf x^{*}\mathsf{\Tilde{A}} \mathsf x- \lambda ||\mathsf x||^2 = -(\lambda-1)\Bar{\lambda} ||\mathsf y||^2 -(\lambda-1)\Bar{\lambda}||\mathsf z||^2,
          $$which after simple algebra gives
          $$\mathsf x^{*}\mathsf{\Tilde{A}} \mathsf x - \lambda ||\mathsf x||^2 = -|\lambda|^2 (||\mathsf y||^2 + ||\mathsf z||^2) + \Bar{\lambda}(||\mathsf y||^2 + ||\mathsf z||^2).
          $$Using $||\mathsf x||^2+||\mathsf y||^2+||\mathsf z||^2=1$ {\color{black}yields}
          \begin{equation}\label{eqn:realimag}
              \mathsf x^{*}\mathsf{\Tilde{A}} \mathsf x = \Bigl( |\lambda|^2 + \lambda - \Bar{\lambda} \Bigr)||\mathsf x||^2 + \Bar{\lambda} - |\lambda|^2.
          \end{equation}
          By writing $\lambda \coloneqq a + ib$, we obtain the {\color{black}following} equations for the real and imaginary parts:
          \begin{equation}
              \begin{cases}\label{eqn:be_real_again}
                  \mathsf x^{*}\mathsf{\Tilde{A}} \mathsf x - (a^2+b^2) ||\mathsf x||^2 -a +a^2 +b^2=0, \\
                  b(1-2||\mathsf x||^2)=0,
              \end{cases}
          \end{equation}from which we get again $b=0$, or $||\mathsf x||^2=\frac{1}{2}$. We assume $b\ne0$, so that Equation~\eqref{eqn:realimag} becomes
          \begin{equation}\label{eqn:imag_in_both_kers}
              2\, \mathsf x^{*}\mathsf{\Tilde{A}} \mathsf x = -|\lambda|^2 + \lambda + \Bar{\lambda}.
          \end{equation}
          Exploiting again the identity $|\lambda|^2-\lambda-\Bar{\lambda}= |\lambda-1|^2 -1$, and dividing the last equation by $\mathsf x^{*}\mathsf x= ||\mathsf x||^2=\frac{1}{2},$ we get
          $$ |\lambda - 1|^2 = 1-\> \frac{\mathsf x^{*} \mathsf{\Tilde{A}} \mathsf x}{\mathsf x^{*}\mathsf x},$$
          therefore
          $$|\lambda -1|^2 \leq 1 - \gamma_{\text{min}}^A.$$
          Hence, if $1-\gamma_{\text{min}}^A \geq 0$, we have
          \begin{equation}
              |\lambda -1| \leq \sqrt{1 - \gamma_{\text{min}}^A}.
          \end{equation}
          If $1 - \gamma_{\text{min}}^A <0$, then there exists no $\lambda$ with nonzero imaginary part satisfying~\eqref{eqn:imag_in_both_kers}.

          We finally consider the case $b=0$. From Equations~\eqref{eqn:second_inexact} and~\eqref{eqn:third_inexact} we immediately obtain

          $$\mathsf y = -\frac{1}{\lambda}\mathsf{R}\mathsf x, \quad \quad \mathsf z = -\frac{1}{\lambda}\mathsf{K}\mathsf x.$$ Plugging these into Equation~\eqref{eqn:first_inexact}, we get

          $$ \mathsf{\Tilde{A}} \mathsf x - \lambda \mathsf x = -\frac{(\lambda-1)}{\lambda}\mathsf{R^TR}\mathsf x -\frac{(\lambda-1)}{\lambda}\mathsf{K^TK}\mathsf x,$$then we multiply both sides by $ \frac{\lambda \mathsf {x^{T}}}{\mathsf {x^{T}} \mathsf x}$, obtaining the {\color{black}following} quadratic equation in $\lambda${\color{black}:}

          \begin{equation}
              \lambda^2 - \lambda \Bigl( \frac{\mathsf {x^{T}} \mathsf{\Tilde{A}} \mathsf x}{\mathsf {x^{T}}\mathsf x} + \frac{\mathsf {x^{T}} \mathsf{R^T R} \mathsf x}{\mathsf {x^{T}}\mathsf x} + \frac{\mathsf {x^{T}} \mathsf{K^TK} \mathsf x}{\mathsf {x^{T}}\mathsf x} \Bigr) +  \Bigl( \frac{\mathsf {x^{T}} \mathsf{R^T R} \mathsf x}{\mathsf {x^{T}}\mathsf x} + \frac{\mathsf {x^{T}} \mathsf{K^TK} \mathsf x}{\mathsf {x^{T}}\mathsf x} \Bigr) =0.
          \end{equation}
          We now define $$\gamma_R \coloneqq \frac{\mathsf{x^T} \mathsf{R^T R}\mathsf x}{\mathsf{x^T} \mathsf x} \quad \text{and} \quad \gamma_K \coloneqq \frac{\mathsf{x^T}\mathsf{K^T K}\mathsf x}{\mathsf{x^T} \mathsf x}.$$
          Therefore,
          \begin{equation}
              \lambda^2 - \lambda ( \gamma_A + \gamma_R + \gamma_K) + \gamma_R + \gamma_K  =0.
          \end{equation}
          Solving for $\lambda$ gives the following two (real) solutions{\color{black}:}

          $$\lambda_{1,2}=\frac{\gamma_A+\gamma_R + \gamma_K \pm \sqrt{(\gamma_A+\gamma_R + \gamma_K)^2 - 4 (\gamma_R + \gamma_K)}}{2}.$$
          The first root $\lambda_{1}$ can be bounded by
          $$\lambda_{1}=\frac{\gamma_A+\gamma_R + \gamma_K + \sqrt{(\gamma_A+\gamma_R + \gamma_K)^2 - 4 (\gamma_R + \gamma_K)}}{2} \leq \gamma_A + \gamma_R + \gamma_K \leq \gamma_{\text{max}}^A + \gamma_{\text{max}}^R + \gamma_{\text{max}}^K,$$while
          $$\lambda_{2}=\frac{2(\gamma_R + \gamma_K)}{(\gamma_A+\gamma_R + \gamma_K)+ \sqrt{(\gamma_A+\gamma_R + \gamma_K)^2 - 4 (\gamma_R + \gamma_K)}} \geq \frac{ \gamma_{\text{min}}^R + \gamma_{\text{min}}^K}{\gamma_{\text{max}}^A + \gamma_{\text{max}}^R + \gamma_{\text{max}}^K}.$$

          From the singular value decomposition of $\mathsf{R}$ it follows that the nonzero eigenvalues of $\mathsf{R^TR}$ and $\mathsf{K^TK}$ are the same ones of $\mathsf{RR^T}$ and $\mathsf{KK^T}$, respectively. Therefore, we have
          $$\gamma_R \in [0,\gamma_{\text{max}}^S] \quad \text{and} \quad \gamma_K \in [0,\gamma_{\text{max}}^X].$$
          Moreover, the assumption $\mathsf x \not \in \ker(\mathsf R) \cup \ker(\mathsf K)$ means that we are effectively working in the space where $\mathsf{R^TR}$ and $\mathsf{K^TK}$ are strictly positive definite, thus $\gamma_R>0$ and $\gamma_K>0$. More precisely, we have $\gamma_R \in [\gamma_{\text{min}}^S,\gamma_{\text{max}}^S]$. The very same argument applies to $\gamma_K$, allowing to write
          $$\lambda_{2} \geq \frac{ \gamma_{\text{min}}^S + \gamma_{\text{min}}^X}{\gamma_{\text{max}}^A + \gamma_{\text{max}}^S + \gamma_{\text{max}}^X} > 0.$$
          We conclude that
          \begin{equation}
              \lambda \in \left[\frac{ \gamma_{\text{min}}^S + \gamma_{\text{min}}^X}{\gamma_{\text{max}}^A + \gamma_{\text{max}}^S + \gamma_{\text{max}}^X}, \gamma_{\text{max}}^A + \gamma_{\text{max}}^S +\gamma_{\text{max}}^X \right].
          \end{equation}

\end{itemize}

\subsection{Spectrum of the matrix preconditioned with the inexact variant of $\mathcal{P}_{\gamma \delta}$}
We verify the correctness of the bounds using the same configuration as in Section ~\ref{subsec:spectrum}.
As an inexact variant of the \emph{ideal} preconditioner in~\eqref{eqn:prec_AL_stokes},
we consider the case in which we approximate the mass matrices on the pressure and multiplier space with diagonal matrices. In particular, we consider $\widehat{\mathsf{M}}_{\mathsf{p}} \coloneqq \text{diag}(\mathsf{M_p})$, which is a widely used choice in the context of the classical Stokes problem~\cite{ElmanSilvesterWathen},
and $\widehat{\mathsf{M}}_\lambda \coloneqq \text{diag}(\mathsf{M_\lambda})^2$. As an approximation for $\mathsf{A_{\gamma\delta}}$, we employ its incomplete Cholesky factorization, computed using the \texttt{ichol} function in \textsc{Matlab} with a drop tolerance set to $10^{-1}$. Note that for the numerical experiments reported in Section~\ref{sec:numerical_experiments}, the augmented (1,1)-block is always inverted using a
CG solver preconditioned by a single V-cycle of AMG, with a loose tolerance of $10^{-2}$, whereas the inversion of $\mathsf{M_p}$ is performed using a CG solver preconditioned by the lumped pressure mass matrix. However, to verify the theoretical findings, here different approximations are considered for the augmented block $\mathsf{A_{\gamma\delta}}$ and for $\mathsf{M_p}$, for simplicity of implementation.\\
To summarize, the following (positive definite) approximations for $\mathsf{\bar A}$, $\mathsf{S}$, and $\mathsf{X}$ are considered:
\begin{itemize}
    \item $\mathsf{\widehat{A}} \coloneqq \text{ichol}(\mathsf{\bar A}, \text{drop\_tol=}10^{-1}),$
    \item $\mathsf{\widehat{S}} \coloneqq  \frac{1}{\gamma} \, \text{diag}\bigl(\mathsf{M_p}\bigr),$
    \item $\mathsf{\widehat{X}} \coloneqq \frac{1}{\delta} \, \text{diag}\bigl(\mathsf{M_\lambda}\bigr)^2.$
\end{itemize}
We show in Figure~\ref{fig:check_bounds} the spectrum of the preconditioned matrix $\Bar{\mathsf P}^{-1} \mathcal{A}_{\gamma \delta}$ when the inexact variant above is employed. In particular, we show the lower (green diamond) and upper (orange square) {\color{black}bound} for the real eigenvalues. We notice how the lower and upper bounds for the real part of the spectrum confirm the theoretical findings above. As already explained, $\lambda = 0$ is an eigenvalue of the original system, and hence it coincides with the numerical lower bound. Moreover, it is evident how all the imaginary eigenvalues are well-contained in the disk of radius $\sqrt{1- \gamma_{\min}^\mathsf{A}}$.

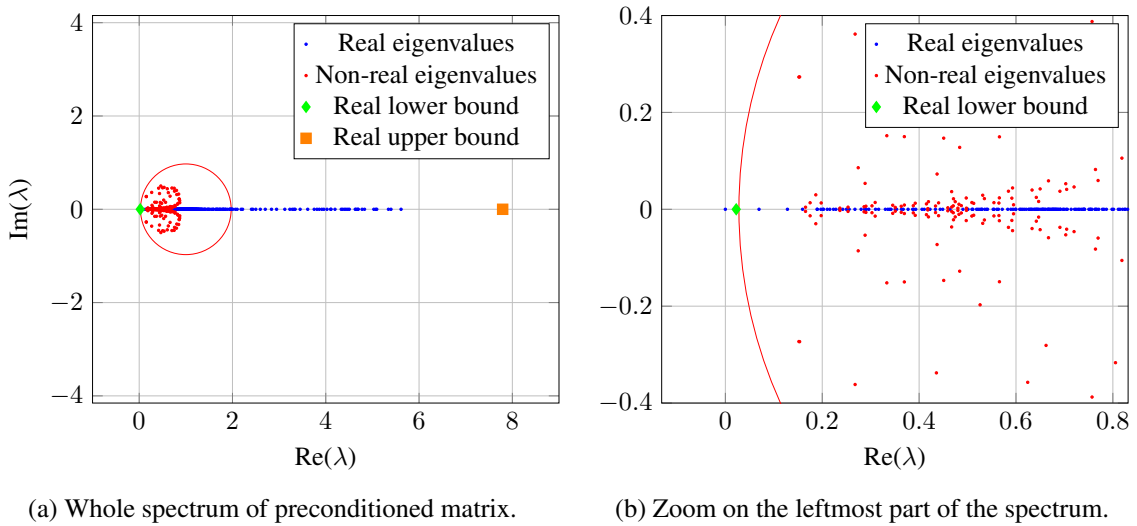
\begin{figure}[h]
    \centering
    \begin{subfigure}{0.45\textwidth} %
        \centering
        \begin{tikzpicture}[scale=0.9] %
            \begin{axis}[
                    xlabel={Re($\lambda$)}, ylabel={Im($\lambda$)},
                    xmin=-1, xmax=9,
                    ymin=-2, ymax=2,
                    grid=both,
                    axis equal,
                ]
                \addplot[only marks, mark=*, mark size=0.5pt, color=blue] table {data/stokes/real.dat};
                \addlegendentry{Real eigenvalues}
                \addplot[only marks, mark=*, mark size=0.5pt, color=red] table {data/stokes/img.dat};
                \addlegendentry{{\color{black}Non-real} eigenvalues}
                \addplot[green, thick, mark=diamond*, only marks]
                coordinates {
                        (0.0224096129431830, 0)

                    };
                \addlegendentry{Real lower bound}
                \addplot[orange, thick, mark=square*, only marks]
                coordinates {

                        (7.79243044745286, 0)
                    };
                \addlegendentry{Real upper bound }
                \draw[red, thin] (1,0) circle[radius=0.971990808981998];
                \addlegendentry{Circle, $r = \sqrt{1-\gamma_{\min}^A}$ }

            \end{axis}
        \end{tikzpicture}
        \caption{Whole spectrum of preconditioned matrix.}
    \end{subfigure}\hspace{0.5cm} %
    \begin{subfigure}{0.45\textwidth}
        \centering
        \begin{tikzpicture}[scale=0.9]
            \begin{axis}[
                    at={(7,0)}, %
                    xlabel={Re($\lambda$)},
                    xmin=-0.1, xmax=0.8,
                    ymin=-0.4, ymax=0.4,
                    grid=both,
                    axis equal,  %
                ]
                \addplot[only marks, mark=*, mark size=0.5pt, color=blue] table {data/stokes/real.dat};
                \addlegendentry{Real eigenvalues}
                \addplot[only marks, mark=*, mark size=0.5pt, color=red] table {data/stokes/img.dat};
                \addlegendentry{{\color{black}Non-real} eigenvalues}
                \addplot[green, thick, mark=diamond*, only marks]
                coordinates {
                        (0.0224096129431830, 0)

                    };
                \addlegendentry{Real lower bound}          %
                \draw[red,  thin] (1,0) circle[radius=0.971990808981998];
            \end{axis}
        \end{tikzpicture}
        \caption{Zoom on the leftmost part of the spectrum.}
    \end{subfigure}
    \vspace{0.2cm}
    \caption{Numerical check of the bounds for the preconditioned matrix when an inexact variant of the preconditioner is employed.}
    \label{fig:check_bounds}
\end{figure}

\section{Numerical experiments}\label{sec:numerical_experiments}

In this Section, we present a series of numerical experiments to assess the performance of the proposed AL-based preconditioners for the problems analyzed in this work.
All the experiments have been performed using the \textsc{C++} finite element library \textsc{deal.II}~\cite{dealIIdesign,dealII95} and are
available at a maintained GitHub repository provided by the authors\footnote{Available at~\texttt{\url{https://github.com/fdrmrc/fictitious_domain_AL_preconditioners}}}. Building on top of
a generic library, we can exploit (in a dimension-independent fashion) features such as adaptive mesh refinement, higher-order elements, handling of non-matching meshes, hanging node constraints, as
well as API to high-performance preconditioners provided by state-of-the-art linear algebra libraries such as~\textsc{Trilinos}~\cite{Trilinos} and~\text{PETSc}~\cite{petsc}. Our code is memory-distributed
and builds on the Message Passing Interface (MPI) communication model~\cite{GROPP1996789}.

In the examples involving the Poisson problem, we employ $\mathcal{Q}_1$ elements both for the background space $V_h$ and the immersed space $\Lambda_h$. For the Stokes test case, we employ the classical
Taylor-Hood $\mathcal{Q}_2$-$\mathcal{Q}_1$ stable pair for the velocity and pressure spaces $V_h$ and $Q_h$, and $\mathcal{Q}_1$ elements for the multiplier space $\Lambda_h$. In
the forthcoming Tables, we will denote by $|V_h|+|\Lambda_h|$ the total number of {\color{black}DoF}, which corresponds to the global size of the linear system in~\eqref{eqn:matrix_aug_p}. Similarly,
for the Stokes problem, we denote with $|V_h|+|Q_h|+|\Lambda_h|$ the total number of {\color{black}DoF} associated with the linear system in~\eqref{eqn:algebraic_form_stokes}. All tests are
performed using background meshes made of quadrilaterals or hexahedra and immersed boundary meshes made of segments and quadrilaterals. We perform an initial pre-processing of the background
grid $\Omega_h$ by applying a localized refinement around the interface, where most of the error is concentrated~\cite{Heltai2019}.
Sample two-dimensional grids resulting from this process in the two-dimensional case are shown in Figure~\ref{fig:interfaces}. A three-dimensional analogue is displayed in Figure~\ref{fig:torus_refined}.
{As we are using conforming elements on quadrilateral or hexahedral meshes for the background domain, local refinement procedures induce the presence of hanging nodes. As usual, the discrete solution
is enforced to be continuous through constraints on the nodal coefficients. In practice, the nodal unknowns associated with those nodes are eliminated by restricting their
value to an interpolation of the nodal values of the finite element unknown in its parent element. The interested reader is referred to~\cite{dealIIdesign} for the implementation details.} Finally, we investigate the iteration counts under simultaneous refinement of both the background and immersed meshes. All the solvers are started with the zero vector as initial guess. Throughout the experiments, we will set $\mathsf{Q=M_p}$, $\mathsf{W=M_\lambda^2}$, and we will consider
absolute tolerances $\mathrm{TOL}$ ranging from $10^{-8}$ to $10^{-10}$.

\subsection{Poisson problem}
We analyze different preconditioning techniques for the linear system~\eqref{eqn:algebraic_form}. In particular, we test three approaches which can be applied to our context: BFBt preconditioning~\cite{ElmanBFBt}, rational preconditioning~\cite{haznics}, and the AL-based preconditioner devised in Section~\ref{sec:AL_preconditioner}.\\
Unless stated otherwise, the resulting linear system~\eqref{eqn:algebraic_form} (or the equivalent augmented version~\eqref{eqn:matrix_aug_p}) is solved with FGMRES(30), preconditioned by each one of the preconditioners defined above. For
ease of presentation, we recall here the definition of the first two preconditioners. With the BFBt approach, the action of the preconditioner is defined as
$$\mathcal{P}_{\text{BFBt}}^{-1} =
  \begin{bmatrix}
    \mathsf{\widehat{A}} & \mathsf{C^T}          \\
    0                    & -\mathsf{\widehat{S}}
  \end{bmatrix}^{-1},$$where $\mathsf{\widehat{S}^{-1}\coloneqq(CC^T)^{-1}C{A}C^T(CC^T)^{-1}}.$
{\color{black}Similar to what was reported} in~\cite{ElmanBFBt}, the following set of experiments shows that convergence rates for the BFBt preconditioner are mesh-dependent, with iteration counts increasing in proportion to $h_{\Omega}^{-\frac{1}{2}}$. The action of $\mathsf{\widehat{A}^{-1}}$ is computed through a conjugate gradient solver, preconditioned by AMG, with inner tolerance set to $10^{-2}$.

When employing the rational-based preconditioner~\cite{haznics,trace1d2d}, its action on a vector is determined by the following inverse
$$\mathcal{P}_{\text{rational}}^{-1}\coloneqq
  \begin{bmatrix}
    \mathsf{\widehat{A}} & 0                    \\
    0                    & \mathsf{\widehat{S}}
  \end{bmatrix}^{-1}.$$
This approach is based on the observation that the Schur complement $\mathsf{S} = \mathsf{C A^{-1} C^T}$, in the context of matching interfaces, is spectrally equivalent to the
fractional Laplacian~\cite{trace1d2d}. The action of the {\color{black}inverse of} $\mathsf{\widehat{S}}$ on a vector $\mathsf{v}$ is computed
using a rational approximation and reads
$$\mathsf{\widehat{S}}^{-1}\mathsf{v} = c_0 \mathsf{M_\lambda^{-1}}\mathsf{v} + \sum_{i=1}^{N_p}c_i \bigl(\mathsf{A_{\lambda}} - \rho p_i \mathsf{M_\lambda}\bigr)^{-1}\mathsf{v},$$ where $c_i$ and $p_i$ are residues and poles, respectively, $\mathsf{A_\lambda}$ is the stiffness matrix on the immersed space $\Lambda_h$, and $\rho$ denotes the spectral
radius of $\mathsf{M_\lambda^{-1}A_\lambda}$. From the implementation standpoint, $\{p_i\}_{i=1}^{N_p}$ and $\{c_i\}_{i=0}^{N_p}$ are computed offline once and for all. In the present case, $N_p=20$. We also need an upper bound on $\rho(\mathsf{M_\lambda^{-1}A_\lambda})$, which we estimate as in~\cite{haznics} with
\begin{equation*}
  \rho(\mathsf{M_\lambda^{-1}A_\lambda}) \leq d(d+1) \|\text{diag}(\mathsf{M_\lambda})^{-1}\|_{\infty} \|\mathsf{A}\|_{\infty}.
\end{equation*}Moreover, as $p_i \in \mathbb{R}$, $p_i \leq 0$, the application of $\mathsf{\widehat{S}}^{-1}$ involves a series of $N_p$ elliptic
problems for which AMG is well-suited. The solution of every shifted linear system is realized using the CG method
with AMG as a preconditioner and a strict tolerance of $10^{-14}$. {\color{black}We observed that employing FGMRES as outer solver, with looser inner tolerances for these sub-systems, leads to high iteration counts in this case.} The inversion of $\mathsf{\widehat{A}^{-1}}$ in the $(1,1)$-block is done through the sparse direct solver~\textsc{UMFPACK}. Since the preconditioner is {\color{black}SPD}, the outer solver
associated with the rational preconditioner is chosen to be MINRES~\cite{PaigeSaunders}.

Concerning the AL preconditioner, we set $\gamma=10$ and a low inner tolerance of $10^{-2}$ for the augmented block $\mathsf{\widehat{A}_{\gamma}}$ which is always inverted through a
CG solver preconditioned by a single V-cycle of AMG. The inversion of $\mathsf{W}$ is performed through the sparse direct solver~\textsc{UMFPACK}. Moreover, the
absolute tolerance employed in these experiments is $\mathrm{TOL}\! =\!10^{-10}$. We notice that we have chosen a rather strict tolerance compared to other papers analyzing preconditioners, which
leads to higher iteration numbers. We have chosen this convergence criterion because it makes seeing trends easier due to the higher number of outer iterations.
When using AMG, we adopt the \text{TrilinosML} implementation~\cite{Trilinos}, {\color{black}with parameters shown in Table~\ref{tab:amg_params}, in Appendix~\ref{sec:app:AMG_params}.}

We fix $\Omega = [0, 1]^2$ as the background domain and consider {\color{black}various} immersed domains $\{\Gamma^{i}\}_{i=1,2,3}$ of different shapes originating from an unfitted discretization of different curves:
{\begin{itemize}
  \item $\Gamma^1 \coloneqq \partial B_r(\boldsymbol{c})$,
  \item $\Gamma^2 \coloneqq
          \Bigl\{
          \bigl(
          R + x_c + r \cos(\theta \pi x)\cos(2\pi x),\,
          R + y_c + r \cos(\theta \pi x)\sin(2\pi x)
          \bigr)
          \,\Big|\,
          x \in [0,1)
          \Bigr\}.$
  \item $\Gamma^3 \coloneqq \left\{ (x,y) \in \mathbb{R}^2: \min \left( x - a, b - x, y - a, b - y \right) = 0 \;\middle|\; (x,y) \in [0,1]^2 \right\}$,
\end{itemize}}
where the second curve is a { parametrization of a} flower-like interface and the last one is the boundary of the square $[a,b]^2$, assuming $a >0$ and $a<b<1$. In all the numerical results reported hereafter, we set $f=1$ and $g=1$ as data. Note that we have also tested various other data and observed the same trends and behavior for the three preconditioners considered. We point out that, in all experiments, the number of outer MINRES iterations
needed by the rational preconditioner with lower tolerances (such as $\mathrm{TOL}\! =\!10^{-6}$) is {\color{black}significantly} lower and agrees with the findings already reported in~\cite{haznics,trace1d2d}. The correctness of
our implementation can be verified in Table~\ref{tab:poisson_smooth_table} and Figure~\ref{fig:poisson_smooth_rates}, where optimal convergence rates against the mesh size $h_{\Omega}$ are reported for
the manufactured solution $u(x,y)\coloneqq \sin(2 \pi x)\sin(2 \pi y)$, which in turn implies $f\coloneqq8 \pi u^2(x,y)$. Notice that from the smoothness of $u$, it follows
that in this case the Lagrange multiplier {\color{black} is identically zero}. This corresponds to a special case where the interface is not truly an interface. Indeed, when an arbitrary value for $u$ is imposed on $\Gamma$, it cannot be expected that the solution is in $H^2(\Omega)$ since its gradient is not a continuous function across the interface. In such cases, it is known that non-matching methods are not able to recover the optimal rate of convergence~\cite{Heltai2019}. The chosen value for the absolute
tolerance $\mathrm{TOL}$ is comparable, if not stricter, to the possible accuracy obtained in the smooth case by the method.

\begin{figure}[H]
  \centering
  \subfloat[\centering]{{\includegraphics[width=7cm]{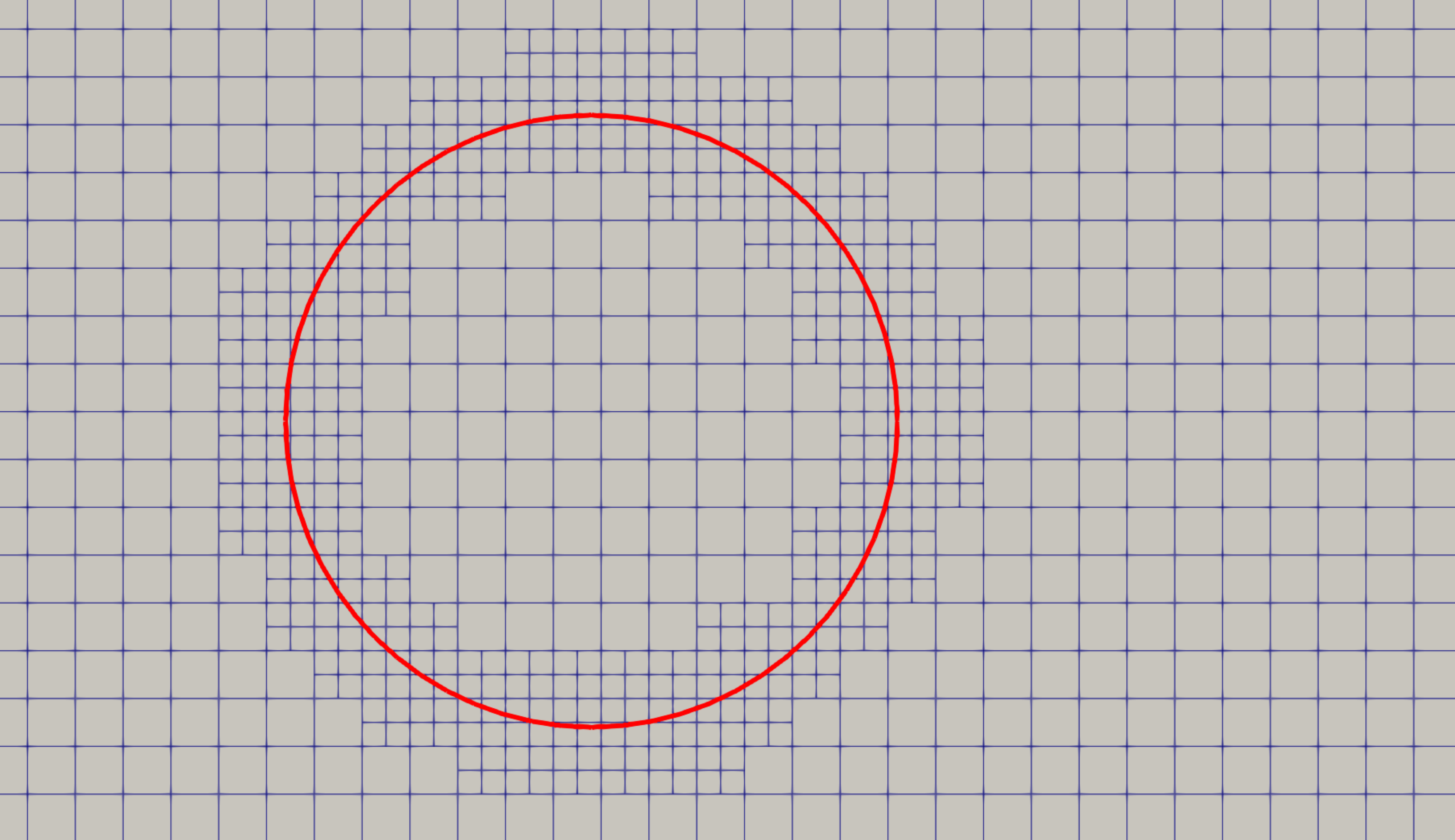}}}%
  \qquad
  \subfloat[\centering]{{\includegraphics[width=7cm]{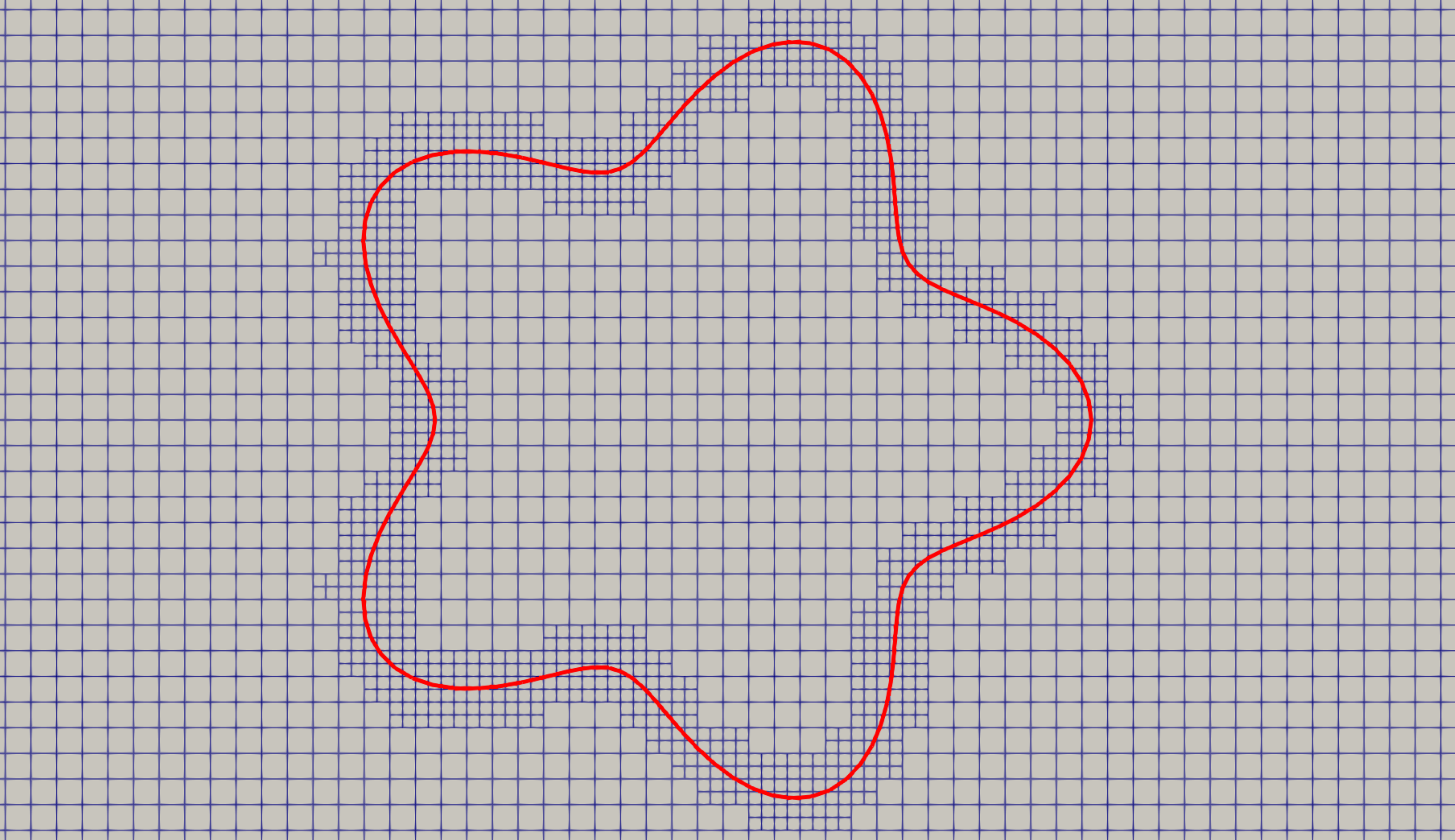} }}%
  \caption{Zoom on pre-processed background grid $\Omega_h$ for the circle interface $\Gamma = \Gamma^1$ (left) and the flower-shaped interface $\Gamma = \Gamma^2$ (right).}%
  \label{fig:interfaces}
\end{figure}

\begin{table}[H]
  \centering
  \begin{minipage}[t]{0.48\textwidth} %
    \centering
    \vspace{0pt} %
    \begin{tikzpicture}
      \begin{axis}[
          xlabel={$h_{\Omega}$},
          ylabel={Errors},
          grid=major,
          width=\textwidth, height=7cm,
          xmode=log,
          ymode=log,
          log basis x=10,
          ymax = 10, ymin = 0.000000001,
          legend pos=south east,
          legend style={font=\footnotesize}
        ]
        \addplot[blue, thick, mark=o]
        table[x expr=sqrt(1/\thisrowno{3}), y index=1, col sep=comma]  {data/smooth_solution-comma.csv};
        \addlegendentry{$L^2(\Omega)$ error}

        \addplot[red, thick, mark=square]
        table[x expr=sqrt(1/\thisrowno{3}), y index=2, col sep=comma]  {data/smooth_solution-comma.csv};
        \addlegendentry{$H^1(\Omega)$ error}

        \addplot[red, dashed, domain=0.0005:0.08] {x};
        \addlegendentry{$\mathcal{O}(h_\Omega)$}

        \addplot[blue, dashed, every mark/.append style={solid}, mark=+, domain=0.0005:0.08] {x^2};
        \addlegendentry{$\mathcal{O}(h_\Omega^2)$}

      \end{axis}
    \end{tikzpicture}
    
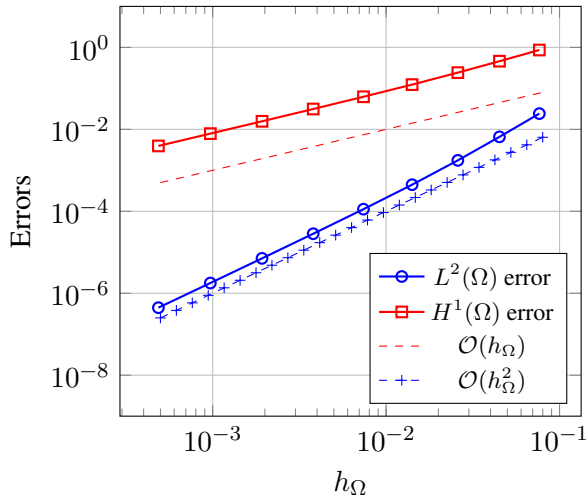
\captionof{figure}{$L^2(\Omega)$ and $H^1(\Omega)$ errors against $h_\Omega$ for interface $\Gamma^1$.}
    \label{fig:poisson_smooth_rates}
  \end{minipage}
  \hfill
  \begin{minipage}[t]{0.48\textwidth} %
    \centering
    \vspace{0pt} %
    \adjustbox{valign=t}{ %
      \begin{filecontents*}{data/smooth_solution.csv}
#DoF            #L2error  #H1error   #homega
171 + 17	2,4158e-02	8,6193e-01	171
493 + 33	6,5512e-03	4,5953e-01	493
1,489 + 65	1,7482e-03	2,4187e-01	1489
5,013 + 129	4,4547e-04	1,2336e-01	5013
18,237 + 257	1,1254e-04	6,2321e-02	18237
69,259 + 513	2,8308e-05	3,1324e-02	69259
269,563 + 1,025	7,0960e-06	1,5702e-02	269563
1,063,467 + 2,049	1,7755e-06	7,8602e-03	1063467
4,223,931 + 4,097	4,4425e-07	3,9325e-03	4223931
\end{filecontents*}
\pgfplotstabletypeset[
        col sep=tab,
        string type,
        header=false,
        columns={0,1,2},
        columns/0/.style={column name={DoF $\bigl(|V_h|+|\Lambda_h|\bigr)$}, column type=l||},
        columns/1/.style={column name=$L^2(\Omega)$ rate, column type=c},
        columns/2/.style={column name=$H^1(\Omega)$ rate, column type=c},
        every head row/.style={
            before row=\toprule
            \multicolumn{3}{c}{\textbf{Convergence history for $\Gamma^1$}} \\ \midrule,
            after row=\midrule
          },
        every last row/.style={after row=\bottomrule}
      ] {data/smooth_solution.csv}
    }
    \caption{$L^2(\Omega)$ and $H^1(\Omega)$ errors for $\Gamma^1$ across mesh refinement for the manufactured solution $u(x,y)\coloneqq \sin(2 \pi x) \sin(2 \pi y)$.}
    \label{tab:poisson_smooth_table}
  \end{minipage}
\end{table}

Tables~\ref{tab:poisson_circle}, ~\ref{tab:poisson_flower} and ~\ref{tab:poisson_square} present the outer iteration counts upon mesh-refinement for the Poisson problem~\eqref{eqn:algebraic_form} using
the three preconditioners outlined above. As interfaces we consider  the circular interface $\Gamma^1$, centered at $\boldsymbol{c} = (0.4, 0.4)$ with radius $r=0.2$, the flower-shaped interface $\Gamma^2$, where we have chosen $R=0.2$, $r=0.04$, $x_c=0.5$, $y_c=0.5$, $\theta=10$, and the square interface $\Gamma^3$ defined by parameters $a=0.25$ and $b=0.5$. The results demonstrate that the AL preconditioner consistently requires fewer
outer iterations compared to BFBt and rational approaches, highlighting its robustness and efficiency. As the number of DoF increases, the
iteration counts for the BFBt preconditioner show a noticeable increase, confirming the dependency on the mesh size. In contrast, the rational and AL preconditioners
exhibit robust iteration counts, with the AL preconditioner showing lower numbers. This can be better appreciated in Figures~\ref{fig:poisson_circle_it_plot},~\ref{fig:poisson_flower_it_plot} and~\ref{fig:poisson_square_it_plot}, where the iteration counts are plotted against
the mesh size of the background $h_{\Omega}$. For what concerns the inner solves for $\mathsf{\widehat{A}_{\gamma}}$ we observed a roughly constant and low number of
iteration counts across every refinement cycle, never exceeding {\color{black}nine} CG iterations for each outer iteration of FGMRES(30).
The low number of outer iterations required by the AL-based preconditioner (combined with the low number of inner ones) yields overall good solution times.
  {For the sequence of refinements performed in Table~\ref{tab:poisson_circle}, we measured
    the wall-clock time spent in the solver\footnote{On a 2.60GHz Intel Xeon processor, using a Release build with optimization flags enabled.} by all the preconditioners presented in this Section. The times are reported
    in Figure~\ref{fig:poisson_circle_times}. Compared with BFBt and the rational preconditioner, the AL-based approach exhibits lower times across all the refinement cycles. Similar trends were observed for all other interfaces, so additional figures and tables are omitted. Albeit a systematic comparison of the
    solvers is out of the scope of this paper, we argue that the observed times are representative.
  }

\begin{table}[H]
  \centering
  \begin{minipage}{0.48\textwidth} %
    \centering
    \begin{tikzpicture}
      \begin{axis}[
          xlabel={$h_{\Omega}$},
          ylabel={Iteration counts},
          grid=major,
          width=\textwidth, height=7cm,
          xmode=log,
          ymode=linear,
          log basis x=10,
          ymax = 110
        ]
        \addplot[blue, thick, mark=o]
        table[x expr=sqrt(1/\thisrowno{4}), y index=1, col sep=comma]  {data/laplace/iterations/circle/data.csv};
        \addlegendentry{BFBt}
        \addplot[green, thick, mark=square]
        table[x expr=sqrt(1/\thisrowno{4}), y index=2, col sep=comma]  {data/laplace/iterations/circle/data.csv};
        \addlegendentry{Rational}
        \addplot[magenta, thick, mark=triangle]
        table[x expr=sqrt(1/\thisrowno{4}), y index=3, col sep=comma]  {data/laplace/iterations/circle/data.csv};
        \addlegendentry{AL}

      \end{axis}
    \end{tikzpicture}
    
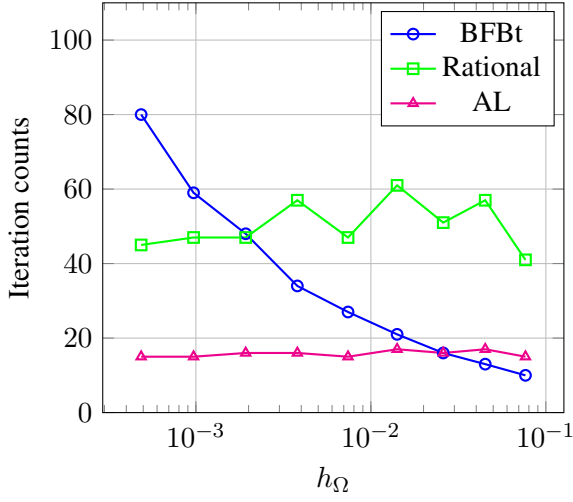
\captionof{figure}{Iteration counts vs. $h_{\Omega}$ for  $\Gamma^1$.}
    \label{fig:poisson_circle_it_plot}
  \end{minipage}
  \hfill
  \begin{minipage}{0.48\textwidth} %
    \centering
    \begin{filecontents*}{data/laplace/iterations/circle/data.csv}
171 + 17       , 10, 41 , 15, 171
493 + 33       , 13, 57 , 17, 493
{1,489 + 65}      , 16, 51 , 16, 1489
{5,013 + 129}     , 21, 61 , 17, 5013
{18,237 + 257 }   , 27, 47 , 15, 18237
{69,259 + 513}    , 34, 57 , 16, 69259
{269,563 + 1,025}  , 48, 47 , 16, 269563
{1,063,467 + 2,049} , 59, 47 , 15, 1063467
{4,223,931 + 4,097} , 80, 45 , 15, 4223931
\end{filecontents*}
\pgfplotstabletypeset[
      col sep=comma,
      string type,
      header=false,
      columns={0,1,2,3},
      columns/0/.style={column name={DoF $\bigl(|V_h|+|\Lambda_h|\bigr)$}, column type=l||},
      columns/1/.style={column name=BFBt, column type=c},
      columns/2/.style={column name=Rational, column type=c},
      columns/3/.style={column name=AL, column type=c},
      every head row/.style={
          before row=\toprule
          \multicolumn{4}{c}{\textbf{Outer iteration counts for $\Gamma^1$}} \\ \midrule,
          after row=\midrule
        },
      every last row/.style={after row=\bottomrule}
    ]{data/laplace/iterations/circle/data.csv}
    \caption{Outer iteration counts for the Poisson problem with circular interface $\Gamma^1$.}
    \label{tab:poisson_circle}
  \end{minipage}
\end{table}

\begin{table}[H]
  \centering
  \begin{minipage}{0.48\textwidth} %
    \centering
    \begin{tikzpicture}
      \begin{axis}[
          xlabel={$h_{\Omega}$},
          ylabel={Iteration counts},
          grid=major,
          width=\textwidth, height=7cm,
          xmode=log,
          ymode=linear,
          log basis x=10,
          ymax = 130
        ]
        \addplot[blue, thick, mark=o]
        table[x expr=sqrt(1/\thisrowno{4}), y index=1, col sep=comma]  {data/laplace/iterations/flower/data.csv};
        \addlegendentry{BFBt}
        \addplot[green, thick, mark=square]
        table[x expr=sqrt(1/\thisrowno{4}), y index=2, col sep=comma]  {data/laplace/iterations/flower/data.csv};
        \addlegendentry{Rational}
        \addplot[magenta, thick, mark=triangle]
        table[x expr=sqrt(1/\thisrowno{4}), y index=3, col sep=comma]  {data/laplace/iterations/flower/data.csv};
        \addlegendentry{AL}

      \end{axis}
    \end{tikzpicture}
    
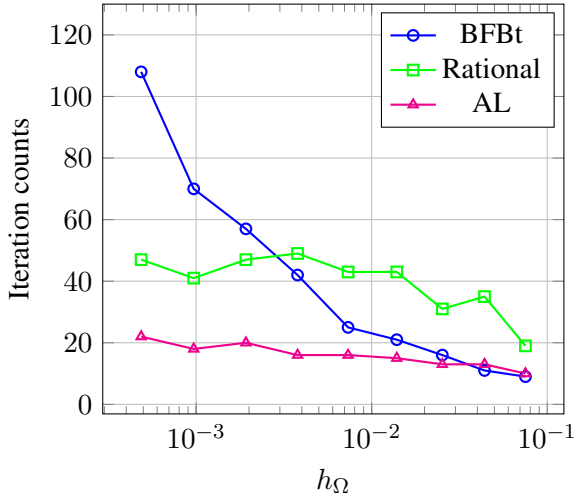
\captionof{figure}{Iteration counts vs. $h_{\Omega}$ for $\Gamma^2$.}
    \label{fig:poisson_flower_it_plot}
  \end{minipage}
  \hfill
  \begin{minipage}{0.48\textwidth} %
    \centering
    \begin{filecontents*}{data/laplace/iterations/flower/data.csv}
177 + 17       , 9, 19 , 10, 177
518 + 33       , 11, 35 , 13, 518
{1,567 + 65}      , 16, 31 , 13, 1567
{5,177 + 129}     , 21, 43 , 15, 5177
{18,575 + 257}    , 25, 43 , 16, 18575
{69,869 + 513}    , 42, 49 , 16, 69869
{270,845 + 1,025}  , 57, 47 , 20, 270845
{1,065,947 + 2,049} , 70, 41 , 18, 1065947
{4,229,139 + 4,097} ,  108 , 47 , 22, 4229139
\end{filecontents*}
\pgfplotstabletypeset[
      col sep=comma,
      string type,
      header=false,
      columns={0,1,2,3},
      columns/0/.style={column name={DoF $\bigl(|V_h|+|\Lambda_h|\bigr)$}, column type=l||},
      columns/1/.style={column name=BFBt, column type=c},
      columns/2/.style={column name=Rational, column type=c},
      columns/3/.style={column name=AL, column type=c},
      every head row/.style={
          before row=\toprule
          \multicolumn{4}{c}{\textbf{Outer iteration counts for $\Gamma^2$}} \\ \midrule,
          after row=\midrule
        },
      every last row/.style={after row=\bottomrule}
    ]{data/laplace/iterations/flower/data.csv}
    \caption{Outer iteration counts for the Poisson problem with flower-shaped interface $\Gamma^2$.}
    \label{tab:poisson_flower}
  \end{minipage}
\end{table}

\begin{table}[H]
  \centering
  \begin{minipage}{0.48\textwidth}
    \centering
    \begin{tikzpicture}
      \begin{axis}[
          xlabel={$h_{\Omega}$},
          ylabel={Iteration counts},
          grid=major,
          width=\textwidth, height=7cm,
          xmode=log,
          ymode=linear,
          log basis x=10,
          ymax = 100
        ]
        \addplot[blue, thick, mark=o]
        table[x expr=sqrt(1/\thisrowno{4}), y index=1, col sep=comma]  {data/laplace/iterations/square/data.csv};
        \addlegendentry{BFBt}
        \addplot[green, thick, mark=square]
        table[x expr=sqrt(1/\thisrowno{4}), y index=2, col sep=comma]  {data/laplace/iterations/square/data.csv};
        \addlegendentry{Rational}
        \addplot[magenta, thick, mark=triangle]
        table[x expr=sqrt(1/\thisrowno{4}), y index=3, col sep=comma]  {data/laplace/iterations/square/data.csv};
        \addlegendentry{AL}

      \end{axis}
    \end{tikzpicture}
    
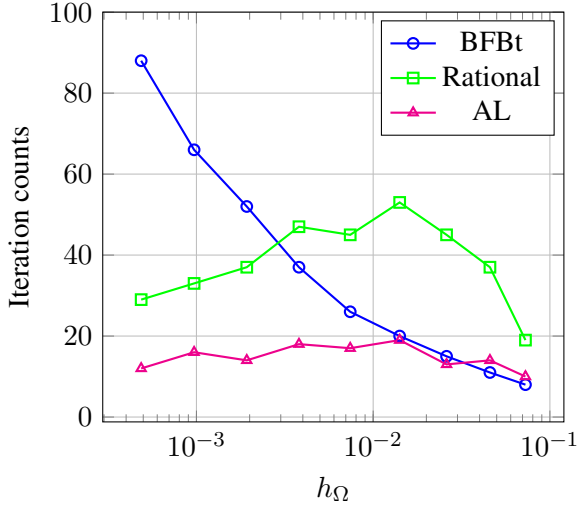
\captionof{figure}{Iteration counts vs. $h_{\Omega}$ for $\Gamma^3$.}
    \label{fig:poisson_square_it_plot}
  \end{minipage}
  \hfill
  \begin{minipage}{0.48\textwidth} %
    \centering
    \begin{filecontents*}{data/laplace/iterations/square/data.csv}
189 + 17       , 8,  19 , 10, 189
477 + 33       , 11, 37 , 14, 477
{1,477 + 65}      , 15, 45 , 13, 1477
{5,013 + 129}     , 20, 53 , 19, 5013
{18,189 + 257}    , 26, 45 , 17, 18189
{69,117 + 513}    , 37, 47 , 18, 69117
{269,317 + 1,025}  , 52, 37 , 14, 269317
{1,062,933 + 2,049} , 66,  33 , 16, 1062933
{4,222,989 + 4,097} , 88,  29 , 12, 4222989
\end{filecontents*}
\pgfplotstabletypeset[
      col sep=comma,
      string type,
      header=false,
      columns={0,1,2,3},
      columns/0/.style={column name={DoF $\bigl(|V_h|+|\Lambda_h|\bigr)$}, column type=l||},
      columns/1/.style={column name=BFBt, column type=c},
      columns/2/.style={column name=Rational, column type=c},
      columns/3/.style={column name=AL, column type=c},
      every head row/.style={
          before row=\toprule
          \multicolumn{4}{c}{\textbf{Outer iteration counts for $\Gamma^3$}} \\ \midrule,
          after row=\midrule
        },
      every last row/.style={after row=\bottomrule}
    ]{data/laplace/iterations/square/data.csv}
    \caption{Outer iteration counts for the Poisson problem with square interface $\Gamma^3$.}
    \label{tab:poisson_square}
  \end{minipage}
\end{table}

{
\begin{figure}[H]
  \centering
  \begin{tikzpicture}
    \begin{axis}[
        width=12cm, height=6cm,
        xlabel={DoF $\bigl(|V_h|+|\Lambda_h|\bigr)$},
        ylabel={Wall-clock time [s]},
        legend pos=north west,
        grid=major,
        ymode=log,
        xmode=log,
        log basis x=10,
        log basis y=10,
        tick label style={font=\small},
        label style={font=\small},
        legend style={font=\small}
      ]
      \addplot[blue, thick, mark=o]
      table[x index=0, y index=1, col sep=comma] {data/laplace/iterations/circle/times.csv};
      \addlegendentry{BFBt}
      \addplot[green, thick, mark=square*]
      table[x index=0, y index=2, col sep=comma] {data/laplace/iterations/circle/times.csv};
      \addlegendentry{Rational}
      \addplot[magenta, thick, mark=triangle*]
      table[x index=0, y index=3, col sep=comma] {data/laplace/iterations/circle/times.csv};
      \addlegendentry{AL}
      \addplot[black, dashed, very thick, domain=1e3:.5*1e7, samples=2] {0.000003*x};
      \addlegendentry{$\mathcal{O}(\# \text{DoF})$}
    \end{axis}
  \end{tikzpicture}
  \caption{Wall-clock time (in seconds) for the Poisson problem with circular interface $\Gamma^1$ using different preconditioners, as a function of the total number of DoF.}
  \label{fig:poisson_circle_times}
\end{figure}
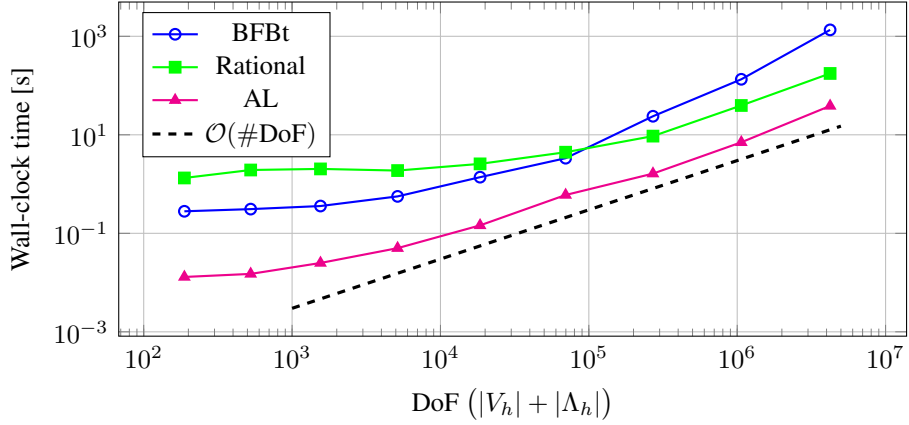}

\subsection{Stokes problem}\label{subsec:stokes}
We analyze the application to the Stokes problem of the AL preconditioner $\mathcal{P_{ \gamma \delta}}$ developed {\color{black}in} Section~\ref{sec:stokes_problem}. The background domain is fixed to be $\Omega=[0,1]^2$, and the following configuration
is considered:
\begin{equation*}
  \begin{aligned}
    \mathbf{f} \coloneqq
    [
      1,
      0
    ]^T, \quad
    \mathbf{g} \coloneqq
    [
      -0.5,
      0.5
    ]^T,\quad
    \Gamma^{4}\coloneqq\partial \mathcal{B}_{r}(\boldsymbol{c}),
  \end{aligned}
\end{equation*}where $\boldsymbol{c}=(0.45,0.45)$ and $r=0.21$. Notice that $\mathbf{g}$ has been chosen in
such a way that the compatibility condition~\eqref{eqn:constraint_incompressibility} is automatically satisfied. Concerning the AL parameters, we have set
$\gamma\! =\! \delta \! =\!10$. In this two-dimensional example, the inversion of the mass matrices $\mathsf{Q}$ and $\mathsf{W}$ is performed through the sparse direct solver~\textsc{UMFPACK}, as we
are still in a regime for which the factorization of the pressure mass matrix $\mathsf{Q}$ is affordable. Notice that $\mathsf{W}$ is a 1D mass matrix, so its explicit inversion
with a sparse direct solver does not constitute a possible bottleneck in this example, even when the cardinality of $V_h$ is large. The augmented term is never factorized, but is always inverted through a
CG solver preconditioned by a single V-cycle of AMG, with a loose tolerance of $10^{-2}$.
The iteration counts for the outer FGMRES(30) solver are reported in Table~\ref{tab:stokes_2D_direct}, corresponding to an
absolute tolerance of $\mathrm{TOL}\! =\!10^{-8}$. The AL preconditioner
shows a robust behavior, requiring low FGMRES(30) iteration {\color{black}counts} which are independent of the refinement levels.
  {We display the inner iterations
    as the average over all outer FGMRES(30) iterations of the number of iterations for the augmented velocity block $\mathsf{\widehat{A}_{\gamma}}$.} From there
it can be seen that the number of inner iterations slightly increases upon mesh refinement, but remains essentially bounded. {The development of a
    $\gamma,\delta$-robust and mesh-independent solver for our double-augmented block $\mathsf{\widehat{A}_{\gamma\delta}}$ is far from trivial and is left for future work. Ad-hoc geometric multigrid schemes have been used and developed in~\cite{ALprec,Farrell2019} in the context of Oseen and Stokes problems, respectively, to overcome the issues introduced by the augmentation.}

Having confirmed the good behavior of the preconditioner with direct solvers, we test its performance
when the action of $\mathsf{Q^{-1}}$ and $\mathsf{W^{-1}}$ is {\color{black}applied} inexactly. The inversion of $\mathsf{Q}$ is performed with a CG solver
preconditioned by the (lumped) pressure mass matrix ${\mathsf{\widehat{M}_p}}$, while for the matrix $\mathsf{W}$ we directly invert its diagonal approximation (see Eq.~\eqref{eqn:W_approx}) . We show in Table~\ref{tab:stokes_2D_iterative} the results with this approach. The usage of iterative
solvers also for the mass matrices does not spoil the overall quality of the preconditioner. We observe only a marginal increase in the number of outer iterations compared to
the case where exact inversions are performed.
However, the overall numbers are robust upon mesh refinement, with a considerable \color{black}saving in computational cost associated with the fact that no direct inversion of any block has been performed.
The computed velocity field $\mathbf{u}$ and its streamlines are shown in Figure~\ref{fig:stokes_IB_2D}. {Finally, we test the performance of the block triangular preconditioner $\mathcal{P}_{\gamma \delta}$ defined in~\eqref{eqn:prec_AL_stokes} against the diagonal and SPD variant $\mathcal{P_{\text{sym}, \gamma \delta}}$ presented in Remark~\ref{rmk:SPD_variant}, which is used as a preconditioner for MINRES. The computing times for
the solver are reported in Figure~\ref{fig:stokes_circle_times}, where it
can be appreciated how the block triangular preconditioner $\mathcal{P}_{\gamma \delta}$ outperforms the diagonal variant by approximately an order of magnitude.}

{
\begin{figure}[H]
  \centering
  \begin{tikzpicture}
    \begin{axis}[
        width=12cm, height=5.8cm,
        xlabel={DoF $\bigl(|V_h|+|Q_h|+|\Lambda_h|\bigr)$},
        ylabel={Wall-clock time [s]},
        legend pos=north west,
        grid=major,
        ymode=log,
        xmode=log,
        log basis x=10,
        log basis y=10,
        tick label style={font=\small},
        label style={font=\small},
        legend style={font=\small}
      ]
      \addplot[blue, thick, mark=o]
      table[x index=0, y index=1, col sep=comma] {data/stokes/iterations/times_circle.csv};
      \addlegendentry{$\mathcal{P_{\text{sym}, \gamma \delta}}$}
      \addplot[magenta, thick, mark=triangle*]
      table[x index=0, y index=2, col sep=comma] {data/stokes/iterations/times_circle.csv};
      \addlegendentry{$\mathcal{P_{\gamma \delta}}$}
      \addplot[black, dashed, very thick, domain=1e3:1e7, samples=2] {0.00002*x};
      \addlegendentry{$\mathcal{O}(\# \text{DoF})$}
    \end{axis}
  \end{tikzpicture}
  \caption{Wall-clock time (in seconds) for the Stokes problem with circular interface $\Gamma^4$ using the block-triangular AL preconditioner $\mathcal{P_{\gamma \delta}}$ and its SPD variant $\mathcal{P_{\text{sym}, \gamma \delta}}$, as a function of the total number of DoF.}
  \label{fig:stokes_circle_times}
\end{figure}
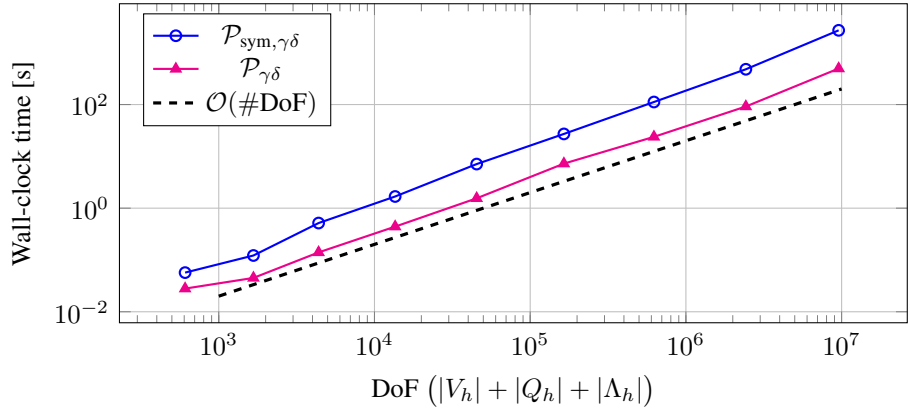}

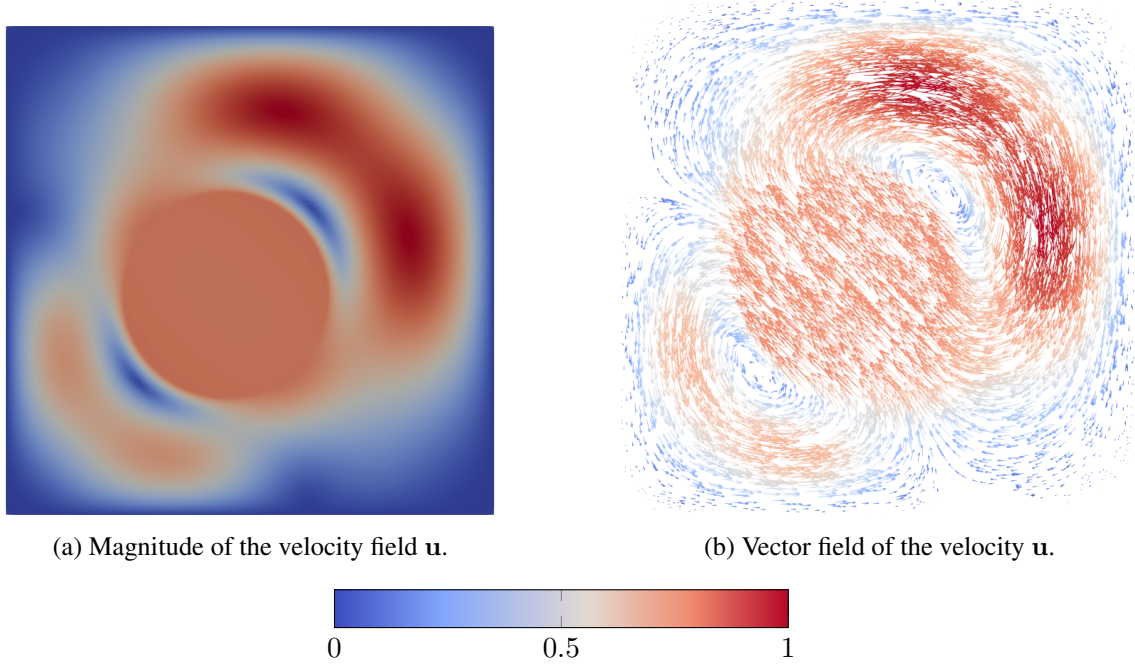
\begin{figure}[H]%
  \centering
  \begin{subfigure}[b]{0.48\textwidth}
    \centering
    \includegraphics[width=0.85\textwidth]{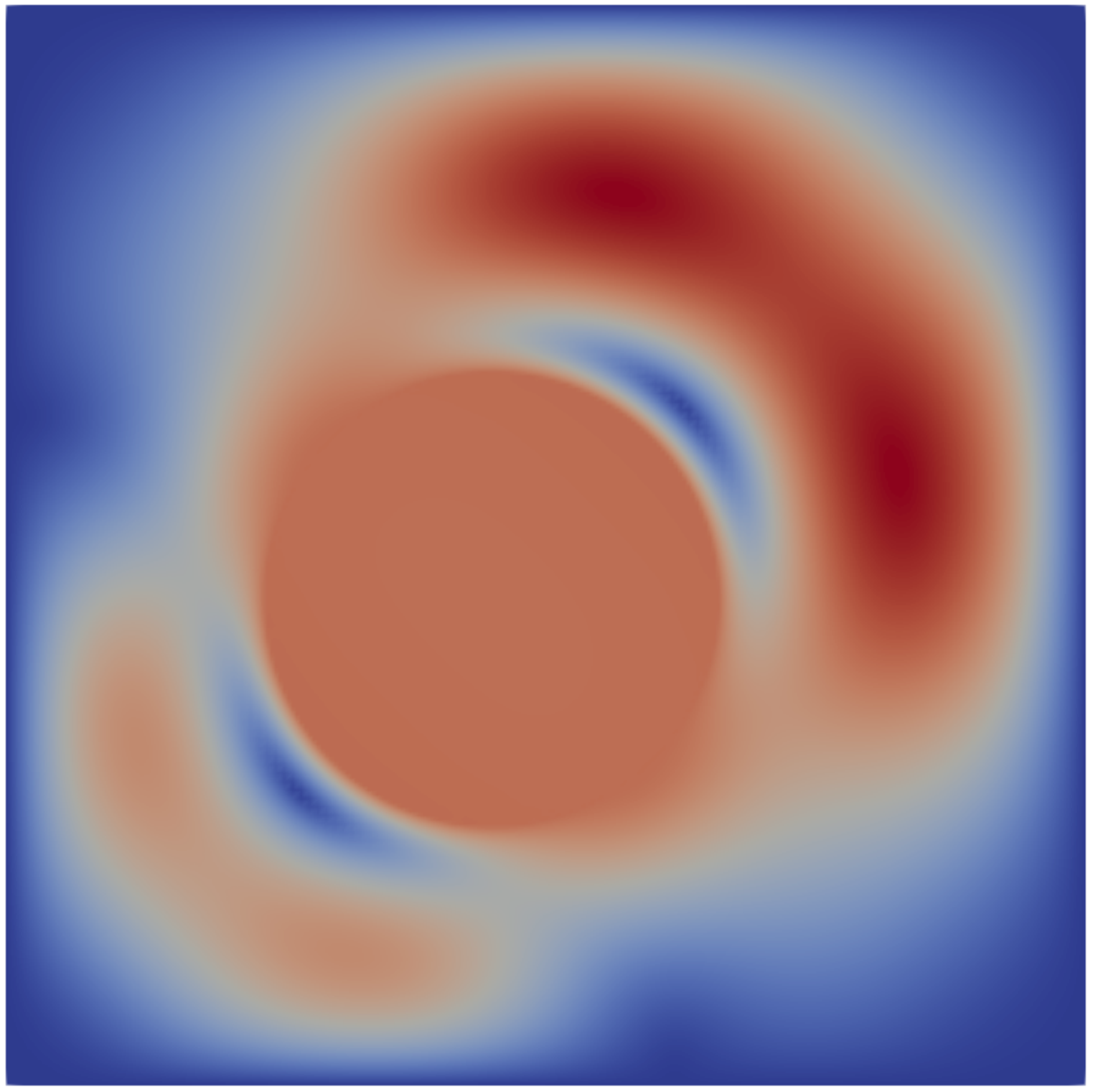}
    \caption{Magnitude of the velocity field $\mathbf{u}$.}
  \end{subfigure}
  \hfill
  \begin{subfigure}[b]{0.48\textwidth}
    \centering
    \includegraphics[width=0.9\textwidth]{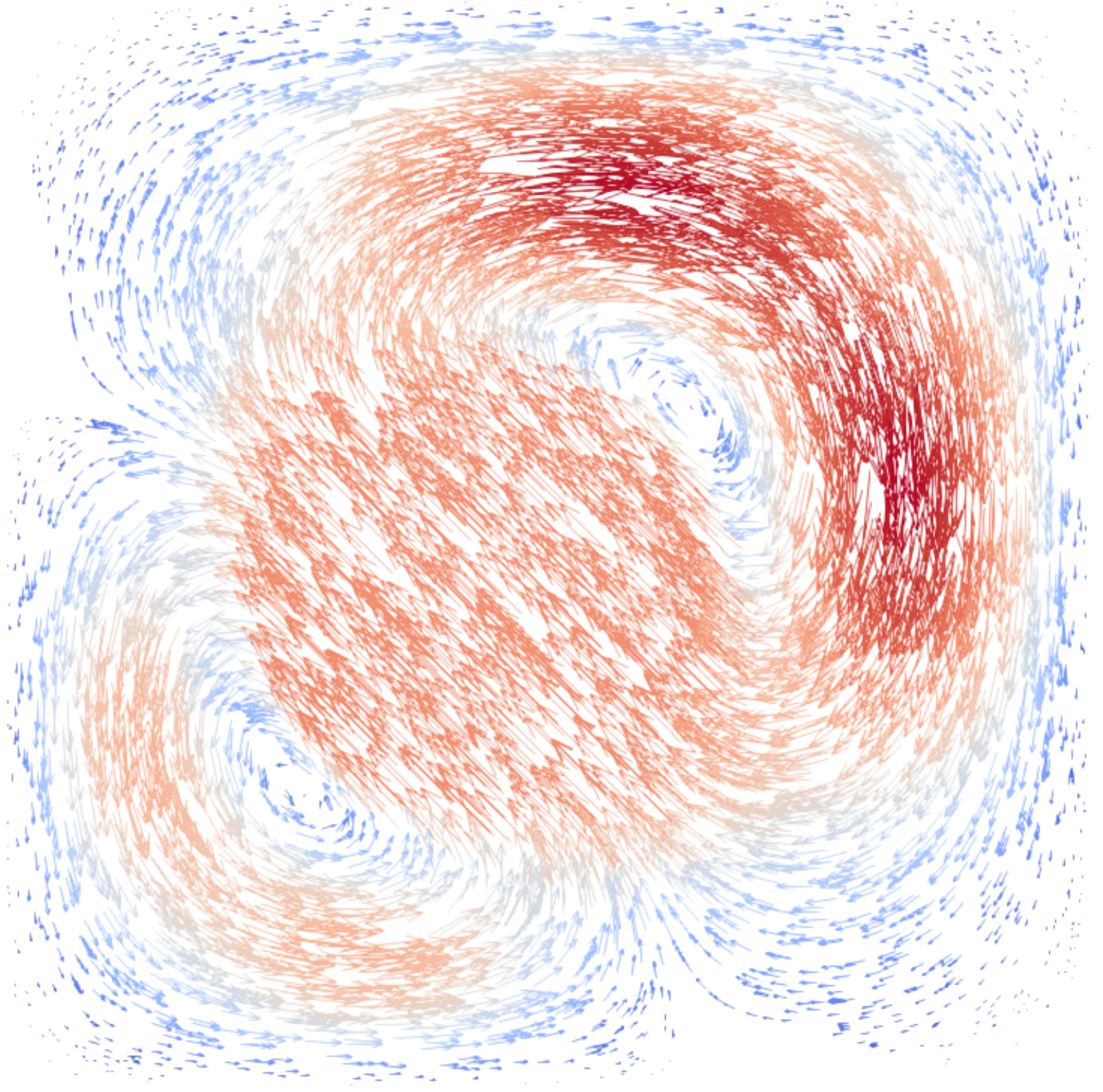}
    \caption{Vector field of the velocity $\mathbf{u}$.}
  \end{subfigure}
  \vspace{0.5cm} %
  \centering
  \begin{tikzpicture}
    \begin{axis}[
        hide axis,
        scale only axis,
        height=0pt,
        width=0pt,
        colormap={Diverging}{%
            rgb(0cm)=(0.231373, 0.298039, 0.752941);   %
            rgb(0.214686cm)=(0.513725,0.654902,0.988235); %
            rgb(0.502917cm)=(0.894118,0.85098,0.819608); %
            rgb(0.706993cm)=(0.941176,0.54902,0.435294); %
            rgb(0.908cm)=(0.705882,0.0156863, 0.14902)    %
          },
        colorbar horizontal,
        point meta min=0,
        point meta max=1,
        colorbar style={
            width=6cm,
            xtick={0,0.5,1.}
          }]
      \addplot [draw=none] coordinates {(0,1)}; %
    \end{axis}
  \end{tikzpicture}
  \vspace{-0.8cm} %
  \caption{Computed vector field $\mathbf{u}$ for the Stokes problem with an embedded circular interface.}
  \label{fig:stokes_IB_2D}
\end{figure}

\FloatBarrier
\subsubsection{Three-dimensional tests}
In order to show the reliability of our preconditioner also for practical usage, we finally consider the three-dimensional extension of the tests in the previous paragraph. The
following embedded surfaces are tested:
\begin{itemize}
  \item $\Gamma^{5} \coloneqq \partial B_r(\boldsymbol{c})$,
  \item $\Gamma^{6} \coloneqq \Bigl\{ (x,y,z) \in \mathbb{R}^3: \left(\sqrt{x^2 + y^2} - 0.3\right)^2 + z^2 - (0.1)^2 = 0\Bigr\},$
\end{itemize}where $\boldsymbol{c}=(\frac{1}{2},\frac{1}{2},\frac{1}{2})$ and $r=0.1$. The interface $\Gamma^6$ describes
a torus with inner and outer radius of {$0.2$ and $0.4$}. As data of the problem, we consider
\begin{equation*}
  \begin{aligned}
    \mathbf{f} \coloneqq
    [
      1,
      0,
      0
    ]^T, \quad
    \mathbf{g} \coloneqq
    [
      -1,
      1,
      0
    ]^T.
  \end{aligned}
\end{equation*}A sample grid illustrating the (local) refinement process near the embedded toroidal surface mesh $\Gamma^6$ is shown in Figure~\ref{fig:torus_refined}, while Figure~\ref{fig:tubes_torus} displays
the streamlines for the computed velocity field $\mathbf{u}$.
The AL parameters are set to $\gamma\! =\! \delta \! =\! 10$. The $(1,1)$-block is solved as in the previous tests with a tolerance of $10^{-2}$. The mass matrices $\mathsf{Q}$ and $\mathsf{W}$ are
inverted iteratively in the same way as explained in the previous two-dimensional test.
The results of this study are reported in Tables~\ref{tab:stokes_3D_sphere} and~\ref{tab:stokes_3D_torus}. For both the interfaces, we observe
outer iteration counts independent of the refinement level. Moreover, the order of magnitude of the inner AMG iterations for the inversion of the $(1,1)$-block is the same as in the two-dimensional tests. We solve problems
up to $56$ million of unknowns, thereby confirming the trends observed so far.
The three-dimensional simulations have been run varying the number of MPI processes employed from $16$ up to $256$ and observing robust inner and
outer iteration counts. Notably, the largest simulations in Tables~\ref{tab:stokes_3D_sphere} and~\ref{tab:stokes_3D_torus} must be run with a sufficiently large number of cores due to the high memory cost associated
with the setup of the AMG preconditioner for the augmented block and the storage of sparse matrices arising from higher-order discretizations. We show in Figure~\ref{fig:strong_scaling} a strong scaling experiment for different
components of our pipeline, where we keep a fixed problem size resulting from a three-dimensional discretization comprising roughly $7$ million DoF and increase incrementally the number of MPI processes from $16$ to $512$. In particular, we measure the
wall-clock time to assemble matrices $\mathsf{A}$ and $\mathsf{B}$ for the Stokes problem, the time to assemble the coupling matrix $\mathsf{C}$, and the time taken by FGMRES(30) to solve the final linear system of equations. The parallel experiments have been performed on the Galileo100\footnote{For the technical specifications, see \url{https://www.hpc.cineca.it/systems/hardware/galileo100/}, retrieved on February 21, 2025.} Italian supercomputer. Its compute nodes have two sockets (each one with 24 cores of Intel CascadeLake). Good scaling properties are confirmed for all the kernels employed in the simulation.
\begin{figure}[H]
  \centering
  \begin{subfigure}[b]{0.48\textwidth}
    \centering
    \includegraphics[width=1\textwidth]{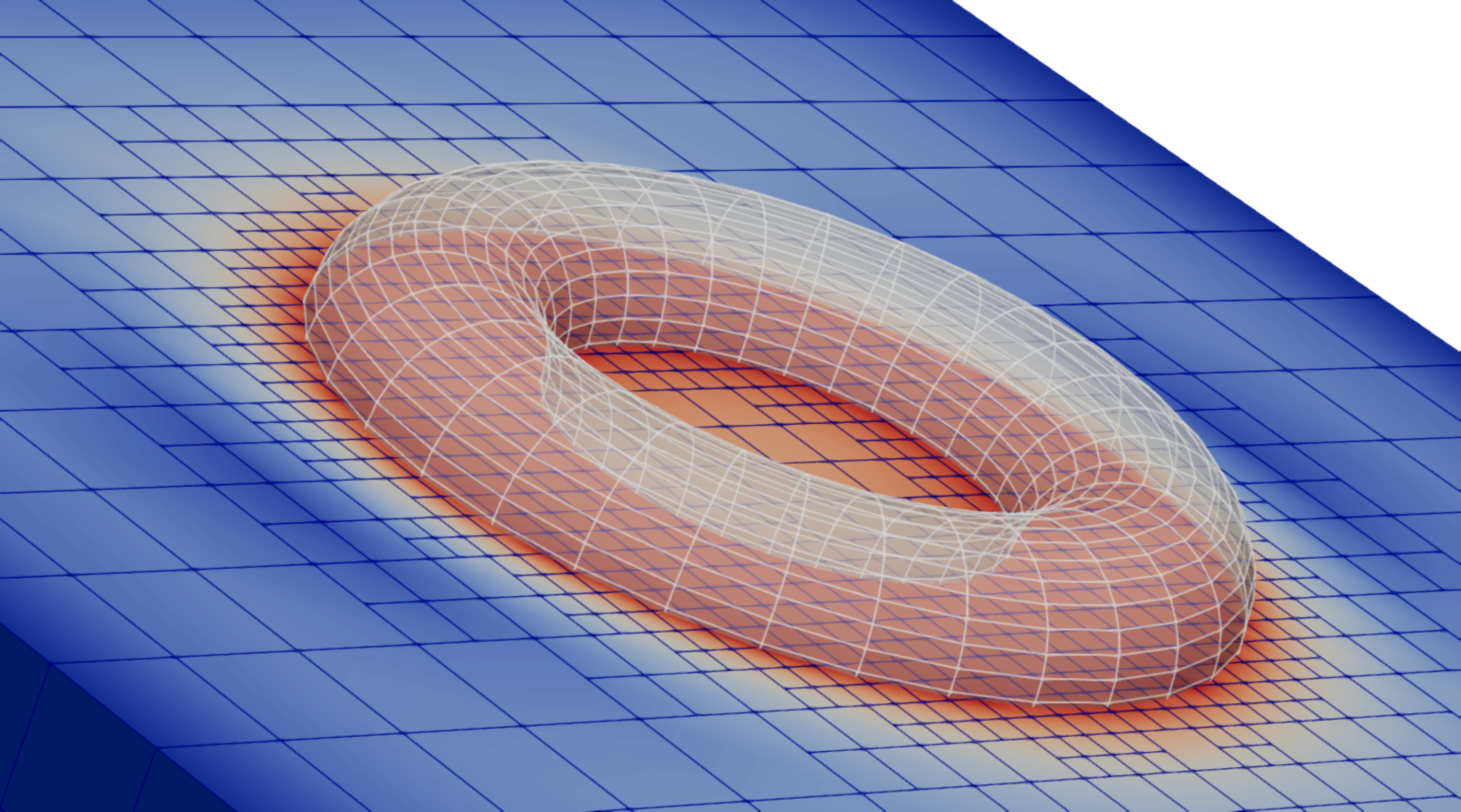}
    \caption{Locally refined background grid $\Omega_h$ around the toroidal interface $\Gamma^6$.}
    \label{fig:torus_refined}
  \end{subfigure}
  \hfill
  \begin{subfigure}[b]{0.46\textwidth}
    \centering
    \includegraphics[width=1.0\textwidth]{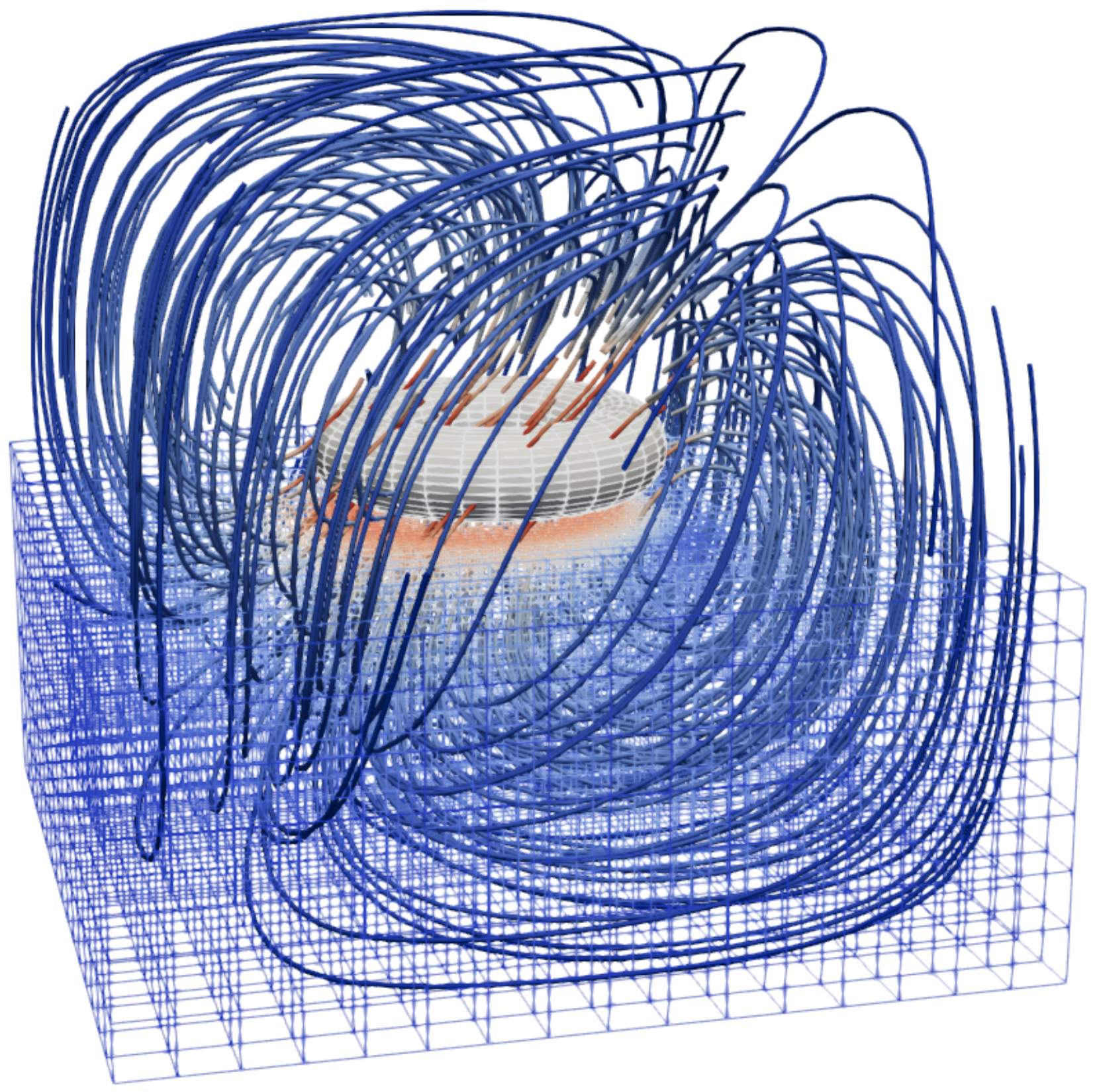}
    \caption{Streamlines associated to the velocity field $\mathbf{u}$ with interface $\Gamma^6$.}
    \label{fig:tubes_torus}
  \end{subfigure}
  \centering
  \begin{tikzpicture}
    \begin{axis}[
        hide axis,
        scale only axis,
        height=0pt,
        width=0pt,
        colormap={Diverging}{%
            rgb(0cm)=(0.231373, 0.298039, 0.752941);   %
            rgb(0.214686cm)=(0.513725,0.654902,0.988235); %
            rgb(0.502917cm)=(0.894118,0.85098,0.819608); %
            rgb(0.706993cm)=(0.941176,0.54902,0.435294); %
            rgb(1.6cm)=(0.705882,0.0156863, 0.14902)    %
          },
        colorbar horizontal,
        point meta min=0,
        point meta max=1.6,
        colorbar style={
            width=6cm,
            xtick={0,0.5,1.0,1.6}
          }]
      \addplot [draw=none] coordinates {(0,1.6)}; %
    \end{axis}
  \end{tikzpicture}
  \caption{Post-processing of the velocity field $\mathbf{u}$ with interface $\Gamma^6$. On the right, the background geometry
    is shown with a wireframe (and clipped) representation.}
\end{figure}
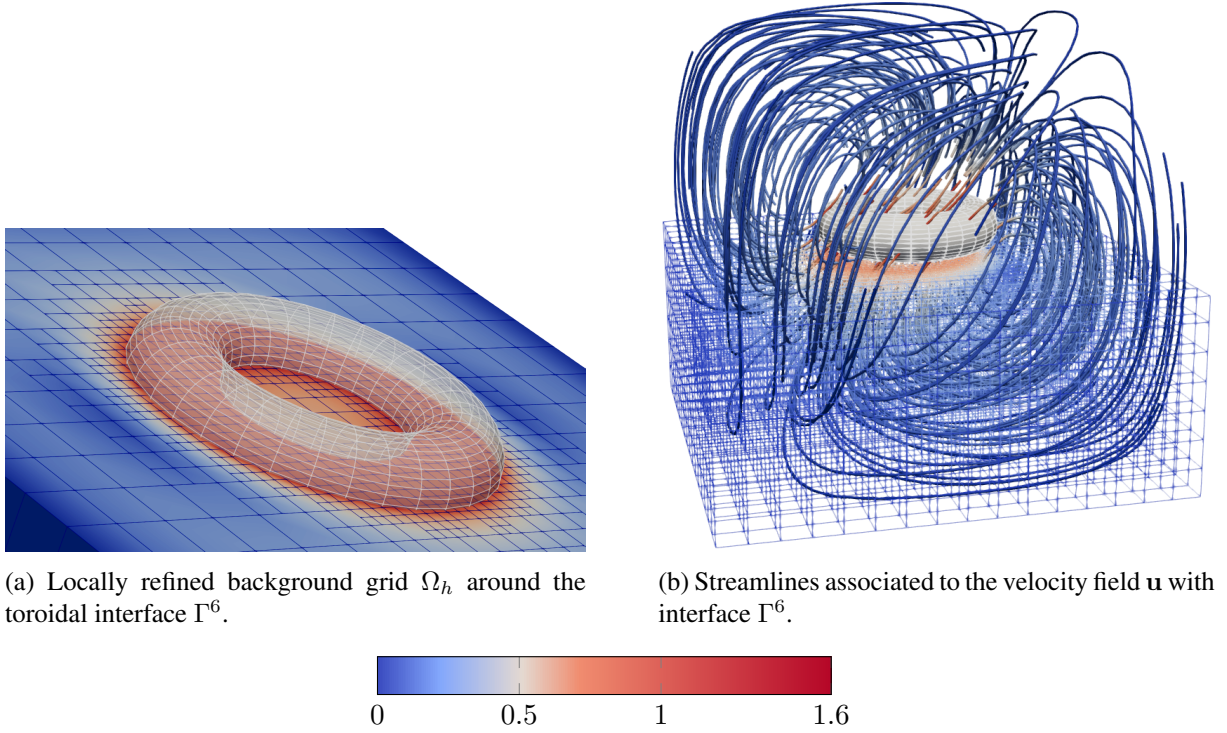

\begin{figure}[H]
  \centering
  \begin{tikzpicture}
    \begin{loglogaxis}[
        width=12cm, height=7cm,
        xlabel={Number of MPI ranks},
        ylabel={Wall-clock time [s]},
        title={Strong scaling experiment},
        grid=both,
        legend pos=north east,
        xtick=data,
        xticklabels={16, 32, 64, 96, 128, 192, 256, 512},
        ymajorgrids=true,
        xmajorgrids=true,
        minor grid style={dashed,gray!30},
        major grid style={dashed,gray!50},
        ymax = 2e3
      ]

      \addplot[color=blue, mark=o, thick] coordinates {
          (16,344.86) (32,195.145) (64,96.8739) (96,64.2033) (128,40.380)
          (192,26.33) (256,31.5158) (512,19.8516)
        };
      \addlegendentry{FGMRES(30)}

      \addplot[color=red, mark=x, thick] coordinates {
          (16,56.2) (32,26.1) (64,14) (96,8.89) (128,7)
          (192,4.53) (256,3.5) (512,1.76)
        };
      \addlegendentry{Assemble system}

      \addplot[color=green, mark=square*, thick] coordinates {
          (16,4.74) (32,2.09) (64,1.29) (96,1.14) (128,0.808)
          (192,0.759) (256,0.525) (512,0.361)
        };
      \addlegendentry{Assemble coupling}

      \addplot[color=black, thick, domain=16:512, samples=8]
      {.5*344.86 * (16/x)};
      \addlegendentry{Ideal Scaling}

    \end{loglogaxis}
  \end{tikzpicture}
  \caption{Strong scaling test of the principal components of the whole solver for a three-dimensional problem with interface $\Gamma^6$ with $7$ million DoF.}
  \label{fig:strong_scaling}
\end{figure}
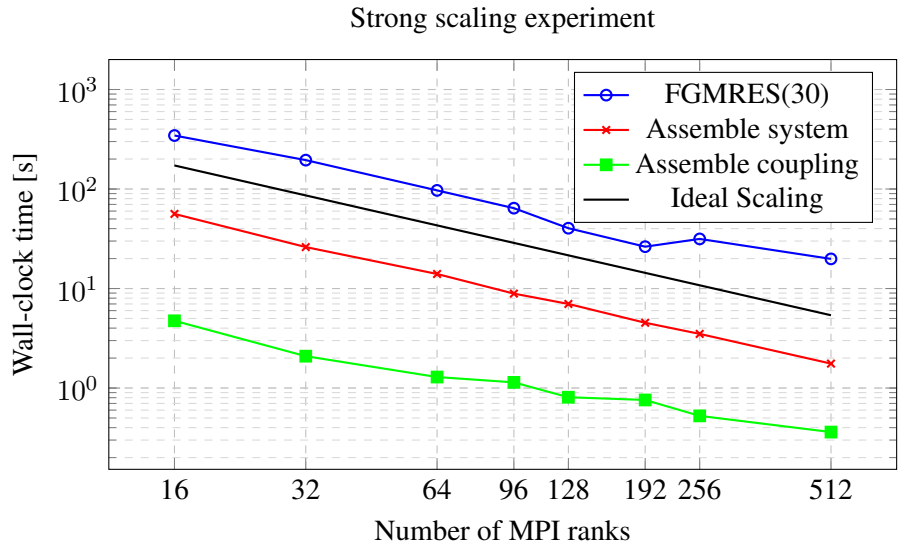

\begin{table}[h]
  \centering
  \begin{minipage}{0.48\linewidth}
    \centering
    \begin{filecontents*}{data/stokes/iterations/data.csv}
#DoF                              #outer    #inner  #outer(1e-6)  #inner(1e-6) #VDof       #outer(1e-8)  #inner(1e-8)     #Inner(avg)
518 + 72 + 18                      , 24    ,   168 ,   17        ,     102   ,    518     ,   17,           127,              7      
{1,446 + 190 + 34}                 , 24    ,   192 ,   16        ,     112   ,    1446    ,   16,           134,              8
{3,818 + 494 + 66}                 , 23    ,   276 ,   15        ,     180   ,    3818    ,   16,           263,              16
{11,922 + 1,523 + 130}             , 24    ,   336 ,   11        ,     154   ,    11922   ,   16,           286,              18
{40,002 + 5,065 + 258}             , 23    ,   322 ,   10        ,     150   ,    40002   ,   15,           316,              21
{145,442 + 18,309 + 514}           , 22    ,   396 ,   9         ,     162   ,    145442  ,   15,           361,              24
{2,154,546 + 269,831 + 2,050}      , 24    ,   432 ,   8         ,     168   ,    2154546 ,   10,           310,              31
{8,503,490 + 1,063,961 + 4,098}    , 23    ,   836 ,   7         ,     210   ,    8503490 ,   10,           423,              42
\end{filecontents*}
\pgfplotstabletypeset[
      col sep=comma,
      string type,
      header=false,
      columns={0,6,8},
      columns/0/.style={column name={DoF $\bigl(|V_h|+|Q_h|+|\Lambda_h|\bigr)$}, column type=l||},
      columns/6/.style={column name=Outer, column type=c},
      columns/8/.style={column name=Inner, column type=c},
      every head row/.style={
          before row=\toprule
          \multicolumn{3}{c}{\textbf{Iteration counts for $\Gamma^4$ (direct solvers)}} \\ \midrule,
          after row=\midrule
        },
      every last row/.style={after row=\bottomrule}
    ]{data/stokes/iterations/data.csv}
    \subcaption{Circular interface $\Gamma^4$ (2D test).}
    \label{tab:stokes_2D_direct}
  \end{minipage}%
  \hspace{0.04\linewidth}%
  \begin{minipage}{0.48\linewidth}
    \centering
    \begin{filecontents*}{data/stokes/iterations/data_iterative.csv}
#DoF                            #outer    #inner  #outer(1e-6)  #inner(1e-6)  #VDoF      #outer(1e-6)  #inner(1e-6)  #mean_inner(1e-6) 
518 + 72 + 18                   , 27    , 27    ,  17        ,     17         , 518     ,  18,           18,           2
{1,446 + 190 + 34}              , 30    , 30    ,  18        ,     18         , 1446    ,  20,           20,           2 
{3,818 + 494 + 66}              , 32    , 320   ,  18        ,     180        , 3818    ,  20,           271,          13
{11,922 + 1,523 + 130}          , 33    , 396   ,  13        ,     156        , 11922   ,  21,           322,          15
{40,002 + 5,065 + 258}          , 32    , 448   ,  11        ,     154        , 40002   ,  20,           362,          18
{145,442 + 18,309 + 514}        , 31    , 465   ,  9         ,     144        , 145442  ,  19,           387,          20
{2,154,546 + 269,831 + 2,050}   , 31    , 693   ,  8         ,     184        , 2154546 ,  12,           319,          25
{8,503,490 + 1,063,961 + 4,098}, 30    , 960   ,  8         ,     272        , 8503490 ,  11,           386,           28
\end{filecontents*}
\pgfplotstabletypeset[
      col sep=comma,
      string type,
      header=false,
      columns={0,6,8},
      columns/0/.style={column name={DoF $\bigl(|V_h|+|Q_h|+|\Lambda_h|\bigr)$}, column type=l||},
      columns/6/.style={column name=Outer, column type=c},
      columns/8/.style={column name=Inner, column type=c},
      every head row/.style={
          before row=\toprule
          \multicolumn{3}{c}{\textbf{Iteration counts for $\Gamma^4$ (inexact solvers)}} \\ \midrule,
          after row=\midrule
        },
      every last row/.style={after row=\bottomrule}
    ]{data/stokes/iterations/data_iterative.csv}
    \subcaption{Circular interface $\Gamma^4$ (2D test).}
    \label{tab:stokes_2D_iterative}
  \end{minipage}

  \vspace{3ex} %

  \begin{minipage}{0.48\linewidth}
    \centering
    \begin{filecontents*}{data/stokes/iterations/data_sphere.csv}
#DoF                               #outer    #inner              #inner(avg)
{9,435+469+24}             , 12    ,   156   ,   9435,               13
{22,131+1,073+192}          , 12    ,   192   ,   22131,             16
{135,951+6,171+294}         , 12    ,   286   ,   135951,            24
{904,143+39,475+3,072}       , 12    ,   290   ,   904143,           24
{6,718,791+286,855+4,614}     , 8     ,   300   ,   6718791,         37
{52,077,519+2,197,115+286,855} , 8     ,   402   ,   52077519,       50
\end{filecontents*}
\pgfplotstabletypeset[
      col sep=comma,
      string type,
      header=false,
      columns={0,1,4},
      columns/0/.style={column name={DoF $\bigl(|V_h|+|Q_h|+|\Lambda_h|\bigr)$}, column type=l||},
      columns/1/.style={column name=Outer, column type=c},
      columns/4/.style={column name=Inner, column type=c},
      every head row/.style={
          before row=\toprule
          \multicolumn{3}{c}{\textbf{Iteration counts for $\Gamma^5$}} \\ \midrule,
          after row=\midrule
        },
      every last row/.style={after row=\bottomrule}
    ]{data/stokes/iterations/data_sphere.csv}
    \subcaption{Spherical interface $\Gamma^5$ (3D test).}
    \label{tab:stokes_3D_sphere}
  \end{minipage}%
  \hspace{0.04\linewidth}%
  \begin{minipage}{0.48\linewidth}
    \centering
    \begin{filecontents*}{data/stokes/iterations/data_torus.csv}
#DoF                          #outer    #inner                  #inner_avg
{8,367+421+48}             , 13    ,   132   ,   8367,              10
{32,829+1,561+192}          , 14    ,   156   ,   32829,            11
{149,319+6,767+768}         , 14    ,   243   ,   149319,           17
{984,387+42,961+3,072}       , 14    ,   312   ,   984387,          19
{7,132,467+304,897+12,288}    , 13    ,   445   ,   7132467,        25
{53,641,515+2,265,401+49,152}  , 13    ,   726   ,   53641515,      28

\end{filecontents*}
\pgfplotstabletypeset[
      col sep=comma,
      string type,
      header=false,
      columns={0,1,4},
      columns/0/.style={column name={DoF $\bigl(|V_h|+|Q_h|+|\Lambda_h|\bigr)$}, column type=l||},
      columns/1/.style={column name=Outer, column type=c},
      columns/4/.style={column name=Inner, column type=c},
      every head row/.style={
          before row=\toprule
          \multicolumn{3}{c}{\textbf{Iteration counts for $\Gamma^6$}} \\ \midrule,
          after row=\midrule
        },
      every last row/.style={after row=\bottomrule}
    ]{data/stokes/iterations/data_torus.csv}
    \subcaption{Toroidal interface $\Gamma^6$ (3D test).}
    \label{tab:stokes_3D_torus}
  \end{minipage}

  \caption{Outer FGMRES(30) iteration counts with an absolute tolerance $\mathrm{TOL}\!=\!10^{-8}$, along with the
  corresponding { number of average} inner iterations for the Stokes problem using AL preconditioner across different geometries and solvers.}
  \label{tab:iteration_counts}
\end{table}

\section{Conclusions}\label{sec:conclusions}
We have presented augmented Lagrangian-based preconditioners to accelerate the convergence of iterative solvers
for linear systems arising from the finite element method applied to the {fictitious domain approach}. We have considered two model problems with an internal interface, namely the Poisson problem and the Stokes problem, which yield two-by-two and three-by-three block systems, respectively. A spectral analysis of
the preconditioner for the Stokes problem has been performed in both the exact and the inexact case. We tested practical and cheap variants of the proposed preconditioner that do not require the usage of sparse direct solvers, showing its effectiveness
through several numerical tests in two and three dimensions and with different immersed geometries. Additionally, the memory-distributed implementation was validated in a three-dimensional context. Future work will focus on
extending such techniques to more complex systems arising from {fictitious domain methodologies}, such as elliptic interface problems, which are currently under investigation, { and the development of tailored solution
        strategies for the inversion of the augmented block.}

\section*{Acknowledgments}
 {\color{black}MB acknowledges partial support from grant MUR PRIN 2022 No. 20227PCCKZ (``Low Rank Structures and Numerical Methods in Matrix and Tensor Computations and their Applications``).} LH and MF acknowledge partial support from grant MUR PRIN 2022 No. 2022WKWZA8 (“Immersed methods for multiscale and multiphysics problems (IMMEDIATE)”), and the support of the European Research Council (ERC) under the European Union's Horizon 2020 research and innovation programme (call HORIZON-EUROHPC-JU-2023-COE-03, grant agreement No. 101172493 ``dealii-X'').
The authors are members of Gruppo Nazionale per il Calcolo Scientifico (GNCS) of Istituto Nazionale di Alta Matematica (INdAM).

\appendix
\section{Appendix: Spectral equivalence of $h$-scaled mass matrices}\label{sec:app:spectral_h_mass}

\renewcommand{\theequation}{A.\arabic{equation}}
\setcounter{equation}{0}
We prove the spectral equivalence between the matrices $(h_\Gamma \mathsf{M_\lambda)^{-1}}$ and $\mathsf{M_\lambda^{-2}}$ {\color{black} when $d=2$. In this case, the mass matrix $\mathsf{M_{\lambda}}$ is defined on the one-dimensional (closed) curve $\Gamma_h$}. Recall that two families of SPD matrices $\{\mathsf{A}_l\}$ and $\{\mathsf{B}_l\}$ (parametrized by their dimension $l$) are said to be \textit{spectrally equivalent} if there exist $l$-independent constants $\alpha$ and $\beta$ with
$$0<\alpha \le \frac{ \mathsf{w^T} \mathsf{A}_l \> \mathsf w}{\mathsf{w^T} \mathsf{B}_l \> \mathsf w} \le \beta, \quad \quad \forall \mathsf w \neq 0.$$
We drop the subscript $l$. To show that
\begin{equation}\label{eq:speceq}
    0<\alpha \le \frac{ \mathsf{\upmu^T} \mathsf{M_\lambda^{-2}} \mathsf{\upmu}}{ \mathsf{\upmu^T}(h_\Gamma \mathsf{M_\lambda})^{-1} \mathsf{\upmu}} \le \beta, \quad \quad \forall \mathsf{\upmu} \neq 0,
\end{equation}
we first multiply and then divide the previous relation by $\mathsf{\upmu^T}\mathsf{\upmu}$, so that it can be rewritten in terms of the Rayleigh quotients associated {with} the matrices $(h_\Gamma \mathsf{M_\lambda})^{-1}$ and $\mathsf{M_\lambda^{-2}} $. Since for a Hermitian matrix $\mathsf{M}$, its Rayleigh quotient lies in the interval $[\lambda_{\min}(\mathsf{M}), \lambda_{\max}(\mathsf{M})]$, and the eigenvalues of the inverse of a matrix are the reciprocals of the eigenvalues of the original matrix, we get
\begin{equation}\label{eq:eig}
    \frac{\lambda_{\min} (h_\Gamma \mathsf{M_\lambda})}{\lambda_{\max} ( \mathsf{M_\lambda^{2}})}\le \frac{\mathsf{\upmu^T}\mathsf{M_\lambda^{-2}} \mathsf{\upmu}}{\mathsf{\upmu^T}(h_\Gamma \mathsf{M_\lambda})^{-1} \mathsf{\upmu}} \le \frac{\lambda_{\max} (h_\Gamma \mathsf{M_\lambda})}{\lambda_{\min} ( \mathsf{M_\lambda}^{2})}.
\end{equation}
Moreover, for a quasi-uniform subdivision in $\mathbb{R}^1$ of shape-regular elements\footnote{\color{black}Since $d=2$, elements of $\Gamma_h$ are one-dimensional intervals, for which the shape-regularity condition is automatically fulfilled.}, it holds
\begin{equation}\label{eq:massbounds}
    c h_\Gamma \le \frac{\mathsf{\upmu^T} \mathsf{M_\lambda} \mathsf{\upmu} }{\mathsf{\upmu^T} \mathsf{\upmu}} \le C h_\Gamma \quad \quad \forall \mathsf{\upmu} \in \mathbb{R}^l.
\end{equation}
Noting that $\lambda (h_\Gamma \mathsf{M_\lambda}) = h_\Gamma \lambda ( \mathsf{M_\lambda})  $ and $\lambda (\mathsf{M_\lambda}^2) =(\lambda ( \mathsf{M_\lambda}))^2$, we use the above bounds for the eigenvalues in~\eqref{eq:eig}, which yields
\begin{equation}
    0 < \frac{c}{C^2} \le \frac{\mathsf{\upmu^T}\mathsf{M_\lambda^{-2}} \mathsf{\upmu}}{ \mathsf{\upmu^T}(h_\Gamma \mathsf{M_\lambda})^{-1} \mathsf{\upmu}} \le \frac{C}{c^2},
\end{equation}
{\color{black}proving} the spectral equivalence between the two matrices.

    {\section{Appendix: Inverse estimate}\label{sec:app:inverse_estimate}

        \noindent\textbf{Lemma.} \textit{Let $\Lambda_h$ be the space for the discrete Lagrange multiplier, defined in~\eqref{eqn:Lambda_h_S}. Then, there exists a constant $C>0$ independent of the mesh size $h_\Gamma$ such that the following estimate holds:}
        \begin{equation}
            ||\mu_h||_{-\frac{1}{2},\Gamma} \geq C h_{\Gamma}^{\frac{1}{2}} ||\mu_h||_{0,\Gamma} \qquad \forall \mu_h \in \Lambda_h.
        \end{equation}

        \begin{proof}
            In the case of $H^1$-conforming Lagrangian elements, by definition of dual norm, we have
            $$||\mu_h||_{-\frac{1}{2},\Gamma} \coloneqq \sup_{0 \ne z \in H^{\frac{1}{2}}(\Gamma)} \frac{\langle \mu_h, z \rangle_{\Gamma}}{||z||_{\frac{1}{2},\Gamma}} \geq \sup_{0 \ne z_h \in \Lambda_h} \frac{\langle \mu_h, z_h\rangle_{\Gamma}}{||z_h||_{\frac{1}{2},\Gamma}},$$
            where the second inequality follows from the fact that $\Lambda_h \subset H^{\frac{1}{2}}(\Gamma)$. Therefore, upon choosing $z_h = \mu_h$, we obtain
            $$||\mu_h||_{-\frac{1}{2},\Gamma} \geq \frac{||\mu_h||_{0,\Gamma}^2}{||\mu_h||_{\frac{1}{2},\Gamma}}.$$Finally, using the following inequality $$||\mu_h||_{\frac{1}{2},\Gamma} \leq c h_{\Gamma}^{-\frac{1}{2}} ||\mu_h||_{0,\Gamma},$$which extends classical inverse
            inequalities (cfr.~\cite[Sect. 1.7]{Ern_Guermond-FEM-2004}) from integer to real indices, gives the desired result. When $\Lambda_h$ is the discontinuous space of piecewise constant functions, the result has been proven by Glowinski and Girault in~\cite{Glowinski} (Theorem 9).
        \end{proof}}

\section{AMG Parameters}\label{sec:app:AMG_params}
The AMG parameter settings used in the numerical experiments are reported in Table~\ref{tab:amg_params}.
\renewcommand{\thetable}{C.\arabic{table}}
\setcounter{table}{0}
\begin{table}[h]
    \centering
    \renewcommand{\arraystretch}{1.2}
    \begin{tabular}{|l|l|}
        \hline
        \textbf{Parameter}    & \textbf{Value} \\
        \hline
        Smoother              & Chebyshev      \\
        Coarse solver         & Amesos-KLU     \\
        Smoother sweeps       & 2              \\
        V-cycle applications  & 1              \\
        Aggregation threshold & \(10^{-4}\)    \\
        Max size coarse level & 2000           \\
        \hline
    \end{tabular}
    \caption{Parameters for ML~\cite{Trilinos} (Trilinos 14.4.0).}
    \label{tab:amg_params}
\end{table}

\bibliography{refs}
\end{document}